\documentclass[11pt]{article}
\usepackage{epsfig,mst-stylefile,amssymb,amsmath,harvard}

\newcommand{\bbZ}{{\Bbb Z}}
\newcommand{\bbR}{{\Bbb R}}
\newcommand{\bbN}{{\Bbb N}}
\newcommand{\bbC}{{\Bbb C}}

\newcommand{\norm}[1]{\left\|#1\right\|}
\newcommand{\Var}{\textnormal{Var}}
\newcommand{\Cov}{\textnormal{Cov}}
\newcommand{\vecoper}{\textnormal{vec}}

\renewcommand{\cite}{\citeyear}

\begin{document}

\title{Wavelet estimation for operator fractional Brownian motion
\thanks{The first author was partially supported by the French ANR AMATIS 2011 grant $\# $ANR-11-BS01-0011.
The second author was partially supported by the Louisiana Board of Regents award LEQSF(2008-11)-RD-A-23 and by the prime award no.\ W911NF-14-1-0475 from the Biomathematics subdivision of the Army Research Office. Support from ENS de Lyon for the author's long-term visit to the school is gratefully acknowledged. The second author is also grateful to the Statistics and Probability Department at MSU and the Laboratoire de Physique at ENS de Lyon for their great hospitality and rich research environments. The authors would like to thank Mark M. Meerschaert for his comments on this work.}
}

\author{ Patrice Abry \\ Physics Lab \\ CNRS and \'{E}cole Normale Sup\'{e}rieure de Lyon  \and   Gustavo Didier \\ Mathematics Department\\ Tulane University}

\bibliographystyle{agsm}

\maketitle

\begin{abstract}
Operator fractional Brownian motion (OFBM) is the natural vector-valued extension of the univariate fractional Brownian motion. Instead of a scalar parameter, the law of an OFBM scales according to a Hurst matrix that affects every component of the process. In this paper, we develop the wavelet analysis of OFBM, as well as a new estimator for the Hurst matrix of bivariate OFBM. For OFBM, the univariate-inspired approach of analyzing the entry-wise behavior of the wavelet spectrum as a function of the (wavelet) scales is fraught with difficulties stemming from mixtures of power laws. The proposed approach consists of considering the evolution along scales of the eigenstructure of the wavelet spectrum. This is shown to yield consistent and asymptotically normal estimators of the Hurst eigenvalues, and also of the coordinate system itself under assumptions. A simulation study is included to demonstrate the good performance of the estimators under finite sample sizes.
\end{abstract}


\section{Introduction}



An $\bbR^n$-valued stochastic process $\{X(t)\}_{t \in \bbR}$ is said to be operator self-similar (o.s.s.) when its law scales according to a matrix (Hurst) exponent $H$, i.e.,
\begin{equation}\label{e:o.s.s.}
\{ X(ct) \}_{t \in \bbR} \stackrel{\mathcal{L}}=  \{ c^{H}X(t) \}_{t \in \bbR}, \quad c > 0,
\end{equation}
where $c^H = \sum^{\infty}_{k=0}\log^k(c) H^k/k!$ and $\stackrel{\mathcal{L}}=$ denotes the equality of finite-dimensional distributions. No specific assumption on the eigenstructure of $H$ is imposed, e.g., canonical vectors are not necessarily eigenvectors. The notion of operator self-similarity underpins the natural multivariate generalization of the univariate fractional Brownian motion (FBM): an operator fractional Brownian motion (OFBM) $B_{H} = \{B_H(t)\}_{t \in \bbR}$ is a proper Gaussian, o.s.s., stationary increment stochastic process. In this paper, we propose using the wavelet eigenstructure of OFBM to estimate $H$. The main motivation behind this methodology is to avoid the difficulties stemming from the extrapolation of univariate techniques to a multivariate context, where operator scaling laws such as \eqref{e:o.s.s.} may arise.\\  


Inferential theory for univariate self-similar processes now comprises a voluminous and well-established literature. A non-exhaustive list includes Fox and Taqqu \cite{fox:taqqu:1986} and Robinson \cite{robinson:1995-gaussian,robinson:1995-logperiodogram_regression} on Fourier domain methods, and Wornell and Oppenheim \cite{wornell:oppenheim:1992}, Flandrin \cite{flandrin:1992}, and Veitch and Abry \cite{veitch:abry:1999} on wavelet domain methods, among many others.
The multivariate framework evokes several applications where matrix-based scaling laws are expected to appear, such as in long range dependent time series (Marinucci and Robinson \cite{marinucci:robinson:2000}, Davidson and de Jong \cite{davidson:dejong:2000}, Chung \cite{chung:2002}, Dolado and Marmol \cite{dolado:marmol:2004}, Davidson and Hashimzade \cite{davidson:hashimzade:2007}, Kechagias and Pipiras \cite{kechagias:pipiras:2015}) and queueing systems (Konstantopoulos and Lin \cite{konstantopoulos:lin:1996}, Majewski \cite{majewski:2003,majewski:2005}, Delgado \cite{delgado:2007}). Like FBM in the univariate setting, OFBM is a natural starting point in the construction of estimators for operator self-similar processes due to its tight connection to stationary fractional processes and its being Gaussian (on the general theory of o.s.s.\ processes, see Laha and Rohatgi \cite{laha:rohatgi:1981}, Hudson and Mason \cite{hudson:mason:1982}, Maejima and Mason \cite{maejima:mason:1994}, Cohen et al \cite{cohen:meerschaert:rosinski:2010}).

A characterization of the covariance structure of OFBM can be derived from stochastic integral representations. Under a mild condition on the eigenvalues of the exponent $H$ (see \eqref{e:eigen-assumption}), Didier and Pipiras \cite{didier:pipiras:2011} showed that any OFBM $B_H$ admits a harmonizable representation
\begin{equation} \label{e:OFBM_harmonizable}
\{B_{H}(t)\}_{t \in \bbR} \stackrel{{\mathcal L}}=
\Big\{\int_{\bbR} \frac{e^{itx} - 1}{ix}
(x^{-D}_{+}A +
x^{-D}_{-}\overline{A}) \widetilde{B}(dx)\Big\}_{t
\in \bbR}
\end{equation}
for some complex-valued matrix $A$. In \eqref{e:OFBM_harmonizable}, $x_{\pm} = \max\{\pm x,0\}$,
\begin{equation}\label{e:D}
D = H - \frac{1}{2}\hspace{1mm}I,
\end{equation}
and $\widetilde{B}(dx)$ is a complex-valued random measure such that $\widetilde{B}(-dx)= \overline{\widetilde{B}(dx)}$, $E\widetilde{B}(dx)\widetilde{B}(dx)^{*} = dx$, where $*$ represents Hermitian transposition. Expression \eqref{e:OFBM_harmonizable} shows that the law of an OFBM can be fully described based on the scaling matrix $H$ and the spectral parameter $A$ (see Remark \ref{r:parametrization} on the parametrization).

Let $H = P J_H P^{-1}$, $P \in GL(n,\bbC)$, be the Jordan form of the Hurst parameter in \eqref{e:OFBM_harmonizable} (see Section \ref{s:preliminaries} for matrix notation). If $H$ is diagonal, then we can assume that $P$ takes the form of a scalar matrix $P = pI$, where $p \in \bbR$ and $I$ is the $n \times n$ identity matrix, and that $J_H = \textnormal{diag}(h_1,\hdots,h_n)$. In this case, \eqref{e:o.s.s.} breaks down into simultaneous entry-wise expressions
\begin{equation}\label{e:o.s.s._entrywise}
\{ X(ct) \}_{t \in \bbR} \stackrel{\mathcal{L}}=  \{ (c^{h_1}X_1(t), \hdots,c^{h_n}X_n(t))^* \}_{t \in \bbR}, \quad c > 0.
\end{equation}
Relation \eqref{e:o.s.s._entrywise} is henceforth called \textit{entry-wise scaling}. In particular, under \eqref{e:o.s.s._entrywise} an OFBM is a vector of correlated FBM entries (Amblard et al.\ \cite{amblard:coeurjolly:lavancier:philippe:2012}, Coeurjolly et al.\ \cite{coeurjolly:amblard:achard:2013}). Several estimators have been developed by building upon the univariate, entry-wise scaling laws, e.g., the Fourier-based multivariate local Whittle (e.g., Shimotsu \cite{shimotsu:2007}, Nielsen \cite{nielsen:2011}) and the multivariate wavelet regression (Wendt et al.\ \cite{WENDT:2009:C}, Amblard and Coeurjolly \cite{amblard:coeurjolly:2011}, Achard and Gannaz \cite{achard:gannaz:2014}).
However, if $H$ is non-diagonal, i.e., if $P$ is not a scalar matrix, then the relation \eqref{e:o.s.s.} mixes together the several entries of $X$. The estimation problem under a non-scalar $P$ turns out to be rather intricate and calls for the construction of methods that are multivariate from their inception.

Although the emergence of o.s.s.\ processes in applications is rightly expected -- e.g., as functional weak limits of multivariate time series --, there is no specific reason to believe a priori that scaling laws occur predominantly entry-wise and \textit{exactly} along the canonical axes. Indeed, this is palpably not true in several applications such as fractional blind-source separation (see Didier et al.\ \cite{Helgason:Didier:Abry:2015}) and fractional cointegration (see Robinson \cite{robinson:2008} for a bivariate local Whittle estimator). The framework of multivariate mixed fractional time series subsumes both cases. Let $W = \{W_t\}_{t \in \bbZ}$ be an unobserved signal. To fix ideas suppose that the signal is an operator fractional Gaussian noise (OFGN), namely, it has the form $W_t = B_{H_W}(t) - B_{H_W}(t-1)$ with spectral parameter $A_W$ (see \eqref{e:OFBM_harmonizable}). Further assume that $H_W = \textnormal{diag}(h_1,\hdots,h_n)$. The observed signal has the form $Y = \{Y_t\}_{t \in \bbZ} = \{P W_t\}_{t \in \bbZ}$, for a non-scalar, mixing matrix $P \in GL(n,\bbR)$. Then, the mixed process $Y$ is another OFGN with parameters $H_Y = P \textnormal{diag}(h_1,\hdots,h_n)P^{-1}$ and $A_Y = PA_{W}$. In blind source separation, the entries of $W$ are uncorrelated, whereas in cointegration they are typically correlated.

From a mathematical standpoint, the mixing of scaling laws can be illustrated by means of the expression for the spectral density $f_X$ of an OFBM with Hurst parameter $H = P \textnormal{diag}(h_1,h_2)P^{-1}$, $0 < h_1 < h_2 < 1$, $P \in GL(2,\bbR)$ (see condition \eqref{e:bivariate_OFBM_h1_h2_real}). For $M := P^{-1}AA^* (P^*)^{-1} \in M(n)$ (c.f.\ condition \eqref{e:time_revers}) and $x > 0$, the spectral density takes the form $\Big( f_{X}(x)_{ij} \Big) = x^{-D}AA^* x^{-D^*} $, where
$$
f_{X}(x)_{11} = p^2_{11} m_{11} x^{-2 d_1} + 2p_{11}p_{12} m_{12} x^{-(d_1+d_2)} +
p^2_{12} m_{22} x^{- 2 d_2},
$$
$$
f_{X}(x)_{21} = p_{11}p_{21}m_{11}x^{-2 d_1}+(p_{11}p_{22}+ p_{12}p_{21})m_{12}x^{-(d_1 + d_2)}+
p_{12}p_{22}m_{22}x^{- 2 d_2},
$$
$$
f_{X}(x)_{22} = p^2_{21} m_{11}x^{- 2 d_1} + 2 p_{21}p_{22} m_{12} x^{- (d_1 + d_2)} +
p^2_{22} m_{22}x^{- 2 d_2},
$$
and $d_1 = h_1 - 1/2$, $d_2 = h_2 - 1/2$ (see \eqref{e:S=a(nu)HEW(j)a(nu)H*=(a_b_b_c)} for the analogous expression in the wavelet domain). The univariate-inspired approach of setting up a Fourier-domain log-regression -- e.g., Whittle-type estimators -- has to cope with the double-sided challenge of mixed power laws. On one hand, under mild assumptions on the amplitude coefficients, the dominant power law $x^{-2 d_2}$ always prevails around the origin of the spectrum. On the other hand, and paradoxically, even if the estimation of $d_2$ is the target, the magnitude of the amplitude coefficients themselves can arbitrarily bias the estimate over finite samples by masking the dominant power law.

In this work, we build upon the wavelet analysis of ($n$-dimensional) OFBM to propose a novel wavelet-based estimation method for bivariate, and potentially multivariate, OFBM. The method yields the Hurst eigenvalues of $H$ and, under mild assumptions, also its eigenvectors when $P \in O(2)$. Its essential ingredient, and the main theme of this paper, is a change of perspective: instead of considering the entry-wise behavior of the wavelet spectrum as a function of wavelet scales, it draws upon the evolution along scales of the eigenstructure of the wavelet spectrum. This way, it avoids much of the difficulty associated with inference in the presence of mixed power laws, as we now explain.

For a wavelet function $\psi \in L^2(\bbR)$ with a number $N_{\psi}$ of vanishing moments (see \eqref{e:N_psi}), the (normalized) vector wavelet transforms of OFBM is naturally defined as
\begin{equation}\label{wavelet_transform}
\bbR^n \ni D(2^j,k) = 2^{-j/2} \int_{\bbR} 2^{-j/2} \psi(2^{-j}t-k) B_{H}(t) dt, \quad j \in \bbN \cup \{0\}, \quad k \in \bbZ,
\end{equation}
provided the integral in \eqref{wavelet_transform} exists in an appropriate sense. The wavelet-domain process $\{D(2^j,k)\}_{k \in \bbZ}$ is stationary in $k$ and o.s.s.\ in $j$ (Proposition \ref{p:wavelet_coefs_properties}). Moreover, whereas the original stochastic process $B_H(t)$ displayed fractional memory, the covariance between (multivariate) wavelet coefficients decays as a function of $| 2^j k - 2^{j'}k'|$ according to an inverse fractional power controlled by $N_{\psi}$ (Proposition \ref{p:decay_Cov_wavelet_coefs}). The wavelet spectrum (variance) at scale $j$ is the positive definite matrix
$$
ED(2^j,k)D(2^j,k)^* = ED(2^j,0)D(2^j,0)^* =: E W(2^j),
$$
and its natural estimator, the sample wavelet variance, is the matrix statistic
\begin{equation}\label{e:W(j)}
W(2^j) = \frac{1}{K_j} \sum^{K_j}_{k=1}D(2^j,k)D(2^j,k)^*, \quad K_j = \frac{\nu}{2^j}, \quad j = j_1,\hdots,j_m,
\end{equation}
for a dyadic total of $\nu$ (wavelet) data points. Within the bivariate framework, the univariate-like entry-wise scaling approach would consist of exploiting the behavior of each component $W(2^j)_{i_1, i_2} $, $i_1, i_2 = 1,2$, of the sample wavelet transform $ W(2^j) $ as a function of the scales $2^j$. Apart from an amplitude effect, the entries are then controlled by the dominant Hurst eigenvalue $h_2$ (see expression \eqref{e:Shat=a(nu)H_EW(j)_a(nu)H*=(a_b_b_c)}). Figure~\ref{fig:figa}, top panels, illustrates the fact that this precludes the estimation of $h_1$.

The proposed estimators of the Hurst eigenvalues $h_1$ and $h_2$ are
\begin{equation}\label{e:def_estimator}
\widehat{h}_1(2^j) = \frac{\log \lambda_{1}(2^j) }{2 \log (2^j)} , \quad  \widehat{h}_2(2^j) = \frac{\log \lambda_{2}(2^j) }{2 \log (2^j)},
\end{equation}
where $\lambda_{1}(2^j) \leq \lambda_{2}(2^j)$ are the eigenvalues of the positive definite symmetric matrix $W(2^j)$ (see Definition \ref{def:estimator} for the precise assumptions). However, as usual with operator self-similarity, the finite sample expressions for $\lambda_1(2^j)$ and $\lambda_2(2^j)$ themselves involve a mixture of distinct power laws $2^{j 2h_1}$, $2^{j (h_1+h_2)}$, $2^{j 2h_2}$ ($h_1 < h_2$). For this reason, one must take the limit at coarse scales, namely, the scale itself must go to infinity. It is a remarkable fact that the power law $2^{j2 h_1}$ ends up prevailing in the expression for $\lambda_1(2^j)$ (see Figure \ref{fig:figa}, bottom panels, and the striking contrast with the top panels; see Remark \ref{r:motivation} for a mathematically motivated, intuitive discussion). The convergence of \eqref{e:def_estimator} in turn allows for the convergence of associated sequences of eigenvectors when $P$ is orthogonal. Moreover, simulation studies show that the estimation procedure is accurate and computationally fast. The asymptotics are mathematically developed in two stages. In the first, the wavelet scales (octaves) are held fixed and the asymptotic distribution of the sample wavelet transform is obtained (Proposition \ref{p:4th_moments_wavecoef} and Theorem \ref{t:asymptotic_normality_wavecoef_fixed_scales}). In the second, one takes the limit with respect to the scales themselves. However, the latter must go to infinity slower than the sample size, a feature that our estimators share with Fourier or wavelet-based semiparametric estimators in general (e.g., Robinson \cite{robinson:1995-gaussian}, Moulines et al.\ \cite{moulines:roueff:taqqu:2007:JTSA,moulines:roueff:taqqu:2007:Fractals,moulines:roueff:taqqu:2008}).

Our results are related to the literature on the estimation of operator stable laws via eigenvalues and eigenvectors of sample quadratic forms (see Meerschaert and Scheffler \cite{meerschaert:scheffler:1999,meerschaert:scheffler:2003}). In this context, one encounters the same problem with the prevalence of some dominant power law (i.e., the tail exponent) in most directions. In Becker-Kern and Pap \cite{becker-kern:pap:2008}, a similar philosophy is applied in the time domain to produce one of the very few available estimators for authentic, mixed scaling o.s.s.\ processes of dimension up to 4. However, the asymptotics provided are restricted to consistency. In our work, the wavelet transform is the main tool for ensuring the consistency and asymptotic normality of the proposed estimators. 

The paper is organized as follows. Section \ref{s:preliminaries} contains the notation, assumptions and basic concepts. Section \ref{s:wavelet_analysis} is dedicated to the wavelet analysis of $n$-dimensional OFBM, as well as the asymptotics of the wavelet transform for fixed scales (most of the proofs can be found in Section \ref{s:asympt_normality_fixed_scales}). In Section \ref{s:estimation}, the estimation method for the Hurst exponent of bivariate OFBM is laid out in full detail and its asymptotics are established at coarse scales. Section \ref{s:simulation_studies} displays finite sample computational studies, including one of the performance of the estimators under blind source separation and cointegrated instances, with the purpose of illustrating the robustness of the estimators' performance with respect to different parametric scenarios. The research contained in this paper leads to a number of interesting open questions, which are mentioned in Section \ref{s:open}. Among these is the extension of the consistency and asymptotic normality to any dimension $n \geq 3$, which will generally require dispensing with explicit formulas for eigenvalues and eigenvectors. The appendix contains several auxiliary mathematical results. In addition, in Section \ref{s:discretized_wavelet}, the performance of the estimators is established under the assumption that only discrete observations are available, instead of a full sample path as in \eqref{wavelet_transform}.

 \begin{figure}[h]
\centerline{
\includegraphics[height=40truemm,keepaspectratio]{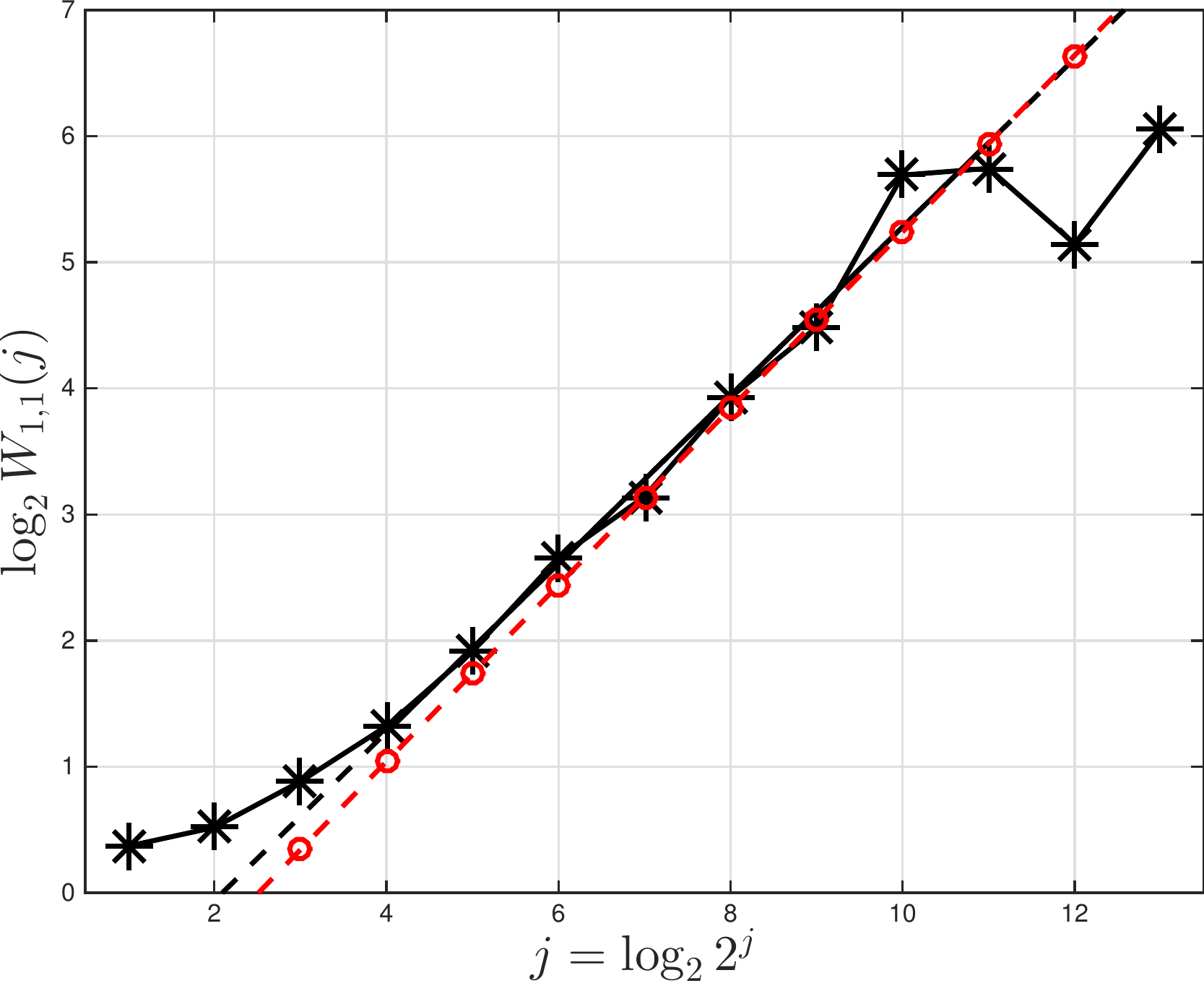}
\includegraphics[height=40truemm,keepaspectratio]{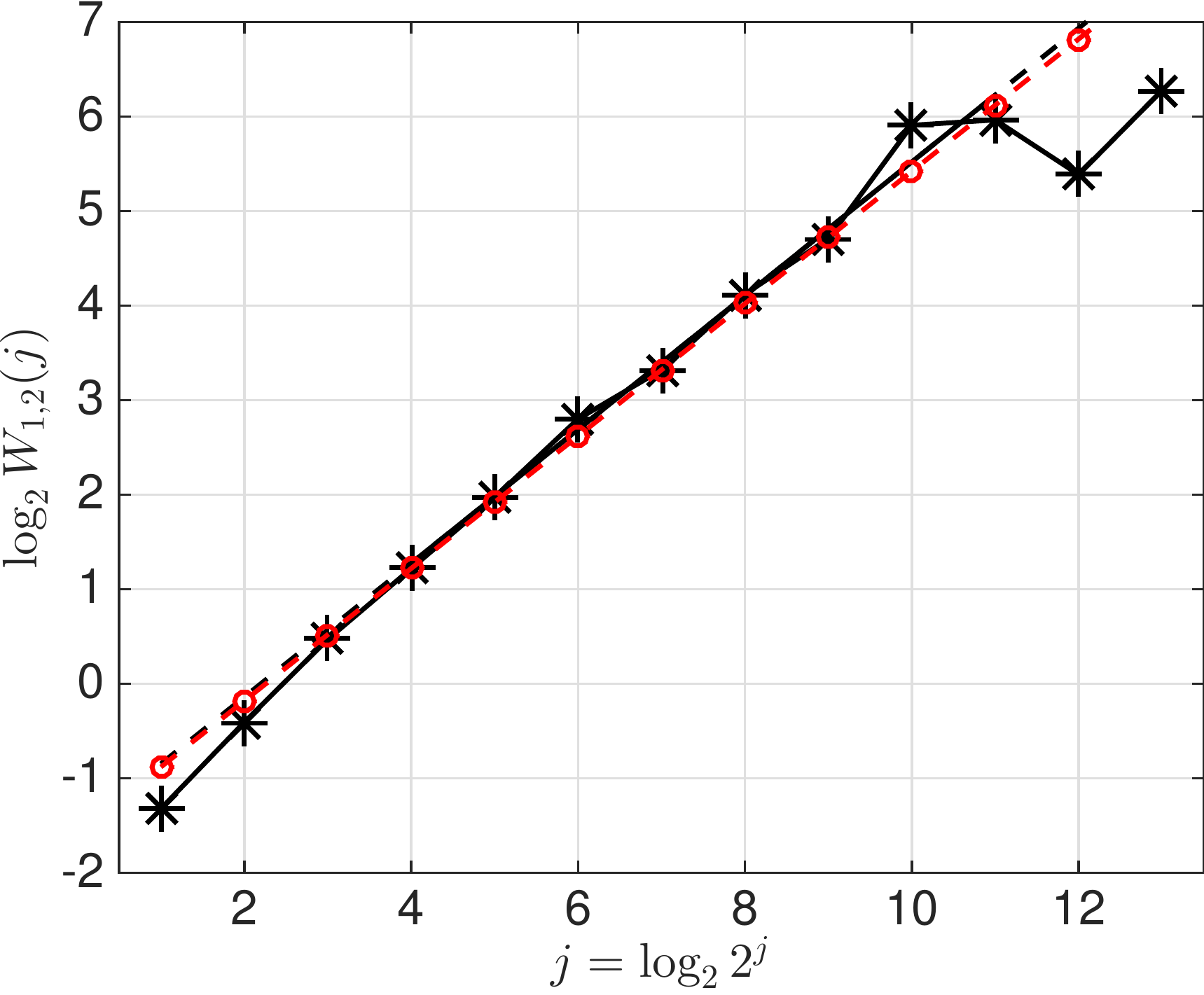}
\includegraphics[height=40truemm,keepaspectratio]{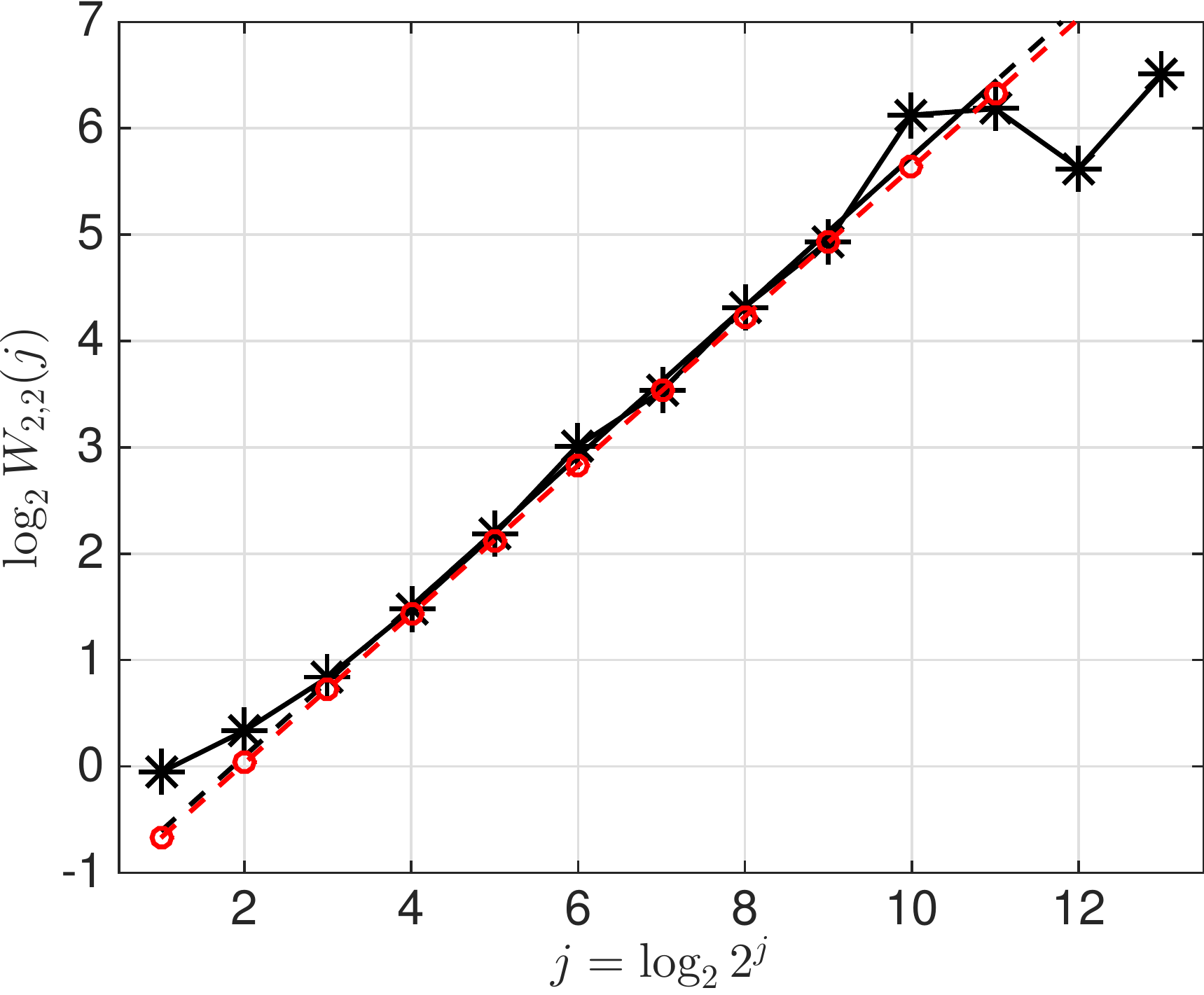}
}
\centerline{
\includegraphics[height=40truemm,keepaspectratio]{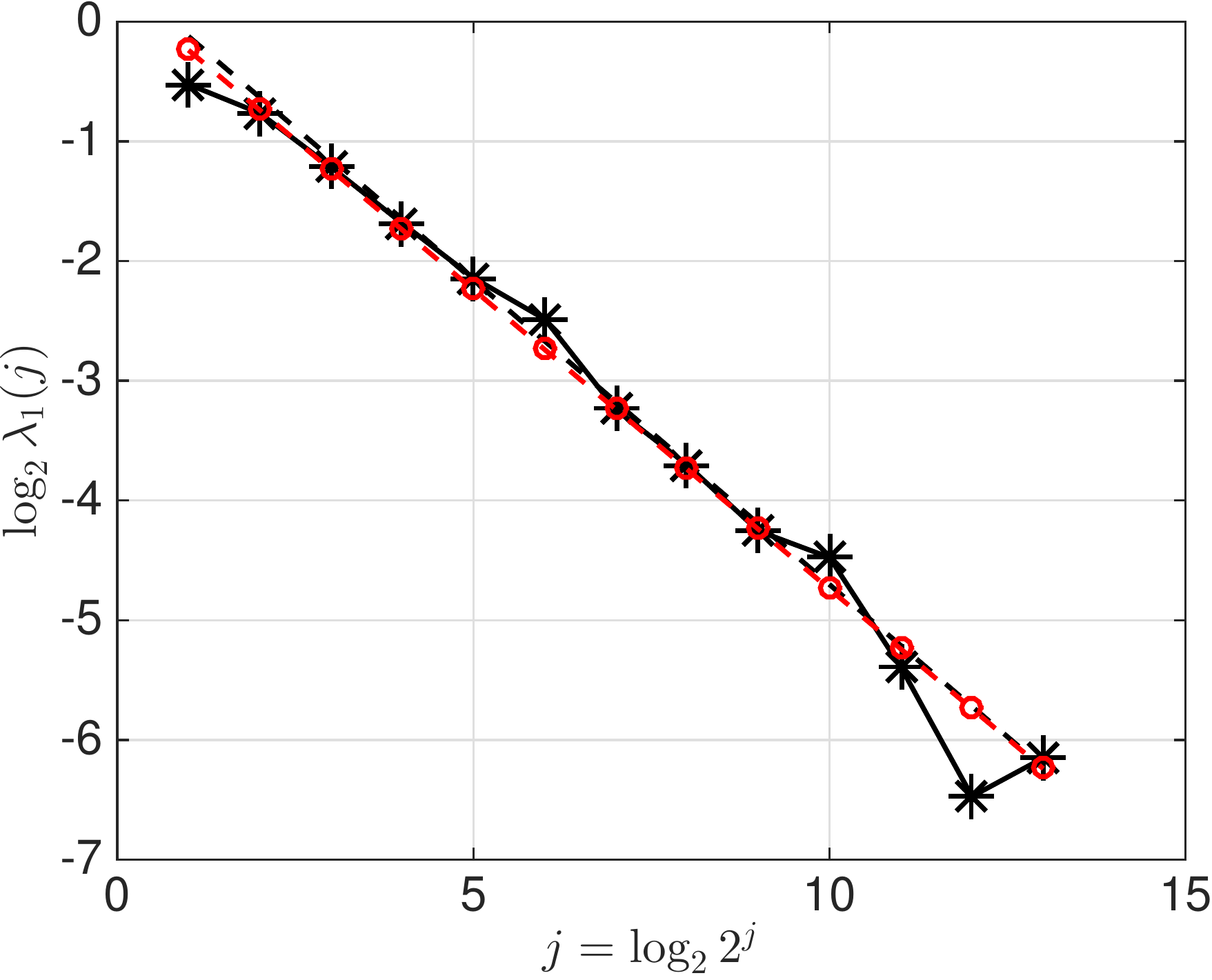}
\includegraphics[height=40truemm,keepaspectratio]{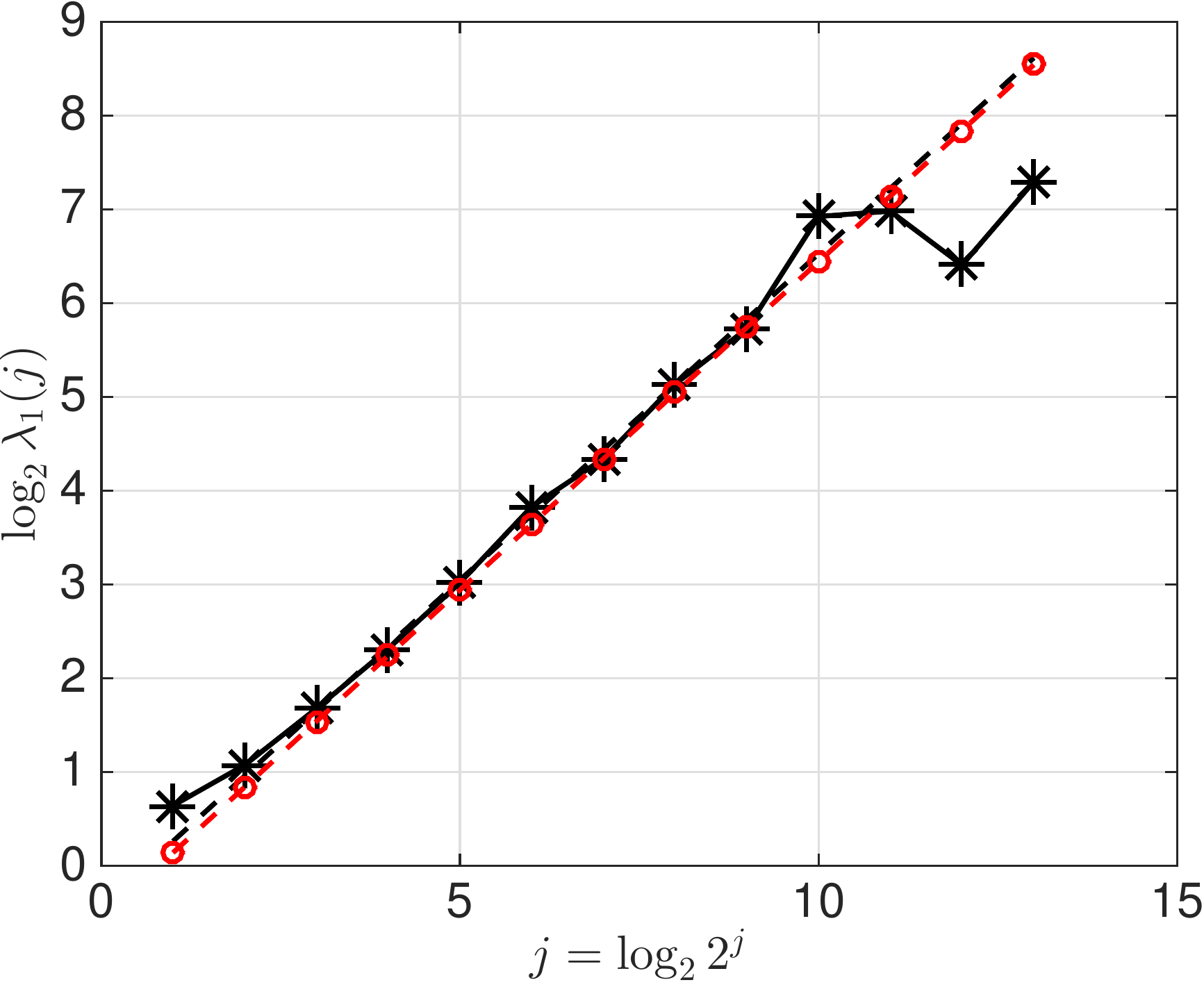}
\includegraphics[height=40truemm,keepaspectratio]{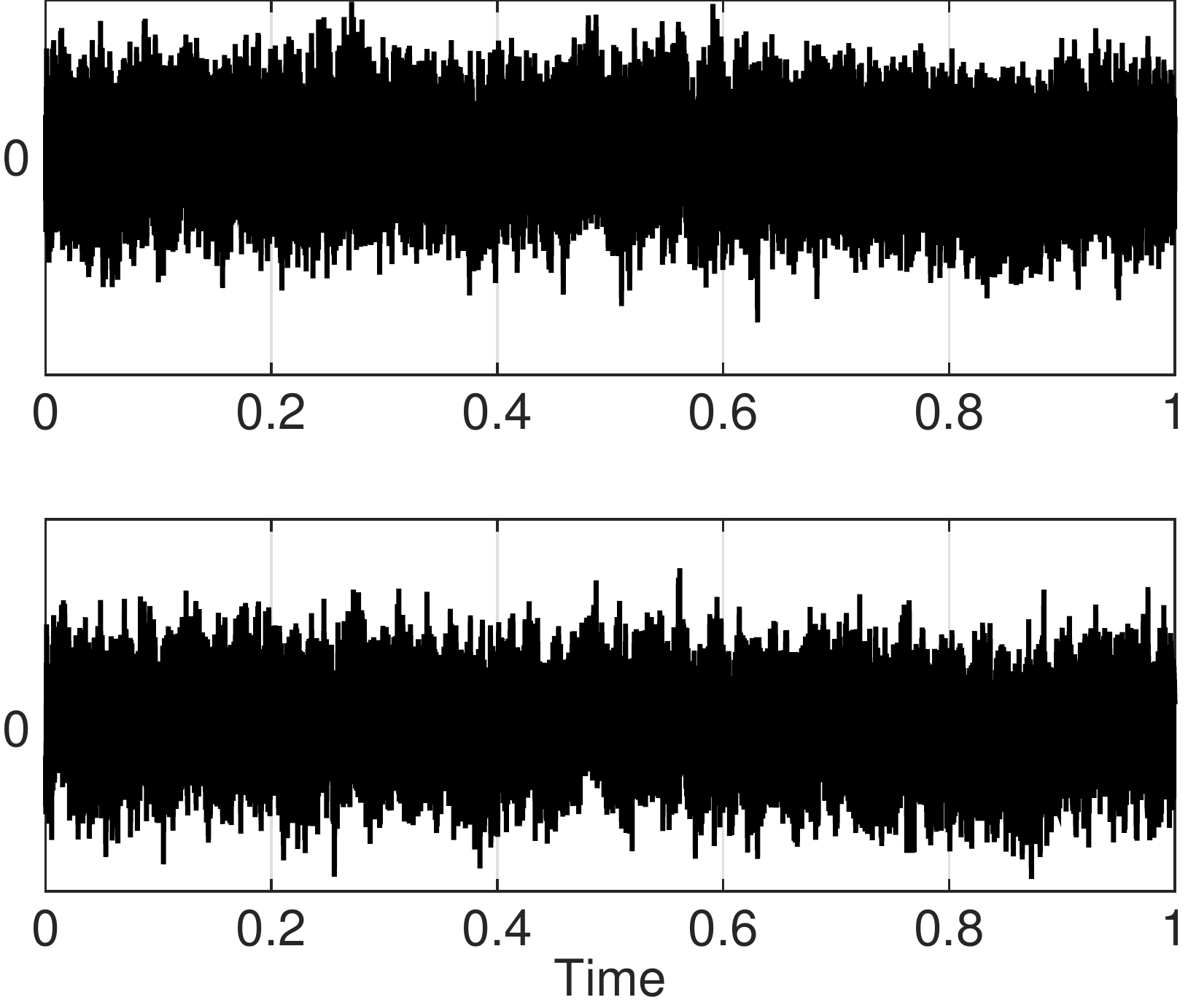}
}
\caption{\label{fig:figa} {\bf Entry-wise vs eigenvalue-based estimation.} From one synthetic realization of OFBM, the top row displays (black solid lines with $\ast$), from left to right, $\log_2 W(2^j)_{1,1} $ vs $ j $, $\log_2 W(2^j)_{1,2} $ vs $ j$ and $\log_2 W(2^j)_{2,2} $ vs. $ j $, on which the asymptotic behavior $j 2 h_2$ is superimposed (red dashed line with `o'). All auto- and cross-components are then driven by the dominant Hurst eigenvalue $h_2$, which precludes the estimation of the eigenvalue $h_1$. The first two bottom row plots display (black solid lines with $\ast$), from left to right, $\log_2 \lambda_1(2^j) $ vs $ j $ and $\log_2 \lambda_2(2^j) $ vs $ j $, with their respective asymptotic trends $j 2 h_1$ and $j 2 h_2$ superimposed (red dashed line with `o'). This demonstrates that both Hurst eigenvalues $h_1$ and $h_2$ can be estimated (see Section \ref{s:entry-wise_vs_eigenvalue} for the simulation details). The bottom right plot displays the increments of the generated OFBM.}
\end{figure}

\section{Notation and assumptions}\label{s:preliminaries}

All through the paper, the dimension of OFBM is denoted by $n \geq 2$.

We shall use throughout the paper the following notation for finite-dimensional operators (matrices). All with respect to the
field $\bbR$, $M(n)$ or $M(n,\bbR)$ is the vector space of all $n \times n$ matrices (endomorphisms), $GL(n)$ or $GL(n,\bbR)$ is the
general linear group (invertible matrices, or automorphisms), $O(n)$ is the (orthogonal) group of matrices $O$ such that $OO^{*} =
I = O^{*}O$ (i.e., the adjoint operator is the inverse), and ${\mathcal S}(n,\bbR)$ is the space of symmetric matrices. We also write $I_n$ to indicate the dimension of the identity matrix $I$. A block-diagonal matrix with main diagonal blocks ${\mathcal P}_1,\hdots,{\mathcal P}_n$ or $m$ times repeated diagonal block ${\mathcal P}$ is represented by
\begin{equation}\label{e:block_diagonal_matrices}
\textnormal{diag}({\mathcal P}_1,\hdots,{\mathcal P}_n), \quad  \textnormal{diag}_m({\mathcal P}),
\end{equation}
respectively. The symbol $\|\cdot\|$ represents a generic matrix or vector norm. For $q \in \bbN$, the $l^q$ entry-wise norm of an $m \times n$ real-valued matrix $M$ is denoted by
\begin{equation}\label{e:|A|p}
\|M\|_{l^q} = \Big(\sum^{m}_{i_1=1}\sum^{n}_{i_2=1} |M_{i_1,i_2}|^q \Big)^{1/q}
\end{equation}
and $\|M\|_{l^\infty} = \sup_{i_1,i_2}|M_{i_1,i_2}|$. A generic matrix $M \in M(n,\bbC)$ has real and imaginary parts $\Re(M)$ and $\Im(M)$, respectively. The functions
\begin{equation}\label{e:projections_def}
\pi_{i}(v), \pi_{i_1,i_2}(M), \quad {\mathbf v} \in \bbR^n, \quad M \in M(n,\bbR),
\end{equation}
denote, respectively, the $i$-th projection (entry) of the vector ${\mathbf v}$ and the $(i_1,i_2)$-th projection (entry)
of the matrix $M$, $i_1,i_2 =1,\hdots,n$. For $S = ( s_{i_1,i_2} )_{i_1,i_2=1,\hdots,n} \in {\mathcal S}(n,\bbR)$,
we define the operator
\begin{equation}\label{e:vec_S}
\textnormal{vec}_{{\mathcal S}}(S) = (s_{11}, s_{12},\hdots, s_{1n}; s_{22},\hdots, s_{2n}; \hdots;s_{n-1,n-1}, s_{n-1,n};s_{n,n})^*.
\end{equation}
In other words, $\textnormal{vec}_{{\mathcal S}}(\cdot)$ vectorizes the upper triangular entries of the symmetric matrix $S$.

When establishing bounds, $C$ stands for a positive constant whose value can change from one line to the next. For a sequence of random vectors $\{X_l, Y_l\}_{l \in \bbN}$, $P(Y_l = 0) = 0$, we write
\begin{equation}\label{e:Xj_approx_Yj}
X_l \stackrel{P}\sim Y_l
\end{equation}
to mean that $X_l / Y_l \stackrel{P}\rightarrow 1$, $l \rightarrow \infty$. Note that this does not imply that $\{X_l, Y_l\}_{l \in \bbN}$ converges in probability. Relations of the type \eqref{e:Xj_approx_Yj} will often appear in the proofs of the results in Section \ref{s:estimation}. We write ${\mathbf X} \stackrel{d}= {\mathbf Y}$ when the random vectors ${\mathbf X}$ and ${\mathbf Y}$ have the same distribution.

All through the paper, we will make the following assumptions on the underlying wavelet basis. For this reason, such assumptions will be omitted in the statements.

\medskip

\noindent {\sc Assumption $(W1)$}: $\psi \in L^1(\bbR)$ is a wavelet function, namely,
\begin{equation}\label{e:N_psi}
\int_{\bbR} \psi^2(t)dt = 1 , \quad \int_{\bbR} t^{q}\psi(t)dt = 0, \quad q = 0,1,\hdots, N_{\psi}-1, \quad N_{\psi} \geq 2.
\end{equation}
\noindent {\sc Assumption $(W2)$}:
\begin{equation}\label{e:supp_psi=compact}
\textnormal{$\textnormal{supp}(\psi)$ is a compact interval}.
\end{equation}
\noindent {\sc Assumption $(W3)$}: there is $\alpha > 1$ such that
\begin{equation}\label{e:psihat_is_slower_than_a_power_function}
\sup_{x \in \bbR} |\widehat{\psi}(x)| (1 + |x|)^{\alpha} < \infty.
\end{equation}

\medskip

Under \eqref{e:N_psi}, \eqref{e:supp_psi=compact} and \eqref{e:psihat_is_slower_than_a_power_function}, $\psi$ is continuous, $\widehat{\psi}(x)$ is everywhere differentiable and its first $N_{\psi}-1$ derivatives are zero at $x = 0$ (see Mallat \cite{mallat:1999}, Theorem 6.1 and the proof of Theorem 7.4).

\begin{example}
If $\psi$ is a Daubechies wavelet with $N_{\psi}$ vanishing moments, $\textnormal{supp}(\psi) = [0,2N_{\psi} -1]$ (see Mallat \cite{mallat:1999}, Proposition 7.4).
\end{example}

Starting from the harmonizable representation \eqref{e:OFBM_harmonizable}, throughout the paper we will make the following assumptions on the OFBM $B_H = \{B_H(t)\}_{t \in \bbR}$.

\medskip

\noindent {\sc Assumption (OFBM1)}: the eigenvalues (characteristic roots) $h_k$ of the matrix exponent $H$ satisfy

\begin{equation}\label{e:eigen-assumption}
    0< \Re(h_k)<1,\quad k=1,\ldots,n.
\end{equation}

\medskip

\noindent {\sc Assumption (OFBM2)}:
\begin{equation}\label{e:full-rank}
\textnormal{$\Re(AA^*)$ has full rank}.
\end{equation}

\medskip

\noindent The condition \eqref{e:eigen-assumption} generalizes the familiar constraint $0 < H < 1$ on the Hurst parameter of a FBM. As shown in Didier and Pipiras \cite{didier:pipiras:2011}, it ensures the existence of the harmonizable representation \eqref{e:OFBM_harmonizable}. Also, \eqref{e:eigen-assumption} implies that the OFBM under consideration has mean zero, which follows by the same reasoning as in Taqqu \cite{taqqu:2003}, p.7, property ($ii$). In turn, recall that a stochastic process is called proper when the distribution of $X(t)$ is full dimensional for $t \neq 0$. The condition \eqref{e:full-rank} is sufficient (though not necessary) for the integral on the right-hand side of \eqref{e:OFBM_harmonizable} to be a proper stochastic process and hence to define an OFBM.


The next two assumptions will appear in some of the results. 

\medskip

\noindent {\sc Assumption (OFBM3)}:
\begin{equation}\label{e:time_revers}
\Im(AA^*) = 0.
\end{equation}

\medskip

\noindent {\sc Assumption (OFBM4)}: $B_H = \{B_H(t)\}_{t \in \bbR}$ is a \textit{bivariate} OFBM with scaling matrix
\begin{equation}\label{e:bivariate_OFBM_h1_h2_real}
H = PJ_H P^{-1} = P \textnormal{diag}(h_1,h_2)P^{-1}, \quad 0< h_1 < h_2 < 1, \quad P \in GL(2,\bbR),
\end{equation}
where the columns of $P$ are unit vectors.

\medskip

\noindent The condition \eqref{e:time_revers} is equivalent to time reversibility, namely, $\{B_{H}(-t)\} \stackrel{{\mathcal L}}= \{B_{H}(t)\}$. In turn, the latter is equivalent to the existence of a closed form expression for the covariance function, i.e.,
\begin{equation}\label{e:time_reversibility}
EB_{H}(s)B_{H}(t)^* = \frac{1}{2} \{|t|^{H} \Sigma |t|^{H^*} + |s|^{H} \Sigma |s|^{H^*} - |t-s|^{H} \Sigma |t-s|^{H^*}\},
\end{equation}
where $\Sigma := EB_{H}(1)B_{H}(1)^*$ (Didier and Pipiras \cite{didier:pipiras:2011}, Proposition 5.2). Time reversibility is used in some of the results in Section \ref{s:wavelet_analysis} and in all Section \ref{s:estimation}. The bivariate framework \eqref{e:bivariate_OFBM_h1_h2_real} is only used in limits at coarse scales (Section \ref{s:estimation}), due to the availability of convenient formulas for eigenvalues and eigenvectors.

\begin{remark}\label{r:parametrization}
In regard to the parametrization, the connection between $\Sigma$ and $(H,AA^*)$ can be obtained implicitly based on the second moment of the harmonizable representation \eqref{e:OFBM_harmonizable} at $t = 1$, and it can be worked out explicitly under \eqref{e:time_reversibility} and stronger additional conditions, e.g., assuming $H$ is diagonalizable with real eigenvalues and eigenvectors.

The condition \eqref{e:bivariate_OFBM_h1_h2_real} renders OFBM identifiable, namely, the mapping from the parametrization $H$ into the space of OFBM laws is injective for a fixed spectral parameter $AA^*$. This is a consequence of only considering Hurst matrices $H$ with real eigenvalues. For a general discussion of the (non)identifiability of OFBM, see Didier and Pipiras \cite{didier:pipiras:2012}.
\end{remark}


\section{Wavelet analysis}\label{s:wavelet_analysis}

In this section, we carry out the wavelet analysis of $n$-dimensional OFBM. The proof of Proposition \ref{p:wavelet_coefs_properties} can be found in Section \ref{s:additional}, whereas Section \ref{s:discretized_wavelet} contains the proofs of Proposition \ref{p:decay_Cov_wavelet_coefs}, Proposition \ref{p:4th_moments_wavecoef} and Theorem \ref{t:asymptotic_normality_wavecoef_fixed_scales}.

\subsection{Basic properties}

The normalized wavelet transform \eqref{wavelet_transform} is itself a vector-valued random field in the scale and shift parameters $j$ and $k$, respectively. It will be convenient to make the change of variables $z = 2^{-j}t - k$, and reexpress
\begin{equation}\label{e:wavelet_transform_after_change_var}
D(2^j,k) =  \int_{\bbR} \psi(z)B_H(2^j z + 2^j k) dz.
\end{equation}
As in the univariate case, the wavelet coefficients of OFBM exhibit a number of nice properties. The next proposition describes such properties as well as the general form of the wavelet spectrum (variance).
\begin{proposition}\label{p:wavelet_coefs_properties}
Under the assumptions (OFBM1)--(2), let $\{D(2^j,k)\}_{j \in \bbN, k \in \bbZ}$ be as in \eqref{e:wavelet_transform_after_change_var}. Then,
\begin{itemize}
\item [(P1)] the wavelet transform \eqref{wavelet_transform} is well-defined in the mean square sense, and $ED(2^j,k)=0$;
\item [(P2)] (stationarity for a fixed scale) $\{D(2^j,k+h)\}_{k \in \bbZ} \stackrel{{\mathcal L}}= \{D(2^j,k)\}_{k \in \bbZ}$, $h \in \bbZ$;
\item [(P3)] (operator self-similarity over different scales) $\{D(2^j,k)\}_{k \in \bbZ} \stackrel{{\mathcal L}}= \{2^{jH}D(1,k)\}_{k \in \bbZ}$;
\item [(P4)] for some $C> 0$, the wavelet spectrum $EW(2^j,k) \equiv EW(2^j)$ is given by
\begin{equation}\label{e:EW(j)_Fourier}
EW(2^j) = C \int_{\bbR} (x^{-D}_{+}AA^* x^{-D^*}_{+} + x^{-D}_{-}\overline{AA^*} x^{-D^*}_{-}) \frac{|\widehat{\psi}(2^j x)|^2}{x^2}dx;
\end{equation}
\item [(P5)] the wavelet spectrum satisfies the operator scaling relation
$$
EW(2^j) = 2^{jH} EW(1) 2^{jH^*}=
2^{jH} \Big\{\int_{\bbR}\int_{\bbR}\psi(t)\psi(t') \hspace{1mm}EB_H(t)B_H(t')^* dtdt' \Big\} 2^{jH^*},
$$
$j \in \bbN$. In particular, under \eqref{e:time_revers},
\begin{equation}\label{e:EWj_timereversible}
EW(2^j) = - \frac{1}{2}\hspace{1mm}2^{jH} \Big\{ \int_{\bbR}\int_{\bbR} \psi(t) \psi(t')\hspace{1mm}|t-t'|^{H}\Sigma |t-t'|^{H^*}dtdt' \Big\}2^{jH^*};
\end{equation}
\item [(P6)] in analogy to ($P5$), $W(2^j) \stackrel{d}= 2^{jH}( \frac{1}{K_j}\sum^{K_j}_{k=1} D(1,k)D(1,k)^* ) 2^{jH^*}$, where $K_j$ is given in \eqref{e:W(j)};
\item [(P7)] the wavelet spectrum has full rank, namely, $\det EW(2^j)  \neq 0$, $j \in \bbN$.
\end{itemize}
\end{proposition}


Fix some $0 < \delta < 1/2$, and consider the range of wavelet parameters $j$, $j'$, $k$, $k'$ such that
\begin{equation}\label{e:wave_param_range}
\Big| \frac{\max\{2^j,2^{j'}\}}{2^j k-2^{j'}k'}\Big| \leq \frac{\delta}{\max\{1,\textnormal{length}(\textnormal{supp} (\psi))\}} .
\end{equation}
If the parameters ($j, k$) and ($j', k'$) of two wavelet coefficients satisfy \eqref{e:wave_param_range}, then we can interpret that the latter are ``far apart" in the parameter space. The next result provides a notion of decay of the covariance between wavelet coefficients under \eqref{e:wave_param_range}. The proof is similar to that for the univariate case, but we provide it in Section \ref{s:asympt_normality_fixed_scales} for the reader's convenience.
\begin{proposition}\label{p:decay_Cov_wavelet_coefs}
Under the assumptions (OFBM1)--(3) and \eqref{e:wave_param_range}, the covariance between wavelet coefficients \eqref{e:wavelet_transform_after_change_var} satisfies the relation
$$
ED(2^j,k)D(2^{j'},k')^*
$$
\begin{equation}\label{e:decay_Cov_wavelet_coefs}
= - \frac{1}{2}\frac{2^{(j+j')N_{\psi}}}{|2^j k- 2^{j'}k'|^{2N_{\psi}}}\hspace{1mm}\Big\{|2^jk-2^{j'}k'|^{H} \Big( O^{N_{\psi}}_{i_1,i_2}(1)\Big)_{i_1,i_2=1,\hdots,n} |2^jk-2^{j'}k'|^{H^*}\Big\},
\end{equation}
where $\Big( O^{N_{\psi}}_{i_1,i_2}(1)\Big)_{i_1,i_2=1,\hdots,n}$ is an entry-wise bounded symmetric-matrix-valued function that depends only on $N_{\psi}$. As a consequence,
\begin{equation}\label{e:decay_Cov_wavelet_coefs_in_norm}
\|ED(2^j,k) D(2^{j'},k')^*\| \leq C(N_{\psi},j,j') \frac{|\log^{\kappa} |2^j k- 2^{j'}k'||}{|2^j k- 2^{j'}k'|^{2N_{\psi}-2 \max_{h \in \textnormal{eig}(H)}\Re(h)}},
\end{equation}
where $\kappa$ is the dimension of the largest Jordan block in the spectrum of $H$, and $C(N_{\psi},j,j') =  C(N_{\psi}) \hspace{1mm} 2^{(j+j')N_{\psi}}$ for some constant $C(N_{\psi}) >0$.
\end{proposition}

\subsection{Asymptotics for sample wavelet transforms: fixed scales}

As typical in the asymptotic study of averages, we begin by investigating the asymptotic covariance of the sample wavelet transforms $W(2^j)$.

For FBM, the asymptotic covariance between wavelet transforms $W(2^j), W(2^{j'}) \in \bbR$ is not available in closed form since it depends on the wavelet function, which is itself not available in closed form (c.f.\ Bardet \cite{bardet:2002}, Proposition II.3). Operator self-similarity adds a layer of intricacy, since in general exact entry-wise scaling relations are not present.

For notational simplicity, let
\begin{equation}\label{e:X,Y_stand_for_wavecoefs}
{\mathbf X}^* = (X_1,\hdots,X_n) = (d_{1}(2^j,k), \hdots,d_{n}(2^j,k)), \quad
{\mathbf Y}^* = (Y_1,\hdots,Y_n) = (d_{1}(2^{j'},k'), \hdots,d_{n}(2^{j'},k')),
\end{equation}
where $d_i(2^j,k)$, $i = 1,\hdots,n$, is the $i$-th entry of the wavelet transform vector $D(2^j,k)$. The bivariate case serves to illustrate the computation of covariances.
\begin{example}
For a zero mean, Gaussian random vector ${\mathbf Z} \in \bbR^m$, the Isserlis theorem (e.g., Vignat \cite{vignat:2012}) yields
\begin{equation}\label{e:Isserlis_theorem}
E(Z_1 \hdots Z_{2k}) = \sum \prod E(Z_i Z_j), \quad E(Z_1 \hdots Z_{2k+1}) = 0, \quad k = 1,\hdots, \lfloor m/2 \rfloor.
\end{equation}
The notation $\sum \prod$ stands for adding over all possible $k$-fold products of pairs $E(Z_i Z_j)$, where the indices partition the set $1,\hdots,2k$. So, let ${\mathbf X}$ and ${\mathbf Y}$ be as in \eqref{e:X,Y_stand_for_wavecoefs} with $n=2$. Then,
$$
\textnormal{Cov}
\left(  \left(\begin{array}{c}
 X^{2}_1\\
X^{2}_2\\
X_1 X_2
\end{array}\right),
\left(\begin{array}{c}
 Y^{2}_1\\
Y^{2}_2\\
Y_1 Y_2
\end{array}\right)\right)
$$
\begin{equation}\label{e:fourth_moments}
=\left(\begin{array}{ccc}
 2(E(X_1Y_1))^2 &  2(E(X_1 Y_2))^2 & 2E(X_1 Y_1)E(X_1Y_2)\\
2(E(X_2 Y_1))^2 &  2(E(X_2Y_2))^2 & 2E(X_2Y_1)E(X_2Y_2)\\
2E(X_1Y_1)E(X_2Y_1)  &  2 E(X_1Y_2)E(X_2Y_2) &  c_{33}\\
\end{array}\right),
\end{equation}
where $c_{33} = E(X_1Y_1)E(X_2Y_2) + E(X_1Y_2)E(X_2Y_1)$.
\end{example}
Expression \eqref{e:fourth_moments} shows that the asymptotic behavior of the second moments of $W(2^j)$ involves several cross products. A notationally economical way of tackling this difficulty is by resorting to Kronecker products. For instance, in the bivariate case,
$$
M(4,\bbR) \ni EXY^* \otimes EXY^* =
\left(\begin{array}{cc}
E(X_1Y_1) & E(X_1 Y_2) \\
E(X_2Y_1) & E(X_2 Y_2) \\
\end{array}\right)
\otimes
\left(\begin{array}{cc}
E(X_1Y_1) & E(X_1 Y_2) \\
E(X_2Y_1) & E(X_2 Y_2) \\
\end{array}\right)
$$
contains all the 9 terms (two-fold products of cross moments), as well as a few repeated ones, needed to express (\ref{e:fourth_moments}). In view of \eqref{e:Isserlis_theorem}, this fact extends to general dimension $n$ by means of the relations
\begin{equation}\label{e:Isserlis}
\textnormal{Cov}(X_{i_1}X_{i_2},Y_{j_1}Y_{j_2}) = E(X_{i_1} Y_{j_1})E(X_{i_2} Y_{j_2})+E(X_{i_1} Y_{j_2})E(X_{i_2} Y_{j_1}),
\end{equation}
for $i_1,i_2,j_1,j_2 = 1,\hdots, n$. The next proposition provides an expression that encompasses the asymptotic fourth moments of the wavelet coefficients.
\begin{proposition}\label{p:4th_moments_wavecoef}
Let $B_H = \{B_{H}(t)\}_{t \in \bbR}$ be an OFBM under the assumptions (OFBM1) -- (3) and let $K_j$ be as in \eqref{e:W(j)}. As $\nu \rightarrow \infty$, for every pair of octaves $j$, $j'$,
\begin{itemize}
\item [$(i)$]
$$
\sqrt{K_j}\sqrt{K_{j'}}\frac{1}{K_j}\frac{1}{K_{j'}} \sum^{K_j}_{k=1}\sum^{K_{j'}}_{k'=1}
ED(2^j,k)D(2^{j'},k')^* \otimes ED(2^j ,k)D(2^{j'},k')^*
$$
\begin{equation}\label{e:limiting_kron}
\rightarrow 2^{-(j+j')/2} \gcd(2^{j},2^{j'}) \sum^{\infty}_{z= - \infty} \Omega_{j,j'}(z)
\otimes \Omega_{j,j'}(z),
\end{equation}
where
$$
\Omega_{j,j'}(z) :=  -\frac{1}{2} \int_{\bbR} \int_{\bbR} \psi( t )\psi( t')
|(2^j t - 2^{j'} t')+ z \hspace{1mm} \textnormal{gcd}(2^j,2^{j'})|^H\Sigma
|(2^j t - 2^{j'} t')+ z \hspace{1mm}\textnormal{gcd}(2^j,2^{j'})|^{H^*} dt' dt;
$$
\item [$(ii)$] there is a matrix $G_{jj'} \in M(n(n+1)/2,\bbR)$, not necessarily symmetric, such that
$$
\sqrt{K_j}\sqrt{K_{j'}}\hspace{1mm}\textnormal{Cov}(\vecoper_{{\mathcal S}}W(2^j),\vecoper_{{\mathcal S}}W(2^{j'})) \rightarrow G_{jj'},
$$
where the entries of $G_{jj'}$ can be retrieved from \eqref{e:limiting_kron} by means of \eqref{e:Isserlis} (see \eqref{e:vec_S} on the notation $\vecoper_{{\mathcal S}}$).
\end{itemize}
\end{proposition}

\begin{remark}
The definition of wavelet only requires $N_\psi \geq 1$, but $N_\psi \geq 2$ (see \eqref{e:N_psi}) is needed for the convergence in Proposition \ref{p:4th_moments_wavecoef} (see \eqref{e:cov_terms_close_together_converge}).
\end{remark}

The next theorem establishes the asymptotics for the vectorized sample wavelet transforms $( \textnormal{vec}_{{\mathcal S}} W(2^j) )_{j=j_1,\hdots,j_m}$ at a fixed set of octaves.

\begin{theorem}\label{t:asymptotic_normality_wavecoef_fixed_scales}
Let $B_H = \{B_{H}\}_{t \in \bbR}$ be an OFBM under the assumptions (OFBM1)--(3). Let $j_1 < \hdots < j_m$ be a fixed set of octaves. Then,
\begin{equation}\label{e:asymptotic_normality_wavecoef_fixed_scales}
\Big(\sqrt{K_j}(\textnormal{vec}_{{\mathcal S}} W(2^j)- \textnormal{vec}_{{\mathcal S}} EW(2^j) ) \Big)_{j = j_1 , \hdots, j_m} \stackrel{d}\rightarrow {\mathcal N}_{\frac{n(n+1)}{2} \times m}(0,F),
\end{equation}
as $\nu \rightarrow \infty$ (see \eqref{e:vec_S} on the notation $\vecoper_{{\mathcal S}}$). In \eqref{e:asymptotic_normality_wavecoef_fixed_scales}, the matrix $F \in {\mathcal S}(\frac{n(n+1)}{2}m,\bbR)$ has the form $F = (G_{jj'})_{j,j'=1,\hdots,m}$, where each block $G_{jj'} \in M(n(n+1)/2,\bbR)$ is described in Proposition \ref{p:4th_moments_wavecoef}.
\end{theorem}

\section{A wavelet-based estimator for bivariate OFBM}\label{s:estimation}

In this section, we switch to the bivariate framework \eqref{e:bivariate_OFBM_h1_h2_real}, i.e., $n = 2$. We draw upon explicit expressions for eigenvalues to establish the consistency and asymptotic normality of the estimators \eqref{e:def_estimator} as the wavelet scale grows according to a factor $a(\nu)\rightarrow \infty$, as $\nu \rightarrow \infty$. We also show the consistency and asymptotic normality, in a sense to be defined, of a sequence of eigenvectors associated with the smallest eigenvalue of the sample wavelet variance matrix.

The proposed estimators make use of the behavior at coarse scales of the sample wavelet variance
\begin{equation}\label{e:W(a(v)2^j)}
W(a(\nu)2^j) = \frac{1}{K_{a,j}}\sum^{K_{a,j}}_{k=1}D(a(\nu)2^{j},k)D(a(\nu)2^{j},k)^*, \quad K_{a,j} = \frac{\nu}{a(\nu)2^j}.
\end{equation}
Recall that $h_1 < h_2$ under \eqref{e:bivariate_OFBM_h1_h2_real}. In \eqref{e:W(a(v)2^j)}, $\{a(\nu)\}_{\nu \in \bbN}$ is assumed to be a dyadic sequence such that
\begin{equation}\label{e:a(nu)/J->infty}
a(\nu) \leq \frac{\nu}{2^j}, \quad \frac{a(\nu)}{\nu} + \frac{\nu}{a(\nu)^{1 + 2 h_1}}\rightarrow 0, \quad \nu \rightarrow \infty
\end{equation}
(see Remark \ref{r:choice_of_a(nu)} on the choice of $a(\nu)$ in practice).

We will make use of some basic relations for bivariate symmetric positive semidefinite matrices. Recall that for a matrix
\begin{equation}\label{e:S=(a_b_b_c)}
S = \left(\begin{array}{cc}
a & b\\
b & c
\end{array}\right)
\end{equation}
the eigenvalues can be expressed in closed form as
\begin{equation}\label{e:lambda1_lambda2_explicit}
\lambda_1 = \frac{(a+c) - \sqrt{\Delta}}{2} \leq  \lambda_2 = \frac{(a+c) + \sqrt{\Delta}}{2}, \quad \Delta = (a+c)^2 -4(ac- b^2).
\end{equation}
By Sylvester's criterion, positive semidefiniteness implies that $a + c \geq 0$. As a consequence, if $\det(S) > 0$,
\begin{equation}\label{e:lambda1_explicit}
\lambda_1 = \frac{1}{2}(a+c) \Big( 1 - \sqrt{1 - \frac{4(ac-b^2)}{(a+c)^2}}\Big) = \frac{1}{2}\frac{4(ac-b^2)}{a+c}\frac{\Big( 1 - \sqrt{1 - \frac{4(ac-b^2)}{(a+c)^2}}\Big)}{4(ac-b^2)/(a+c)^2},
\end{equation}
and
\begin{equation}\label{e:lambda2_explicit}
\lambda_2 = \frac{1}{2}(a+c) \Big( 1 + \sqrt{1 - \frac{4(ac-b^2)}{(a+c)^2}}\Big).
\end{equation}
Let $\textbf{v} = (v_1,v_2)^* \in \bbR^2$ be an eigenvector associated with an eigenvalue $\lambda$. By further assuming $b \neq 0$, the relation $S \textbf{v} = \lambda \textbf{v}$ yields
\begin{equation}\label{e:relation_v1_v2}
v_2 = \frac{\lambda - a}{b} \hspace{1mm}v_1.
\end{equation}

\subsection{The weak limit of eigenvalues}

The next definition describes the proposed estimators for the Hurst eigenvalues $h_1 < h_2$.

\begin{definition}\label{def:estimator}
Let $B_H = \{B_H(t)\}_{t \in \bbR}$ be an OFBM under the assumptions (OFBM1)--(4). For a dyadic number $a(\nu)$, let $W(a(\nu)2^j)$ be the associated (symmetric) sample wavelet spectrum at scale $a(\nu)2^j$, and let
\begin{equation}\label{e:lambda1_lambda2}
\lambda_{1}(a(\nu)2^j) \leq \lambda_{2}(a(\nu)2^j)
\end{equation}
be its eigenvalues. The wavelet estimators at scale $a(\nu)2^j$ of the eigenvalues $h_1 < h_2$ are defined, respectively, as in expression \eqref{e:def_estimator} with $a(\nu)2^j$ in place of $2^j$.
\end{definition}

By analogy to \eqref{e:lambda1_lambda2} and \eqref{e:def_estimator}, we denote the eigenvalues and normalized log-eigenvalues of $EW(a(\nu)2^j)$, respectively, by
\begin{equation}\label{e:lambda^E_1_lambda^E_2}
\lambda^E_{1}(a(\nu)2^j) \leq \lambda^E_{2}(a(\nu)2^j), \quad h^E_{1}(a(\nu)2^j) = \frac{\log \lambda^E_{1}(a(\nu)2^j)}{2\log(a(\nu)2^j)} \leq h^E_{2}(a(\nu)2^j) = \frac{\log \lambda^E_{2}(a(\nu)2^j)}{2\log(a(\nu)2^j)}.
\end{equation}
When developing asymptotics for the estimators \eqref{e:lambda1_lambda2}, in view of the operator self-similarity property $(P6)$ we will consider the matrix statistics
\begin{equation}\label{e:B^(j)}
\widehat{B}_a(2^j) := P^{-1}W_a(2^j)(P^{*})^{-1}, \quad j = j_1 < \hdots <j_m,
\end{equation}
where $ W_a(2^j) = \frac{1}{K_{a,j}} \sum^{K_{a,j}}_{k=1} D(2^j,k)D(2^j,k)^*$ (c.f.\ \eqref{e:W(j)} and \eqref{e:W(a(v)2^j)}). Each $\widehat{B}_a(2^j)$ is only a pseudo-estimator of
\begin{equation}\label{e:B(j)}
B(2^j) := P^{-1}EW(2^j) (P^{*})^{-1}, \quad P \in GL(2,\bbR),
\end{equation}
because its expression involves the unknown matrix parameter $P$. It will be convenient to describe the matrices \eqref{e:B^(j)}, \eqref{e:B(j)} entry-wise as
\begin{equation}\label{e:B^a(j)_B(j)}
\widehat{B}_a(2^j) = \Big(\widehat{b}_{i_1,i_2}(2^j) \Big)_{i_1,i_2=1,2}, \quad B(2^j) = \Big( b_{i_1,i_2}(2^j) \Big)_{i_1,i_2=1,2}.
\end{equation}
The asymptotic distribution of the matrix statistics \eqref{e:B^(j)} is given in the following lemma.
\begin{lemma}\label{l:vecB^a(2^j)_asympt}
For $m \in \bbN$, let $j_1 < \hdots < j_m$ be a set of fixed octaves $j$. Let $\Pi = \Big( \pi_{i_1,i_2}\Big)_{i_1,i_2=1,2} = P^{-1}$, and let $\widehat{B}_a(2^j)$, $B(2^j)$ be as in \eqref{e:B^(j)}, \eqref{e:B(j)}, respectively. Under the assumptions (OFBM1) -- (4) and \eqref{e:a(nu)/J->infty},
\begin{equation}\label{e:asymptotic_normality_pseudoestimator_Bhat}
\left( \begin{array}{c}
\sqrt{K_{a,j}}(\vecoper_{{\mathcal S}}\widehat{B}_a(2^j) - \vecoper_{{\mathcal S}}B(2^j))
\end{array}\right)_{j=j_1,\hdots,j_m}\stackrel{d}\rightarrow N(0,\Sigma_{B}(j_1,\hdots,j_m)), \quad \nu \rightarrow \infty.
\end{equation}
In \eqref{e:asymptotic_normality_pseudoestimator_Bhat},
\begin{equation}\label{e:SigmaB}
\Sigma_{B}(j_1,\hdots,j_m) = \textnormal{diag}_{m}({\mathcal P})F \textnormal{diag}_{m}({\mathcal P})^*,
\end{equation}
where $F$ is the asymptotic covariance matrix in \eqref{e:asymptotic_normality_wavecoef_fixed_scales}, $\textnormal{diag}_m({\mathcal P})$ is as defined in \eqref{e:block_diagonal_matrices}, and
$$
{\mathcal P} = \left(\begin{array}{ccc}
\pi^2_{11} & 2 \pi_{11}\pi_{12} & \pi^{2}_{12}\\
\pi_{11}\pi_{21} & \pi_{11}\pi_{22} + \pi_{12}\pi_{21} & \pi_{12} \pi_{22}\\
\pi^2_{21} & 2 \pi_{21}\pi_{22} & \pi^{2}_{22}\\
\end{array}\right).
$$
\end{lemma}
\begin{proof}
For any $j$, a brief calculation shows that $\vecoper_{{\mathcal S}}(\Pi W_{a}(2^j)\Pi^*) = {\mathcal P} \vecoper_{{\mathcal S}}W_a(2^j)$.
Likewise, $\vecoper_{{\mathcal S}}(\Pi EW(2^j)\Pi^*) = {\mathcal P} \vecoper_{{\mathcal S}}EW(2^j)$. Therefore, we can recast the left-hand side of \eqref{e:asymptotic_normality_pseudoestimator_Bhat} as
$$
\textnormal{diag}_{m}({\mathcal P}) \left( \begin{array}{c}
\sqrt{K_{a,j}}(\vecoper_{{\mathcal S}}W_a(2^j) - \vecoper_{{\mathcal S}}EW(2^j))
\end{array}\right)_{j=j_1,\hdots,j_m}.
$$
The weak limit in \eqref{e:asymptotic_normality_pseudoestimator_Bhat} is now a consequence of \eqref{e:a(nu)/J->infty} and Theorem \ref{t:asymptotic_normality_wavecoef_fixed_scales}. $\Box$\\
\end{proof}

The next lemma contains expressions for the wavelet spectrum and its sample counterpart. The notation $(\hspace{0.5mm}\cdot \hspace{0.5mm})_{j=j_1,\hdots,j_m}$ designates a vector of $2 \times 2$ matrices.
\begin{lemma}\label{l:EW(a(n)2j)_W(a(n)2j)_scaling}
For $m \in \bbN$, let $j_1 < \hdots < j_m$ be a set of fixed scales $j$. Let $\widehat{B}_a(2^j)$, $B(2^j)$ be as in \eqref{e:B^(j)}, \eqref{e:B(j)}, respectively, and let $\widehat{b}_{i_1,i_2}(2^j)$, $b_{i_1,i_2}(2^j)$ be as in \eqref{e:B^a(j)_B(j)}. For a dyadic number $a(\nu)$, under the assumptions (OFBM1) -- (4), we can express
\begin{equation}\label{e:W(a(nu)2^j)}
\Big(W(a(\nu)2^{j}) \Big)_{j=j_1,\hdots,j_m} \stackrel{d}= \left(\begin{array}{cc}
\widehat{a}_{j,a(\nu)} & \widehat{b}_{j,a(\nu)} \\
\widehat{b}_{j,a(\nu)} & \widehat{c}_{j,a(\nu)}
\end{array}\right)_{j=j_1,\hdots,j_m},
\end{equation}
where
$$
\widehat{a}_{j,a(\nu)} = p^2_{11} \widehat{b}_{11}(2^j)a(\nu)^{2h_1} + 2p_{11}p_{12}\widehat{b}_{12}(2^j)a(\nu)^{h_1 + h_2} + p^{2}_{12}\widehat{b}_{22}(2^j)a(\nu)^{2h_2},
$$
$$
\widehat{b}_{j,a(\nu)} = p_{11}p_{21}\widehat{b}_{11}(2^j)a(\nu)^{2h_1} + (p_{11}p_{22}+ p_{12}p_{21})\widehat{b}_{12}(2^j)a(\nu)^{h_1+h_2} + p_{12}p_{22}\widehat{b}_{22}(2^j)a(\nu)^{2h_2},
$$
\begin{equation}\label{e:Shat=a(nu)H_EW(j)_a(nu)H*=(a_b_b_c)}
\widehat{c}_{j,a(\nu)} = p^{2}_{21}\widehat{b}_{11}(2^j)a(\nu)^{2h_1} + 2p_{21}p_{22}\widehat{b}_{12}(2^j)a(\nu)^{h_1 + h_2} + p^{2}_{22}\widehat{b}_{22}(2^j)a(\nu)^{2h_2}.
\end{equation}
Likewise,
\begin{equation}\label{e:EW(a(nu)2^j)}
E W(a(\nu)2^{j}) = \left(\begin{array}{cc}
a_{j,a(\nu)} & b_{j,a(\nu)} \\
b_{j,a(\nu)} & c_{j,a(\nu)}
\end{array}\right),
\end{equation}
where
$$
a_{j,a(\nu)} = p^2_{11} b_{11}(2^j)a(\nu)^{2h_1} + 2p_{11}p_{12}b_{12}(2^j)a(\nu)^{h_1 + h_2} + p^{2}_{12}b_{22}(2^j)a(\nu)^{2h_2},
$$
$$
b_{j,a(\nu)} = p_{11}p_{21}b_{11}(2^j)a(\nu)^{2h_1} + (p_{11}p_{22}+ p_{12}p_{21})b_{12}(2^j)a(\nu)^{h_1+h_2} + p_{12}p_{22}b_{22}(2^j)a(\nu)^{2h_2},
$$
\begin{equation}\label{e:S=a(nu)HEW(j)a(nu)H*=(a_b_b_c)}
c_{j,a(\nu)} = p^{2}_{21}b_{11}(2^j)a(\nu)^{2h_1} + 2p_{21}p_{22}b_{12}(2^j)a(\nu)^{h_1 + h_2} + p^{2}_{22}b_{22}(2^j)a(\nu)^{2h_2}.
\end{equation}
\end{lemma}
\begin{proof}
The operator scaling properties $(P5)$ and $(P6)$ yield
\begin{equation}\label{e:W(a(nu)2^j)=PdiagB^diagP*}
\Big(W(a(\nu)2^{j}) \Big)_{j=j_1,\hdots,j_m} \stackrel{d}=\Big(P \textnormal{diag}(a(\nu)^{h_1},a(\nu)^{h_2})\hspace{1mm}\widehat{B}_a(2^j) \hspace{1mm}\textnormal{diag}(a(\nu)^{h_1},a(\nu)^{h_2})P^*\Big)_{j=j_1,\hdots,j_m},
\end{equation}
\begin{equation}\label{e:EW(a(nu)2^j)=PdiagBdiagP*}
E W(a(\nu)2^{j}) = P \textnormal{diag}(a(\nu)^{h_1},a(\nu)^{h_2})\hspace{1mm}B(2^j) \hspace{1mm}\textnormal{diag}(a(\nu)^{h_1},a(\nu)^{h_2})P^*,
\end{equation}
from which \eqref{e:Shat=a(nu)H_EW(j)_a(nu)H*=(a_b_b_c)} and \eqref{e:S=a(nu)HEW(j)a(nu)H*=(a_b_b_c)} follow. $\Box$\\
\end{proof}

The next theorem establishes the consistency and asymptotic normality of the estimators described in Definition \ref{def:estimator}. Intuitively, the theorem states that
$$
\widehat{h}_1(a(\nu)2^j) \approx h_1 , \quad \widehat{h}_2(a(\nu)2^j) \approx h_2 \textnormal{ with convergence rate $2 \log (a(\nu)2^j) \sqrt{K_{a,j}}$}.
$$

\begin{theorem}\label{t:asympt_normality_lambda1}
For $m \in \bbN$, let $j_1 < \hdots < j_m$ be a set of fixed octaves $j$. Let $\widehat{h}_{1}$, $\widehat{h}_{2}$, $h^E_{1}$, $h^E_{2}$ be as in \eqref{e:def_estimator} and \eqref{e:lambda^E_1_lambda^E_2}. Under the assumptions (OFBM1) -- (4) and \eqref{e:a(nu)/J->infty}, as $\nu \rightarrow \infty$,
\begin{itemize}
\item [($i$)]
\begin{equation}
(\widehat{h}_1(a(\nu)2^{j}), \widehat{h}_2(a(\nu)2^{j})) \stackrel{P}\rightarrow (h_1,h_2), \quad (h^E_1(a(\nu)2^{j}), h^E_2(a(\nu)2^{j})) \rightarrow (h_1,h_2),
\end{equation}
for $j = j_1,\hdots,j_m$;
\item [($ii$)]
$$
\left( \begin{array}{c}
2\log(a(\nu)2^{j})\sqrt{K_{a,j}}[\widehat{h}_1(a(\nu)2^{j}) - h^E_1(a(\nu)2^{j})]  \\
2\log(a(\nu)2^{j})\sqrt{K_{a,j}}[\widehat{h}_2(a(\nu)2^{j}) - h^E_2(a(\nu)2^{j})] \\
\end{array}\right)_{j=j_1,\hdots,j_m}
$$
\begin{equation}\label{e:claim_theo_asympt_normality_lambda1}
\stackrel{d}\rightarrow N(0,\Sigma_{h_1,h_2}(j_1,\hdots,j_m)).
\end{equation}
In \eqref{e:claim_theo_asympt_normality_lambda1},
\begin{equation}\label{e:Sigmah1h2}
\Sigma_{h_1,h_2}(j_1,\hdots,j_m) = {\mathcal Q} \Sigma_B(j_1,\hdots,j_m) {\mathcal Q}^*,
\end{equation}
where $\Sigma_B(j_1,\hdots,j_m)$ is the covariance matrix in \eqref{e:SigmaB}, and ${\mathcal Q} = \Big( {\mathcal Q}_{jj'}\Big)_{j,j'= j_1,\hdots,j_m}$ is a block matrix whose blocks have dimension $2 \times 3$ and satisfy
$$
{\mathcal Q}_{jj} = \left(\begin{array}{ccc}
\frac{b_{22}(2^j)}{\det B(2^j)} & - \frac{2 b_{12}(2^j)}{\det B(2^j)} & \Big( \frac{b_{11}(2^j)b_{22}(2^j)}{\det B(2^j)} - 1\Big)\frac{1}{b_{22}(2^j)}\\
0 & 0 & \frac{1}{b_{22}(2^j)}
\end{array}\right) \quad \textnormal{and }{\mathcal Q}_{jj'} = {\boldsymbol 0}, \textnormal{ if }j \neq j'.
$$
\end{itemize}
\end{theorem}
\begin{proof}
Fix an arbitrary $j$. For notational simplicity we write $B = B(2^j) = \Big( b_{i_1,i_2}\Big)_{i_1,i_2=1,2}$. We also drop the subscripts $j,\nu$ in the expressions \eqref{e:Shat=a(nu)H_EW(j)_a(nu)H*=(a_b_b_c)} and \eqref{e:S=a(nu)HEW(j)a(nu)H*=(a_b_b_c)}. Recall that the asymptotic distribution for $\Big(\widehat{B}_a(2^j)\Big)_{j=j_1,\hdots,j_m}$ is given by \eqref{e:asymptotic_normality_pseudoestimator_Bhat}.

The statement ($i$) is a consequence of Lemma \ref{l:elementary_results} and Theorem \ref{t:asymptotic_normality_wavecoef_fixed_scales}. In regard to statement ($ii$), the asymptotics will be written out for just one general term indexed $j$, but the conclusions apply to the whole vector comprising the terms associated with $j = j_1,\hdots,j_m$. Since all meaningful limits will boil down to a function of the variables that appear on the right-hand side of \eqref{e:asymptotic_normality_pseudoestimator_Bhat}, they depend on the omitted, fixed octave $j$.

Note that $2\log(a(\nu)2^{j})[\widehat{h}_i(a(\nu)2^{j}) - h^E_i(a(\nu)2^{j})]  = \log \lambda_{i}(a(\nu)2^{j}) - \log \lambda^E_{i}(a(\nu)2^{j})$, $i=1,2$. Wiht probability going to 1, we can reexpress the latter for $i = 1$ as
$$
\Big\{\log \frac{\widehat{a}\widehat{c}-\widehat{b}^2}{\widehat{a} + \widehat{c}}
- \log \frac{ac-b^2}{a + c} \Big\}
$$
\begin{equation}\label{e:lambda1-lambdaE1=sum}
+ \Big\{\log \frac{ \sqrt{1 + (-4)(\widehat{a}\widehat{c}-\widehat{b}^2)/(\widehat{a}+\widehat{c})^2}-1}{(-4)(\widehat{a}\widehat{c}-\widehat{b}^2)/(\widehat{a}+\widehat{c})^2}
- \log \frac{ \sqrt{1 + (-4)(ac-b^2)/(a+c)^2}-1}{(-4)(ac-b^2)/(a+c)^2}\Big\}.
\end{equation}
The mean value theorem will allow us to show that the second term in the sum \eqref{e:lambda1-lambdaE1=sum} does not contribute to the asymptotic limit. Let
\begin{equation}\label{e:f1_f2}
f_1(x) = (\sqrt{1-x}-1)(-x)^{-1}, \quad f_2(x) = \log f_1(x).
\end{equation}
By a second order Taylor expansion of the square root function around 1,
\begin{equation}\label{e:f1'(x)->1/4}
f_1'(x) = \frac{1}{\sqrt{1-x}} \frac{-\frac{x}{2} - (\sqrt{1-x}-1)}{x^2} = \frac{1}{\sqrt{1-x}}\Big\{\frac{1}{8} + o(1)\Big\} \rightarrow \frac{1}{8}, \quad x \rightarrow 0.
\end{equation}
Therefore, $\lim_{x \rightarrow 0}f_2'(x) = \frac{1}{4}$. Thus, the second term in the sum \eqref{e:lambda1-lambdaE1=sum} can be reexpressed as
$$
f_2\Big( \frac{-4(\widehat{a}\widehat{c}-\widehat{b}^2)}{(\widehat{a}+\widehat{c})^2}\Big) - f_2\Big(\frac{-4(ac-b^2)}{(a+c)^2} \Big)
= f_2'(\xi_{abc}) \Big\{\frac{4(ac-b^2)}{(a+c)^2} - \frac{4(\widehat{a}\widehat{c}-\widehat{b}^2)}{(\widehat{a}+\widehat{c})^2}\Big\},
$$
where the random variable $\xi_{abc}$ lies between $\frac{4(ac-b^2)}{(a+c)^2}$ and $\frac{4(\widehat{a}\widehat{c}-\widehat{b}^2)}{(\widehat{a}+\widehat{c})^2}$. As a consequence of Lemma \ref{l:elementary_results}, it satisfies the limit $\xi_{abc} \stackrel{P}\rightarrow 0$; hence, $f_{2}'(\xi_{abc}) \stackrel{P}\rightarrow \frac{1}{4}$. By expressions \eqref{e:W(a(nu)2^j)=PdiagB^diagP*} and \eqref{e:EW(a(nu)2^j)=PdiagBdiagP*},
\begin{equation}\label{e:ac-b^2=det}
\widehat{a}\widehat{c}-\widehat{b}^2 = a(\nu)^{2(h_1 + h_2)} \det W_a(2^j), \quad ac-b^2 = a(\nu)^{2(h_1 + h_2)}\det EW(2^j).
\end{equation}
Thus,
$$
\frac{(ac-b^2)}{(a+c)^2} - \frac{(\widehat{a}\widehat{c}-\widehat{b}^2)}{(\widehat{a}+\widehat{c})^2} = \frac{(ac-b^2)-(\widehat{a}\widehat{c}-\widehat{b}^2)}{(a+c)^2} + (\widehat{a}\widehat{c} - \widehat{b}^2)\Big\{\frac{1}{(a+c)^2} - \frac{1}{(\widehat{a}+\widehat{c})^2}\Big\}
$$
$$
= \frac{1}{a(\nu)^{2(h_2-h_1)}} \Big\{ \frac{\det W_a(2^j) - \det EW(2^j)}{[O(a(\nu)^{h_1-h_2}) + (p^{2}_{12}+p^{2}_{22})b_{22}]^2}
$$
\begin{equation}\label{e:lambda1-lambdaE1=sum_second_term_reexpressed}
+ \det W_a(2^j) \Big[ \frac{1}{[O(a(\nu)^{h_1-h_2})+ (p^{2}_{12}+p^{2}_{22})b_{22}]^2}
- \frac{1}{[O_P(a(\nu)^{h_1-h_2})+ (p^{2}_{12}+p^{2}_{22})\widehat{b}_{22}]^2}\Big] \Big\}.
\end{equation}
After premultiplication by the rate factor $\sqrt{K_{a,j}}$, by \eqref{e:a(nu)/J->infty} the expression \eqref{e:lambda1-lambdaE1=sum_second_term_reexpressed} becomes
$$
\frac{1}{a(\nu)^{2(h_2-h_1)}} \Big\{ \frac{O_P(1)}{[O(a(\nu)^{h_1-h_2})+ (p^{2}_{12}+p^{2}_{22})b_{22}]^2}
$$
\begin{equation}\label{e:lambda1-lambdaE1=sum_second_term_limit}
+ \det W_a(2^j) \Big[ \frac{ O_P(1) }{[O(a(\nu)^{h_1-h_2})+ (p^{2}_{12}+p^{2}_{22})b_{22}]^2 [O_P(a(\nu)^{h_1-h_2})+ (p^{2}_{12}+p^{2}_{22})\widehat{b}_{22}]^2}\Big] \Big\} \stackrel{P}\rightarrow 0.
\end{equation}
The numerators of the first and second terms appearing in the sum on the left-hand side of \eqref{e:lambda1-lambdaE1=sum_second_term_limit} are bounded in probability as a consequence of Theorem \ref{t:asymptotic_normality_wavecoef_fixed_scales} and of applying the mean value theorem to the function $f_3(x) = x^2$.

We now look at the first term in the sum \eqref{e:lambda1-lambdaE1=sum}. In view of the relations \eqref{e:ac-b^2=det} and by applying the mean value theorem twice, it can be rewritten as
$$
\log \frac{\det W_a(2^j)}{O_P(a(\nu)^{h_1-h_2}) + (p^2_{12}+p^2_{22})\widehat{b}_{22}} - \log \frac{\det EW(2^j)}{O(a(\nu)^{h_1-h_2}) + (p^2_{12}+p^2_{22})b_{22}}
$$
\begin{equation}\label{e:lambda1-lambdaE1=sum_first_term}
= \frac{1}{\xi_{W}} \{\det W_a(2^j) - \det EW(2^j)\}- \frac{1}{\xi_{b_{22}}} \Big\{O_P\Big(\frac{a(\nu)^{h_1-h_2}}{\sqrt{K_{a,j}}} \Big) + (p^{2}_{12}+p^2_{22})(\widehat{b}_{22} - b_{22}) \Big\}.
\end{equation}
The random variable $\xi_{W}$ lies between $\det W_a(2^j)$ and $\det EW(2^j)$, whereas the random variable $\xi_{b_{22}}$ lies between $O_P(a(\nu)^{h_1-h_2}) + (p^2_{12}+p^2_{22})\widehat{b}_{22}$ and $O(a(\nu)^{h_1-h_2}) + (p^2_{12}+p^2_{22})b_{22}$. Moreover, these random variables display asymptotic behavior $\xi_{W} \stackrel{P}\rightarrow \det EW(2^j)$, $\xi_{b_{22}} \stackrel{P}\rightarrow (p^{2}_{12}+p^2_{22})b_{22}$, $\nu \rightarrow \infty$, where the latter limits stem from Theorem \ref{t:asymptotic_normality_wavecoef_fixed_scales} and Lemma \ref{l:vecB^a(2^j)_asympt}, respectively. Therefore, after premultiplication by the rate factor $\sqrt{K_{a,j}}$, the expression \eqref{e:lambda1-lambdaE1=sum_first_term} is asymptotically equivalent in probability to
\begin{equation}\label{e:theo_asympt_normality_lambda1_explicit}
\sqrt{K_{a,j}} \hspace{0.5mm} \frac{\det \widehat{B}_a(2^j) - \det B(2^j)}{\det B(2^j)} - \sqrt{K_{a,j}}\hspace{0.5mm} \frac{\widehat{b}_{22}- b_{22}}{b_{22}}.
\end{equation}
Let $f_4(x,y,z) = xz - y^2$. By the mean value theorem,
\begin{equation}\label{e:detBhat_Taylor}
\det \widehat{B}_a(2^j) - \det B(2^j) = \nabla f_4(\xi_{b_{11}},\xi_{b_{12}},\xi_{b_{22}})(\vecoper_{{\mathcal S}}\widehat{B}_a(2^j) - \vecoper_{{\mathcal S}}B(2^j)),
\end{equation}
where $\nabla f_4(\xi_{b_{11}},\xi_{b_{12}},\xi_{b_{22}})  \stackrel{P}\rightarrow \nabla f_4(b_{11},b_{12},b_{22})$ as a consequence of Lemma \ref{l:vecB^a(2^j)_asympt}. Reintroducing $j$, from \eqref{e:lambda1-lambdaE1=sum_second_term_limit}, \eqref{e:theo_asympt_normality_lambda1_explicit} and \eqref{e:detBhat_Taylor}, we conclude that
$$
\sqrt{K_{a,j}}(\log \lambda_1(a(\nu)2^j)-\log \lambda^E_1(a(\nu)2^j))\stackrel{P}\sim \frac{b_{22}(2^j)}{\det B(2^j)} \sqrt{K_{a,j}} (\widehat{b}_{11}(2^j)- b_{11}(2^j))
$$
\begin{equation}\label{e:lambda1_asympt_equiv}
- 2 \frac{b_{12}(2^j) }{\det B(2^j)} \sqrt{K_{a,j}} (\widehat{b}_{12}(2^j)- b_{12}(2^j)) + \Big(\frac{b_{11}(2^j)b_{22}(2^j)}{\det B(2^j)}-1\Big) \sqrt{K_{a,j}} \Big(\frac{\widehat{b}_{22}(2^j)- b_{22}(2^j)}{b_{22}(2^j)} \Big).
\end{equation}

A similar calculation can be developed for $\lambda_2(a(\nu)2^{j}) - \lambda^E_2(a(\nu)2^{j})$. The latter can be recast as
\begin{equation}\label{e:lambda2-lambda2E=sum}
\Big\{\log(\widehat{a}+\widehat{c}) - \log(a+c)\Big\} + \Big\{\log\Big( 1 + \sqrt{1 - \frac{4(\widehat{a}\widehat{c}-\widehat{b}^2)}{(\widehat{a}+\widehat{c})^2}}\Big) - \log\Big( 1 + \sqrt{1 - \frac{4(ac-b^2)}{(a+c)^2}}\Big)\Big\}.
\end{equation}
By the mean value theorem based on the function $f_5(x) = \log(1 + \sqrt{1 - x})$, after premultiplication by the rate factor $\sqrt{K_{a,j}}$ the second term in the sum \eqref{e:lambda2-lambda2E=sum} can be reexpressed as
$$
f_5'(\xi_{abc}) \sqrt{K_{a,j}}\Big\{ \frac{4(\widehat{a}\widehat{c}-\widehat{b}^2)}{(\widehat{a}+\widehat{c})^2} - \frac{4(ac-b^2)}{(a+c)^2}\Big\} \stackrel{P}\rightarrow 0,
$$
since $f_5 '(\xi_{abc}) \stackrel{P}\rightarrow - \frac{1}{4}$ and by \eqref{e:lambda1-lambdaE1=sum_second_term_limit}. As for the first term in the sum \eqref{e:lambda2-lambda2E=sum}, by the mean value theorem we can rewrite it with probability going to 1 as
$$
\log\{O_P(a(\nu)^{h_{1}-h_{2}})+ (p^{2}_{12}+p^2_{22})\widehat{b}_{22}\} - \log\{O(a(\nu)^{h_{1}-h_{2}})+ (p^{2}_{12}+p^2_{22})b_{22}\}
$$
$$
= \frac{1}{\xi_{b_{22}}} \Big\{O_P\Big(\frac{a(\nu)^{h_{1}-h_{2}}}{\sqrt{K_{a,j}}}\Big)+ (p^{2}_{12}+p^2_{22})(\widehat{b}_{22}- b_{22})\Big\},
$$
where $\xi_{b_{22}} \stackrel{P}\rightarrow (p^{2}_{12}+p^2_{22})b_{22}$. As a consequence, and reintroducing $j$,
\begin{equation}\label{e:lambda2_asympt_equiv}
\sqrt{K_{a,j}}\{ \log(\lambda_2(a(\nu)2^{j})) - \log(\lambda^E_2(a(\nu)2^{j})) \} \stackrel{P}\sim \sqrt{K_{a,j}} \Big(\frac{\widehat{b}_{22}(2^j)- b_{22}(2^j)}{b_{22}(2^j)} \Big).
\end{equation}
The expressions \eqref{e:lambda1_asympt_equiv} and \eqref{e:lambda2_asympt_equiv} for $j = j_1 < \hdots < j_m$ imply the weak limit \eqref{e:claim_theo_asympt_normality_lambda1} with asymptotic covariance matrix \eqref{e:Sigmah1h2}. $\Box$\\
\end{proof}

\begin{remark}\label{r:motivation}
At this point, we can shed more light on the apparent paradox anticipated in the Introduction: the asymptotic behavior of each entry of the matrices $W(a(\nu)2^j)$ and $EW(a(\nu)2^j)$ (see \eqref{e:W(a(nu)2^j)} and \eqref{e:EW(a(nu)2^j)}) is generally governed by the power law $a(\nu)^{2h_2}$, whereas that of the eigenvalues $\lambda_1(a(\nu)2^{j})$, $\lambda^E_1(a(\nu)2^j)$ is dominated by $a(\nu)^{2h_1}$. For mathematical simplicity, consider only the matrix $EW(a(\nu)2^j)$ under the instance $A = P$ in \eqref{e:OFBM_harmonizable}, where the columns of $P$ are unit vectors. The operator scaling property $(P5)$ of the wavelet transform yields
\begin{equation}\label{e:EW(j)=simplified_scaling}
EW(a(\nu)2^j) = P \textnormal{diag}(w_1 a(\nu)^{2h_1},w_2 a(\nu)^{2h_2}) P^*,
\end{equation}
where $EW(2^j) = P \textnormal{diag}(w_1,w_2) P^*$, for $w_1 = w_1(2^j) > 0$, $w_2 = w_2(2^j) > 0$. Then,
\begin{equation}\label{e:Epower_law}
\lambda^{E}_1(a(\nu)2^j) = \inf_{\textbf{u} \in S^1} \textbf{u}^* EW(a(\nu)2^j) \textbf{u} \leq \textbf{u}^*_0 \hspace{1mm}EW(a(\nu)2^j) \hspace{1mm}\textbf{u}_0 = \frac{w_1 a(\nu)^{2h_1} }{\|\textbf{u}_0 \|^2},
\end{equation}
where $\textbf{u}_0 = (P^*)^{-1}e_1/\|(P^*)^{-1}e_1 \|$ and $e_1$ is the first Euclidean vector. In other words, $\lambda^{E}_1(a(\nu)2^j)$ cannot go to infinity faster than $a(\nu)^{2h_1}$.

Notwithstanding \eqref{e:Epower_law}, it should be stressed that $\lambda^{E}_1(a(\nu)2^j)$ does not generally follow an exact power law. This can be seen based on the explicit expression for $\lambda^{E}_1(a(\nu)2^j)$ (see \eqref{e:lambda1_lambda2_explicit}), i.e.,
\begin{equation}\label{e:lambda1(j)_explicit}
\lambda^{E}_1(a(\nu)2^j) = \frac{1}{2} \{EW(a(\nu)2^j)_{11} + EW(a(\nu)2^j)_{22} - \sqrt{\Delta(a(\nu)2^j)}\}
\end{equation}
where $\Delta(a(\nu)2^j) = (EW(a(\nu)2^j)_{11}- EW(a(\nu)2^j)_{22})^2 + 4 (EW(a(\nu)2^j)_{12})^2$. Under \eqref{e:EW(j)=simplified_scaling}, the matrix $EW(a(\nu)2^j) = \Big(EW(a(\nu)2^j)_{i_1,i_2} \Big)_{i_1,i_2=1,2}$ can be expressed entry-wise as in \eqref{e:S=a(nu)HEW(j)a(nu)H*=(a_b_b_c)} with $b_{11}(2^j) = w_1$, $b_{12}(2^j) = 0$ and $b_{22}(2^j) = w_2$. It is of interest to note that the quality of the approximation $\lambda^{E}_1(a(\nu)2^j) \sim C a(\nu)^{2h_1}$ increases with the difference $h_2 - h_1$. Indeed, by \eqref{e:lambda1_explicit} and a second order Taylor expansion of the function $f(x) = \sqrt{x}$ around 1,
$$
\lambda^{E}_1(a(\nu)2^j) = \frac{2 \det EW(2^j)}{ w_{1}a(\nu)^{2(h_1 - h_2)} + w_{2}}a(\nu)^{2h_1}
$$
\begin{equation}\label{e:expansion_lambdaE1_secondorder_Taylor}
\Big\{ \frac{1}{2} + \frac{1}{2} \frac{\det EW(2^j)}{[w_{1} a(\nu)^{2(h_1 - h_2)}+ w_{2}]^2} \frac{1}{a(\nu)^{2(h_2 - h_1)}} + o\Big(\frac{1}{a(\nu)^{2(h_2 - h_1)}}\Big)\Big\}.
\end{equation}
The discussion and conclusions presented in this remark can be easily generalized beyond the instance \eqref{e:EW(j)=simplified_scaling}.
\end{remark}

\begin{remark}\label{r:choice_of_a(nu)}
In practice, the choice of $a(\nu)$ involves a statistical compromise. A large value of $a(\nu)$ with respect to $\nu$ implies that the estimator variance is relatively large and the estimator distribution is not very close to Gaussian, but it also yields a relatively small bias. Computational experiments suggest that the ratio $\nu/a(\nu)2^j$ should be no less than $2^3$.
\end{remark}


\subsection{The weak limit of unit eigenvectors}

For given $j, \log_2 a(\nu) \in \bbN$, consider the relation \eqref{e:relation_v1_v2} for the eigenvector entries
\begin{equation}\label{e:eigenvec_v(nu)}
\widehat{\textbf{v}}(a(\nu)2^j) = (\widehat{v}_1(a(\nu)2^j),\widehat{v}_2(a(\nu)2^j))^* \in \bbR^2
\end{equation}
associated with the smallest eigenvalue $\lambda_1 := \lambda_1(a(\nu)2^j)$ of the (symmetric) sample wavelet variance matrix $\widehat{S}:= W(a(\nu)2^j)$. Further consider their wavelet variance counterparts $(v_1(a(\nu)2^j),v_2(a(\nu)2^j))^*$, $\lambda^E_1 := \lambda^E_1(a(\nu)2^j)$ and $S:=EW(a(\nu)2^j)$. The ratios $(\widehat{v}_2/\widehat{v}_1)(a(\nu)2^j)$, $(v_2/v_1)(a(\nu)2^j)$ represent the tangents of the angles that determine the associated eigenspaces. In the next definition, we propose to use $(\widehat{v}_2/\widehat{v}_1)(a(\nu)2^j)$ as an estimator of the tangent of the angle determined by the entries of $P$ associated with $\lambda_1$ when $P \in O(2)$, i.e., $\tan(\theta) = -p_{12}/p_{22}$. For notational simplicity, we will just write
\begin{equation}\label{e:theta=-p12/p22}
\theta = -\frac{p_{12}}{p_{22}}.
\end{equation}
\begin{definition}\label{def:theta_estimator}
Let $j, \log_2 a(\nu) \in \bbN$. Under the assumptions of Definition \ref{def:estimator}, we define the estimator of $\theta$ (see \eqref{e:theta=-p12/p22}) at scale $a(\nu)2^{j}$ and its wavelet spectrum counterpart as
$$
\widehat{\theta}(a(\nu)2^{j}) = \Big(\frac{\widehat{v}_2}{\widehat{v}_1}\Big)(a(\nu)2^{j}) = \frac{\lambda_{1}(a(\nu)2^j) - \widehat{a}_{j,a(\nu)}}{\widehat{b}_{j,a(\nu)}},
$$
\begin{equation}\label{e:angle_estimator}
\theta(a(\nu)2^{j}) = \Big(\frac{v_2}{v_1}\Big)(a(\nu)2^{j}) = \frac{\lambda^E_{1}(a(\nu)2^j) - a_{j,a(\nu)}}{b_{j,a(\nu)}}
\end{equation}
where the definition of each individual term on the right-hand of the expressions in \eqref{e:angle_estimator} is given by \eqref{e:lambda1_lambda2}, \eqref{e:lambda^E_1_lambda^E_2}, \eqref{e:W(a(nu)2^j)} and \eqref{e:EW(a(nu)2^j)}.
\end{definition}

The consistency and asymptotic normality of the estimator $\widehat{\theta}(a(\nu)2^{j})$ in \eqref{e:angle_estimator} are precisely stated and shown in Theorem \ref{t:weak_limit_eigenvector} below. The limits themselves do not depend on the orthogonality of $P$. Note that the parametric assumption \eqref{e:p12_neq_0_or_p11=0} below rules out diagonal wavelet spectra, the latter being associated with entry-wise scaling instances of OFBM (as in \eqref{e:o.s.s._entrywise}). For further comments on the assumptions of Theorem \ref{t:weak_limit_eigenvector}, see Remark \ref{r:on_p12=0}.

\begin{theorem}\label{t:weak_limit_eigenvector}
Let $j \in \bbN$, and let $\widehat{\theta}(a(\nu)2^{j})$, $\theta(a(\nu)2^{j})$ be as in \eqref{e:angle_estimator}. Suppose the assumptions (OFBM1)--(4) and \eqref{e:a(nu)/J->infty} hold, as well as the parametric assumption $p_{22} \neq 0$.
\begin{itemize}
\item [($i$)] If
\begin{equation}\label{e:p12_neq_0_or_p11=0}
\textnormal{either $p_{12} \neq 0$ \hspace{2mm} or \hspace{2mm} $p_{12} = 0 \textnormal{ and }b_{12} \neq 0 $},
\end{equation}
then
\begin{equation}\label{e:thetahat_consistency}
\widehat{\theta}(a(\nu)2^j) \stackrel{P}\rightarrow \theta = - \frac{p_{12}}{p_{22}}, \quad \theta(a(\nu)2^j) \rightarrow \theta = - \frac{p_{12}}{p_{22}},\quad \nu \rightarrow \infty;
\end{equation}
\item [($ii$)] if
\begin{equation}\label{e:p12_neq_0_and_b12_neq_0}
p_{12} \neq 0 \textnormal{ and } b_{12} \neq 0,
\end{equation}
then,
\begin{equation}\label{e:thetahat_weak_limit}
a(\nu)^{h_2-h_1}\sqrt{K_{a,j}} \{ \widehat{\theta}(a(\nu)2^j) - \theta(a(\nu)2^j) \} \stackrel{d}\rightarrow \hspace{0.5mm} N(0,\sigma^2_{\theta}), \quad \nu \rightarrow \infty.
\end{equation}
In \eqref{e:thetahat_weak_limit}, $\sigma^2_{\theta} = {\mathcal R}^*\Sigma_{B}(2^j){\mathcal R}$, where $\Sigma_{B}(2^j)$ is the $3 \times 3$ block, associated with $j$, on the main diagonal of $\Sigma_{B}(j_1,\hdots,j_m)$ (see \eqref{e:asymptotic_normality_pseudoestimator_Bhat}), and ${\mathcal R}^* = \frac{\det P}{b_{22}p^{2}_{22}}\Big(0, -1 , \frac{b_{12}}{b_{22}} \Big)$.
\end{itemize}
\end{theorem}
\begin{proof}
As in the proof of Theorem \ref{t:asympt_normality_lambda1}, we will drop the subscripts $j$ and $\nu$ whenever convenient. To show ($i$) for $\theta(a(\nu)2^j)$, first assume that $p_{12} \neq 0$. The property $(P7)$ ensures that $b_{22} \neq 0$, whence $b = b_{j,a(\nu)} \neq 0$ for large $\nu$ in expression \eqref{e:S=a(nu)HEW(j)a(nu)H*=(a_b_b_c)}. The same applies to $a = a_{j,a(\nu)}$. Therefore, $v_1 \neq 0$ in \eqref{e:relation_v1_v2}, and
\begin{equation}\label{e:v2/v1_limit}
\frac{v_2}{v_1} = \frac{a}{b} \Big( \frac{\lambda^E_1}{a} - 1\Big) \sim - \frac{p_{12}}{p_{22}}, \quad \nu \rightarrow \infty,
\end{equation}
by Lemma \ref{l:elementary_results}.

Now assume that $p_{12} = 0$. Then, $p_{11} \neq 0$, and $a = a_{j,a(\nu)} \neq 0 $ for large $\nu$. Since $b_{12} \neq 0$ by assumption, then $b = b_{j,a(\nu)} \neq 0$. Again, the limit \eqref{e:v2/v1_limit} follows.

In regard to $\widehat{\theta}(a(\nu)2^j)$, in either case in \eqref{e:p12_neq_0_or_p11=0} the relations above can be simply rewritten for $\widehat{S} = W(a(\nu)2^{j})$ with eigenvector $\widehat{{\mathbf v}}$. Because of \eqref{e:asymptotic_normality_pseudoestimator_Bhat}, $\widehat{B}_a(2^j) \stackrel{P}\rightarrow B(2^j)$ and thus the expression $\widehat{v}_2/\widehat{v}_1$ is well-defined with probability increasing to 1, as $\nu \rightarrow \infty$. Again by Lemma \ref{l:elementary_results}, $\widehat{v}_2/\widehat{v}_1 \stackrel{P}\rightarrow - \frac{p_{12}}{p_{22}}$.\\

For the ensuing developments, recall that $b^{-1}_{22}$ is well-defined. So, to show ($ii$), reexpress
$$
a(\nu)^{h_2-h_1}\sqrt{K_{a,j}} \Big[ \Big( \frac{\lambda_{1}- \widehat{a}}{\widehat{b}}\Big) - \Big( \frac{\lambda^E_{1}- a}{b}\Big)\Big]
$$
\begin{equation}\label{e:1+2}
= a(\nu)^{h_2-h_1}\sqrt{K_{a,j}} \Big\{\frac{\lambda_1}{\widehat{b}} - \frac{\lambda^E_1}{b}\Big\} + a(\nu)^{h_2-h_1}\sqrt{K_{a,j}} \Big\{ \frac{a}{b}- \frac{\widehat{a}}{\widehat{b}}\Big\}.
\end{equation}
The first term on the right-hand side of \eqref{e:1+2} can be further developed into
\begin{equation}\label{e:1a+1b}
a(\nu)^{h_2-h_1}\frac{1}{\widehat{b}} \sqrt{K_{a,j}} \{ \lambda_1 - \lambda^E_1\} + a(\nu)^{h_2-h_1}(-1) \frac{\lambda^E_1}{b} \sqrt{K_{a,j}}\Big\{ \frac{\widehat{b} - b}{\widehat{b}}\Big\}.
\end{equation}
In view of Lemmas \ref{l:elementary_results} and \ref{l:vecB^a(2^j)_asympt},
$$
\frac{\lambda^E_1}{b} \sim \frac{1}{a(\nu)^{2(h_2 - h_1)}} \frac{\det EW(2^j)}{p_{12}p_{22}b^2_{22}(p^{2}_{12} + p^{2}_{22})},
\quad \sqrt{K_{a,j}} \Big\{ \frac{\widehat{b}- b}{\widehat{b}}\Big\} \stackrel{P}\sim \sqrt{K_{a,j}}\frac{\{\widehat{b}_{22} - b_{22}\}}{\widehat{b}_{22}} \stackrel{d}\rightarrow \frac{N(0,\sigma^2(b_{22}(2^j)))}{b_{22}(2^j)},
$$
where $\sigma^2(b_{22}(2^j))$ comes from the matrix $\Sigma_{B}(j_1,\hdots,j_m)$ in \eqref{e:asymptotic_normality_pseudoestimator_Bhat}.
Consequently, in regard to the second term in the sum \eqref{e:1a+1b},
\begin{equation}\label{e:1a_second_term}
a(\nu)^{h_2-h_1}(-1) \frac{\lambda^E_1}{b} \sqrt{K_{a,j}}\Big\{ \frac{\widehat{b} - b}{\widehat{b}}\Big\} \stackrel{P}\rightarrow 0.
\end{equation}
We now look at the first term in the sum \eqref{e:1a+1b}. Let $\widehat{r} = 4\frac{\widehat{a}\widehat{c}-\widehat{b}^2}{(\widehat{a}+\widehat{c})^2}$ and $r = 4\frac{ac-b^2}{(a+c)^2}$. Then,
$$
\sqrt{K_{a,j}}\{\lambda_{1}-\lambda^E_1\} = \sqrt{K_{a,j}}\Big\{2 \Big( \frac{\widehat{a}\widehat{c}-\widehat{b}^2}{\widehat{a}+\widehat{c}}\Big)-2 \Big( \frac{ac-b^2}{a+c}\Big)\Big\}\Big(\frac{\sqrt{1 - \widehat{r}} - 1}{-\widehat{r}}\Big)
$$
\begin{equation}\label{e:sqrtKaj(lambda1-lambdaE1)}
+ 2\Big(\frac{ac-b^2}{a+c}\Big)\sqrt{K_{a,j}}\Big\{\Big(\frac{\sqrt{1 - \widehat{r}} - 1}{-\widehat{r}}\Big) - \Big(\frac{\sqrt{1 - r} - 1}{-r}\Big)\Big\}.
\end{equation}
By the same reasoning leading to \eqref{e:lambda1-lambdaE1=sum_second_term_limit},
$$
\sqrt{K_{a,j}}\Big\{2\Big( \frac{\widehat{a}\widehat{c}-\widehat{b}^2}{\widehat{a}+\widehat{c}}\Big)- 2\Big( \frac{ac-b^2}{a+c}\Big)\Big\}
= a(\nu)^{2h_1}\Big\{\frac{O_{P}(1)}{O_{P}(a(\nu)^{h_1 - h_2})+ (p^{2}_{12}+p^{2}_{22})\widehat{b}_{22}}
$$
\begin{equation}\label{e:sqrt(K)_(ac-b2/a+c-allhat)-ac-b2/a+c}
+ \det EW(2^j) \frac{O_P(1)}{[O_P(a(\nu)^{h_1-h_2}) + (p^{2}_{12}+p^{2}_{22})\widehat{b}_{22}][O(a(\nu)^{h_1-h_2}) + (p^{2}_{12}+p^{2}_{22})b_{22}]}\Big\}.
\end{equation}
Moreover, for $f_1(x)$ as in \eqref{e:f1_f2}, with probability going to 1 the mean value theorem gives $\sqrt{K_{a,j}}\{f_1(\widehat{r}) - f_1(r)\} = f_1'(\xi_r) \sqrt{K_{a,j}}(\widehat{r} - r)$, where the random variable $\xi_r$ lies between $r$ and $\widehat{r}$. Thus, by \eqref{e:4(ac-b2)/a+c_4(ac-b2)/(a+c^2_p12_p22neq0}, \eqref{e:4(ac-b2)/a+c_4(ac-b2)/(a+c^2_p12_p22neq0_in_prob}, \eqref{e:f1'(x)->1/4} and \eqref{e:lambda1-lambdaE1=sum_second_term_limit},
\begin{equation}\label{e:sqrt(K)(sqrt(1-rhat)-1/-rhat)-(sqrt(1-r)-1/-r)}
f_1'(\xi_r) \stackrel{P}\rightarrow \frac{1}{8}, \quad \sqrt{K_{a,j}}(\widehat{r} - r) \stackrel{P}\sim \frac{O_P(1)}{a(\nu)^{2(h_2-h_1)}} .
\end{equation}
Therefore, by \eqref{e:sqrt(K)_(ac-b2/a+c-allhat)-ac-b2/a+c}, \eqref{e:sqrt(K)(sqrt(1-rhat)-1/-rhat)-(sqrt(1-r)-1/-r)}, \eqref{e:4(ac-b2)/a+c_4(ac-b2)/(a+c^2_p12_p22neq0} and \eqref{e:4(ac-b2)/a+c_4(ac-b2)/(a+c^2_p12_p22neq0_in_prob}, the first term in \eqref{e:1a+1b} is asymptotic equivalent in probability to
\begin{equation}\label{e:1a_first_term}
a(\nu)^{h_2-h_1}\Big(\frac{1}{[O_{P}(a(\nu)^{h_1-h_2})+ p_{12}p_{22}\widehat{b}_{22}]} \frac{1}{a(\nu)^{2h_2}} \Big)\Big\{a(\nu)^{2h_1}O_P(1)+ \frac{O_P(1)}{a(\nu)^{2(h_2 - h_1)}}\Big\}  \stackrel{P}\rightarrow 0.
\end{equation}
As a consequence, in regard to the first term on the right-hand side of \eqref{e:1+2},
\begin{equation}\label{e:1a+1b_limit}
a(\nu)^{h_2 - h_1}\sqrt{K_{a,j}} \Big\{\frac{\lambda_1}{\widehat{b}} - \frac{\lambda^E_1}{b}\Big\} \stackrel{P}\rightarrow 0.
\end{equation}

We now turn to the second term on the right-hand side of \eqref{e:1+2}. Note that
\begin{equation}\label{e:a/b-ahat/bhat}
\frac{a}{b} -\frac{\widehat{a}}{\widehat{b}}= \frac{a(\widehat{b} - b) - b(\widehat{a}-a)}{b \hspace{0.5mm}\widehat{b}} \stackrel{P}\sim a(\nu)^{-2h_{2}}\frac{a(\widehat{b} - b) - b(\widehat{a}-a)}{p^2_{12}p^2_{22}b^2_{22}(2^j)}.
\end{equation}
Note that the ratio on the right-hand side of \eqref{e:a/b-ahat/bhat} is well-defined, since $p_{12}, p_{22} \neq 0$, by assumption. Recall the expressions \eqref{e:Shat=a(nu)H_EW(j)_a(nu)H*=(a_b_b_c)} and \eqref{e:S=a(nu)HEW(j)a(nu)H*=(a_b_b_c)} and denote $a(\nu)^{-2 h_2}a$, $a(\nu)^{-2 h_2}b$, $a(\nu)^{-2 h_2}\widehat{a}$ and $a(\nu)^{-2 h_2} \widehat{b}$, respectively, by
$$
O(a(\nu)^{2(h_1 - h_2)}) + \alpha_{12}a(\nu)^{h_1 - h_2} + \alpha_{22}, \quad O(a(\nu)^{2(h_1 - h_2)}) + \beta_{12}a(\nu)^{h_1-h_2} + \beta_{22},
$$
\begin{equation}\label{e:a,ahat,b,bhat}
O_P(a(\nu)^{2(h_1 - h_2)}) + \widehat{\alpha}_{12}a(\nu)^{h_1 - h_2} + \widehat{\alpha}_{22} \quad \textnormal{and}\quad O_P(a(\nu)^{2(h_1 - h_2)}) + \widehat{\beta}_{12}a(\nu)^{h_1-h_2} + \widehat{\beta}_{22}.
\end{equation}
For the sake of mathematical clarity, we will henceforth omit the terms of order $a(\nu)^{2(h_1 - h_2)}$ in the expressions \eqref{e:a,ahat,b,bhat}, since the ensuing argument can be easily adapted to account for them. So, based on \eqref{e:a,ahat,b,bhat}, we can rewrite the numerator in \eqref{e:a/b-ahat/bhat} as
$$
(\alpha_{12}a(\nu)^{h_1-h_2} + \alpha_{22}) \{(\widehat{\beta}_{12} - \beta_{12})a(\nu)^{h_1 - h_2} + (\widehat{\beta}_{22} - \beta_{22})\}
$$
$$
- (\beta_{12}a(\nu)^{h_1-h_2} + \beta_{22}) \{(\widehat{\alpha}_{12} - \alpha_{12})a(\nu)^{h_1 - h_2} + (\widehat{\alpha}_{22} - \alpha_{22})\}
$$
$$
= a(\nu)^{h_1-h_2}\{(\alpha_{12}a(\nu)^{h_1 - h_2} + \alpha_{22})(\widehat{\beta}_{12} - \beta_{12}) - (\beta_{12}a(\nu)^{h_1 - h_2} + \beta_{22})(\widehat{\alpha}_{12} - \alpha_{12})\}
$$
\begin{equation}\label{e:a(bhat-b)+b(a-ahat)}
+ a(\nu)^{h_1-h_2}\{\alpha_{12}(\widehat{\beta}_{22} - \beta_{22}) - \beta_{12}(\widehat{\alpha}_{22} - \alpha_{22})\},
\end{equation}
where the equality is a consequence of the fact that $\alpha_{22}(\widehat{\beta}_{22} - \beta_{22}) - \beta_{22}(\widehat{\alpha}_{22}- \alpha_{22}) = 0$.
When premultiplied by $a(\nu)^{h_2 - h_1}\sqrt{K_{a,j}}$, the first term in the sum \eqref{e:a(bhat-b)+b(a-ahat)} is asymptotically equivalent in probability to
\begin{equation}\label{e:N(0,sigma2(b12))_limit}
\sqrt{K_{a,j}}\{\alpha_{22}(\widehat{\beta}_{12} - \beta_{12}) - \beta_{22}(\widehat{\alpha}_{12} - \alpha_{12})\} \stackrel{d}\rightarrow - p^2_{12}b_{22}(2^j)\det(P) \hspace{1mm}N(0,\sigma^2(b_{12}(2^j))),
\end{equation}
where $\sigma^2(b_{12}(2^j))$ is taken from the matrix $\Sigma_B(j_1,\hdots,j_m)$ in \eqref{e:asymptotic_normality_pseudoestimator_Bhat}.
In turn, when premultiplied by $a(\nu)^{h_2 - h_1}\sqrt{K_{a,j}}$, the second term in the sum \eqref{e:a(bhat-b)+b(a-ahat)} is asymptotically equivalent in probability to
\begin{equation}\label{e:N(0,sigma2(b22))_limit}
\sqrt{K_{a,j}}\{\alpha_{12}(\widehat{\beta}_{22} - \beta_{22}) - \beta_{12}(\widehat{\alpha}_{22} - \alpha_{22})\}
\stackrel{d}\rightarrow  p^2_{12}b_{12}(2^j)\det(P) \hspace{1mm}N(0,\sigma^2(b_{22}(2^j))).
\end{equation}
In view of \eqref{e:1a+1b_limit}, and assuming $p_{22} > 0$, the relation \eqref{e:thetahat_weak_limit} is obtained by dividing the sum of the weak limits \eqref{e:N(0,sigma2(b12))_limit} and \eqref{e:N(0,sigma2(b22))_limit} by the denominator on the right-hand side of \eqref{e:a/b-ahat/bhat}. $\Box$\\
\end{proof}

\begin{remark}\label{r:ceigenv}
From a different but mathematically equivalent perspective, the statement \eqref{e:thetahat_consistency} implies that $W(a(\nu)2^j)$ and $E W(a(\nu)2^j)$ have sequences of unit eigenvectors associated with $\lambda_1(a(\nu)2^j)$ and $\lambda^E_1(a(\nu)2^j)$, respectively, which converge (in probability and deterministically, respectively) to the limiting unit vector
\begin{equation}\label{e:thetahat_consistency_in_terms_of_p}
\Big( \frac{|p_{22}|}{\sqrt{p^{2}_{12} + p^{2}_{22}}}, \frac{- p_{12} \hspace{1mm}\textnormal{sign}(p_{22})}{\sqrt{p^{2}_{12} + p^{2}_{22}}}\Big).
\end{equation}
If $P \in O(2)$, then eigenvectors of $H$ in two orthogonal directions can be consistently estimated by the eigenvectors of $W(a(\nu)2^j)$. In view of \eqref{e:thetahat_consistency_in_terms_of_p}, there is also a unit eigenvector associated with $\lambda_{2}(a(\nu)2^j)$ that converges to $(p_{21},-p_{11})^*$, since the latter is orthogonal to $(p_{22},-p_{12})^*$. It is easy to see that the latter vectors generate the same (eigen)spaces as $(p_{12},p_{22})^*$, $(p_{11},p_{21})^*$, respectively.

\end{remark}

\begin{remark}\label{r:on_p12=0}
In regard to the assumptions of Theorem \ref{t:weak_limit_eigenvector}, when $p_{22} = 0$, it is clear that \eqref{e:angle_estimator} cannot be consistent. By a similar proof to that of Theorem \ref{t:weak_limit_eigenvector}, asymptotic normality (with different variances) for $\{ \widehat{\theta}(a(\nu)2^j) - \theta(a(\nu)2^j)\}$ can also be established when $p_{12} = 0$ or $b_{12} = 0$. In the former case, for instance, the convergence rate is $a(\nu)^{h_2 - h_1}\sqrt{K_{a,j}}$ or $\sqrt{K_{a,j}}$, respectively, depending on whether $b_{12} \neq 0$ or both $b_{12} = 0$ and $p_{21} \neq 0$.
\end{remark}

\begin{remark}
In Theorem \ref{t:weak_limit_eigenvector}, the convergence rate is the non-standard $a(\nu)^{h_2-h_1}\sqrt{K_{a,j}} \sim a(\nu)^{h_2-h_1 -1/2}\sqrt{\nu /2^j} $, which depends on the parameters to be estimated $h_1$, $h_2$. In practice, since $a(\nu)$ is much slower than $\nu$ (see condition \eqref{e:a(nu)/J->infty}), its effect may not be noticeable (see Section \ref{s:simulation_studies} below).
\end{remark}

\begin{remark}
Substituting $\lambda_{2}(a(\nu)2^j)$ for $\lambda_{1}(a(\nu)2^j)$ in \eqref{e:angle_estimator} does not provide information on $p_{11}$ or $p_{21}$. Under mild assumptions on the parameters, Lemma \ref{l:elementary_results} yields
$$
\frac{a}{b}\Big(\frac{\lambda_{2}(a(\nu)2^{j})}{a} - 1 \Big) \stackrel{P}\rightarrow \frac{p_{22}}{p_{12}}.
$$
\end{remark}

\section{Simulation studies}
\label{s:simulation_studies}

In this section, we lay out and discuss a broad computational study of the performance of the proposed estimators. OFBM was simulated by means of the Hermir Toolbox, devised in Helgason et al.\ \cite{Helgason:Pipiras:Abry:2011a,Helgason:Pipiras:Abry:2011b} and available at {\tt www.hermir.org}. When describing the results, we drop the asymptotic scaling factor $a(\nu)$ used in Theorems \ref{t:asympt_normality_lambda1} and \ref{t:weak_limit_eigenvector} and only speak of shifting scales $2^j \in \bbN$.

The chosen sample path size was $N = 2^{16}$. The purpose of picking a large $N$ was to provide a compelling illustration of the estimators' ability to capture both Hurst eigenvalues. In biomedical applications, typical recordings can be much shorter (of the order $ N=2^{10}$; see, for instance, Ivanov \cite{ivanov:1999}). Nevertheless, sample paths of size $N = 2^{16}$ are, indeed, encountered in Internet traffic analysis (Abry et al.\ \cite{abfrv:2002}) and hydrodynamic turbulence (Frisch \cite{frisch:1995}).

In the wavelet analysis, we used least asymmetric orthogonal Daubechies wavelets with $N_\psi = 2$ (Daubechies \cite{daubechies:1992}, chapters 6--8). It was verified that similar conclusions can be obtained when using $N_\psi > 2$ or some other wavelet with a large enough number of vanishing moments.

\begin{remark}\label{r:discretized_wavelet_transform}
In practice, a continuous time OFBM sample path is not available and thus the theoretical wavelet coefficient $D(2^j,k)$ cannot be computed. Instead, one approximates the latter by means of the classical recursive (or pyramidal) discrete filter bank algorithm (Mallat \cite{mallat:1999}, chapter 7). The algorithm's main input is a discrete time sequence, which, following Veitch et al.\ \cite{VEITCH:2003:A}, can be a discrete OFBM sample $\{B_H(k)\}_{k \in T}$, $T \subseteq \bbZ$. In Section \ref{s:discretized_wavelet}, we lay out in detail the mathematical framework for estimation based on discretized wavelet coefficients. In the main result of the section, Theorem \ref{t:asympt_log_a(nu)_discrete}, we establish that, under mild conditions, the estimated Hurst eigenvalues and eigenvector stemming from the pyramidal algorithm also satisfy the weak limits \eqref{e:claim_theo_asympt_normality_lambda1} and \eqref{e:thetahat_weak_limit}, respectively.
\end{remark}

\subsection{Entry-wise vs eigenvalue-based estimation}\label{s:entry-wise_vs_eigenvalue}

A matrix $W(2^j)$ was computed from a single sample path (of size $N = 2^{16}$) of a synthetic OFBM with matrix parameters
$
P =  \left(\begin{array}{cc}
0.98 & 0.57 \\
0.20 & 0.82
\end{array}\right)$, $J_H =  \textnormal{diag}(0.25,0.85)$. In the spirit of the entry-wise approach, the wavelet-based estimation of the Hurst eigenvalues relies on performing a linear regression on a $\log_2 W(2^j) $ vs $ j = \log_2 2^j$ diagram, motivated by the log-transformed scalar version of $EW(2^j)$ as in $(P5)$, i.e., $ E W(2^j) = C(H) 2^{j 2H}$ for some $C(H) > 0$. This is shown for each auto- and cross-wavelet components of the bivariate OFBM in Figure \ref{fig:figa}. The top row displays, in order, plots for $\log_2 W(2^j)_{1,1} $ vs $ j = \log_2 2^j$, $\log_2 W(2^j)_{1,2} $ vs $ j = \log_2 2^j$ and $\log_2 W(2^j)_{2,2} $ vs $ j = \log_2 2^j$. The asymptotic behavior $j 2 h_2$ is superimposed on each of these plots. In view of the expression \eqref{e:Shat=a(nu)H_EW(j)_a(nu)H*=(a_b_b_c)}, it is unsurprising that at coarse scales all auto- and cross-components end up driven by the dominant Hurst eigenvalue $h_2$. In other words, the conspicuous prevalence of the latter precludes the estimation of the Hurst eigenvalue $h_1$. By contrast, Figure \ref{fig:figa}, bottom row, displays, in order, plots for $\log_2 \lambda_1(2^j) $ vs $ j = \log_2 2^j$ and $\log_2 \lambda_2(2^j) $ vs $ j = \log_2 2^j$, as well as the superimposed asymptotic behaviors $j 2 h_1$ and $j 2 h_2$, respectively. The eigenvalue-based procedure leads to two scaling laws individually driven by each of the Hurst eigenvalues $h_1$ and $h_2$.

\subsection{Estimation performance and asymptotic normality}

We also numerically synthesized $10,000$ bivariate OFBM sample paths to study the finite-sample effectiveness of the normal approximation described in Theorem \ref{t:asympt_normality_lambda1}. For each path, the estimates $W(2^j)$, $\lambda_1(2^j)$, $ \lambda_2(2^j)$ and $ \hat v_2(2^j)/\hat v_1(2^j)$ were computed. Averaging over realizations yields Monte Carlo estimates $ \hat E \lambda_1(2^j)  $, $ \hat E \lambda_2(2^j)  $ and $\hat E p_{12}(2^j)/p_{22}(2^j) $ of the ensemble averages $ E\lambda_1(2^j)$, $ E \lambda_2(2^j)$ and $- E \widehat{\theta}(2^j) $, together with estimates of the variances $\widehat{\textnormal{Var}} \lambda_1(2^j)$, $\widehat{\textnormal{Var}} \lambda_2(2^j) $, $\widehat{\textnormal{Var}} \hspace{0.5mm}\theta(2^j)$.

In Figures \ref{fig:figcc}, \ref{fig:figb} and \ref{fig:figc}, the simulated OFBM has parameter $ J_H =  \textnormal{diag}(0.25,0.85)$ and path size $N = 2^{16}$ for the three cases, but different mixing matrices of the general form
$P =  \left(\begin{array}{cc}
1/\sqrt{1+\gamma^2} & \beta/\sqrt{1+\beta^2} \\
\gamma/\sqrt{1+\gamma^2} & 1/\sqrt{1+\beta^2}
\end{array}\right) $.
In this parametrization, $\beta = -\theta$ (see \eqref{e:theta=-p12/p22}). The simulated instances are representative of different parametric settings. The choice $\beta=0.7$, $\gamma =0.2$ (Figure \ref{fig:figcc}) illustrates the situation of a general mixing matrix $P$. The choice $\beta= -\gamma$ and $\beta/\sqrt{1+\beta^2} = \sin \pi/6$  (Figure \ref{fig:figb}) represents the case $P \in O(2)$, whereas $\gamma = 0 $ and $\beta= 0.2$ (Figure \ref{fig:figc}) portrays the case often referred to as fractional cointegration. A comparison of the estimates (black solid lines with `o') with the theoretical values (red dashed lines) in Figures \ref{fig:figcc}, \ref{fig:figb}, and \ref{fig:figc} reveals the robustness of the proposed estimators (top row). The qq-plots (bottom row) also show that no deviation from a ${\cal N}(0,1)$ distribution can be observed within $\pm 2$ standard deviations for $\log_2 \lambda_1(2^j) $ and $\log_2 \lambda_1(2^j) $, and within $\pm 2$ standard deviations for $p_{12}/p_{22}$.

The simulations indicate that, beyond asymptotics, Theorem \ref{t:asympt_normality_lambda1} for $\widehat{h}_1(2^j) $ and $\widehat{h}_2(2^j) $, and Theorem \ref{t:weak_limit_eigenvector} for $\widehat{\theta}(2^j) = - \widehat{\beta}(2^j)$ provide effective normal approximations to the finite-sample estimator distributions. Whether or not $P \in O(2)$ does not impact the performance of $\log_2 \lambda_1 (2^j)/ 2 j$, $\log_2 \lambda_2(2^j) / 2 j $ and $\widehat{\beta}(2^j)$. However, assuming $P \in O(2)$ additionally allows for the full estimation of $P$, and thus of $H$, as discussed in Remark~\ref{r:ceigenv} and illustrated in Figure \ref{fig:figd}.

\begin{figure}[h]
\centerline{
\includegraphics[height=40truemm,keepaspectratio]{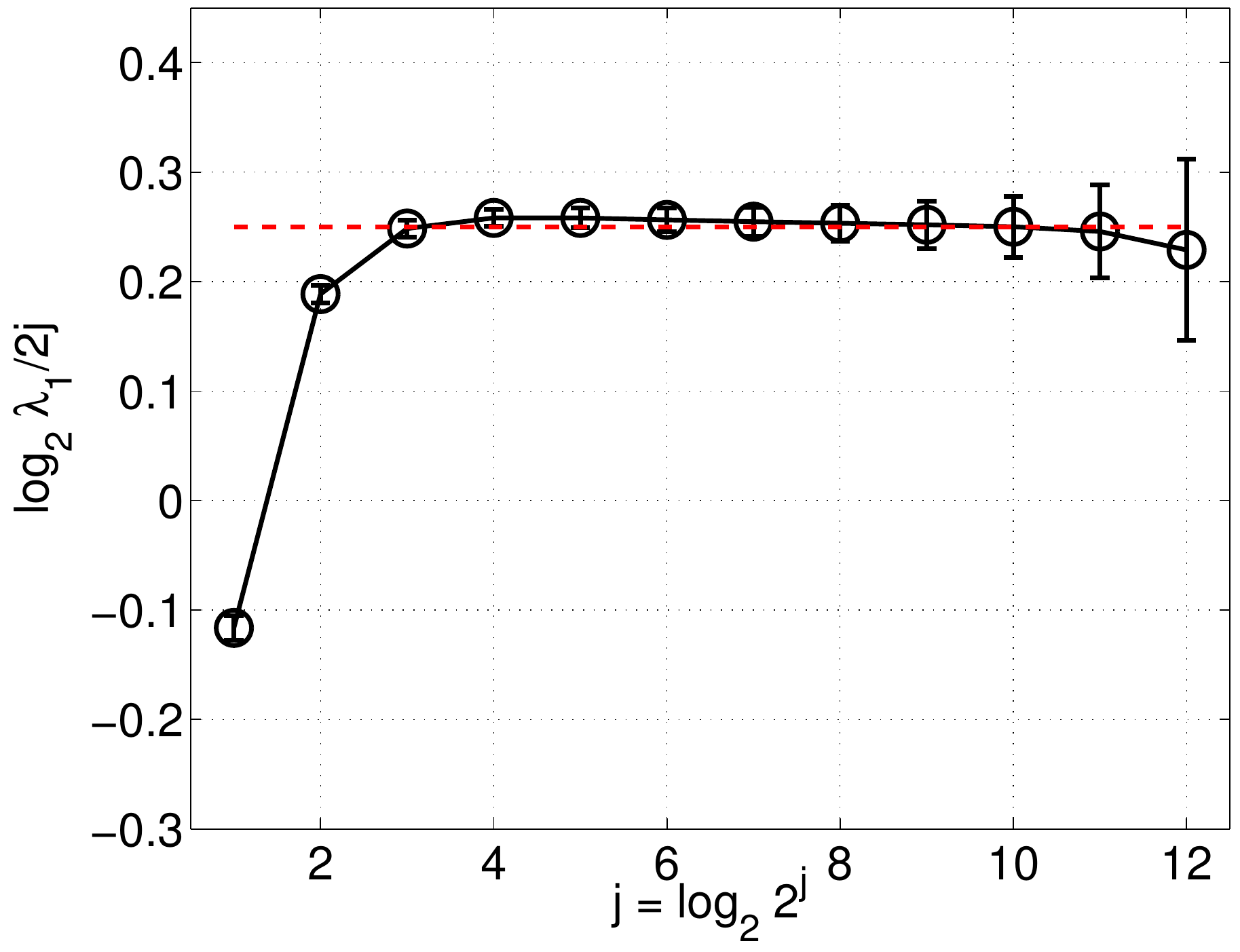}
\includegraphics[height=40truemm,keepaspectratio]{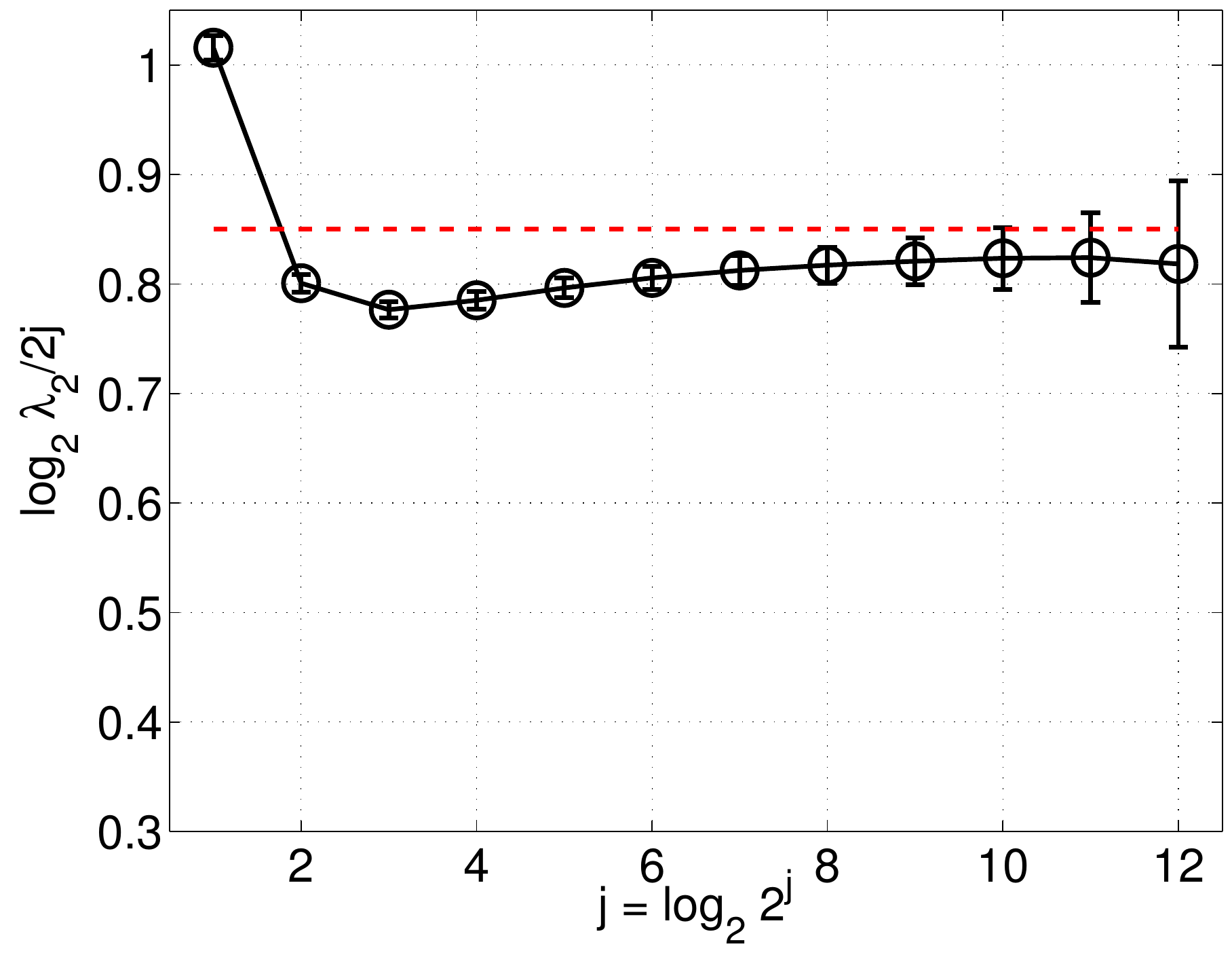}
\includegraphics[height=40truemm,keepaspectratio]{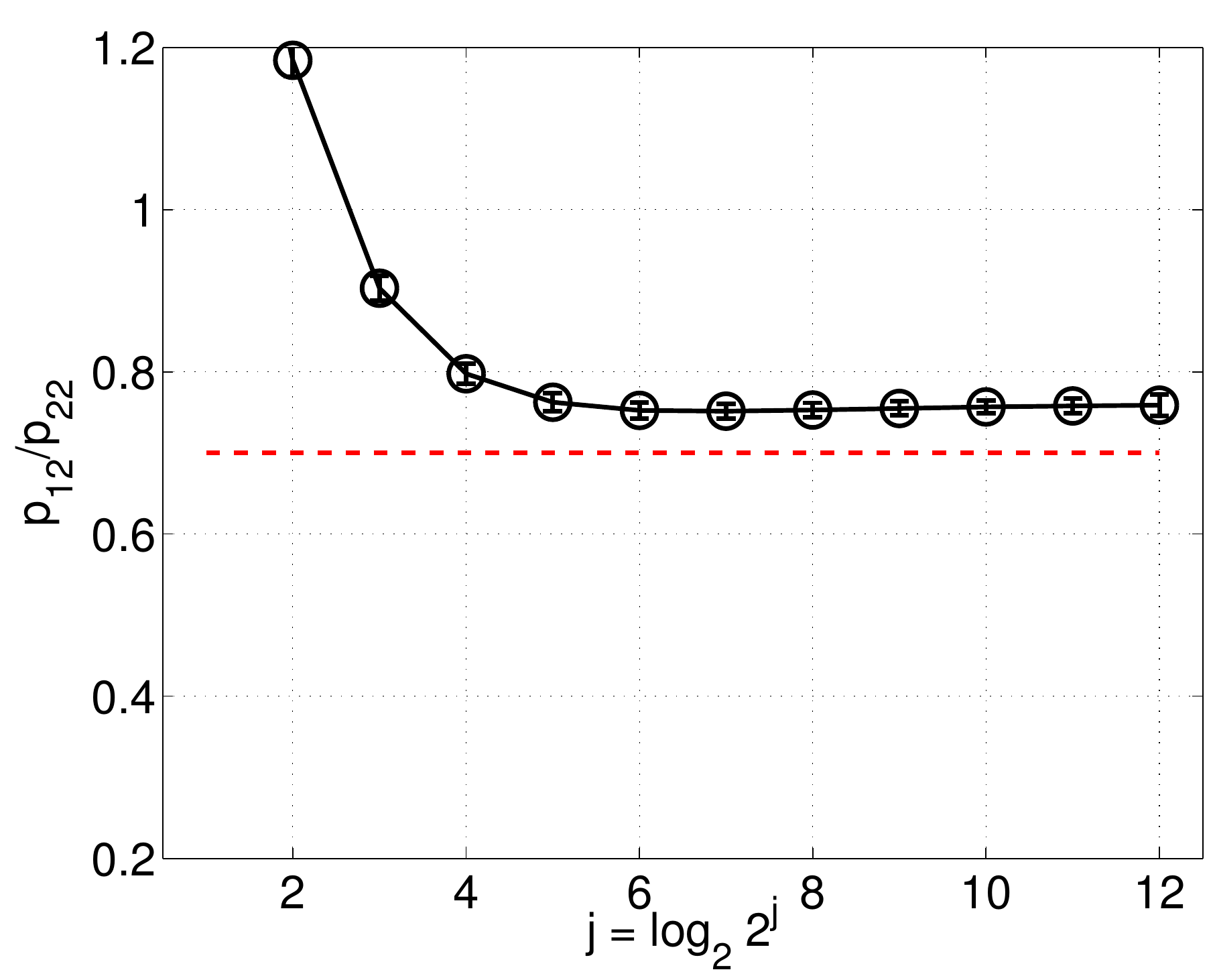}
}
\centerline{
\includegraphics[height=40truemm,keepaspectratio]{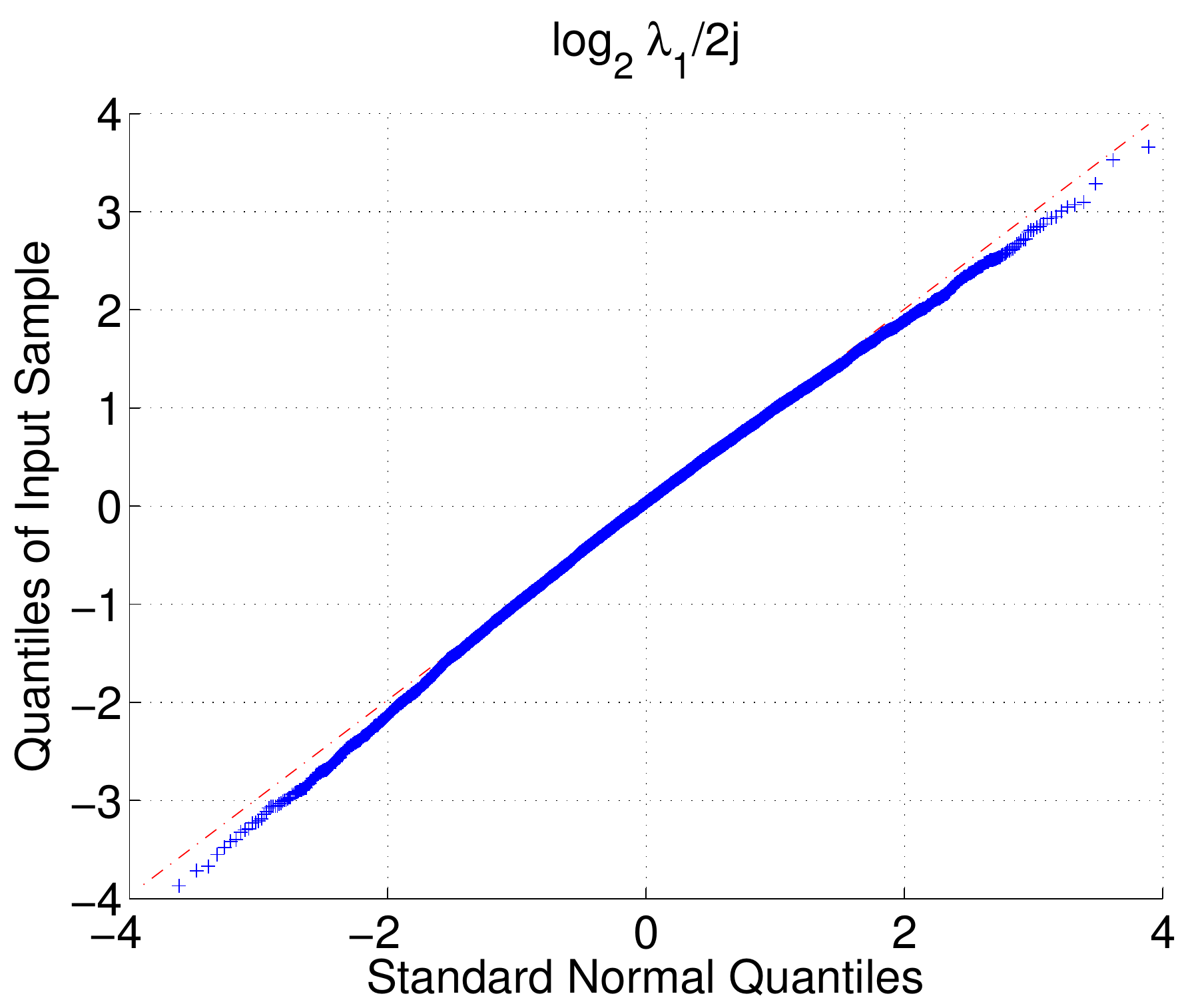} \hspace{3mm}
\includegraphics[height=40truemm,keepaspectratio]{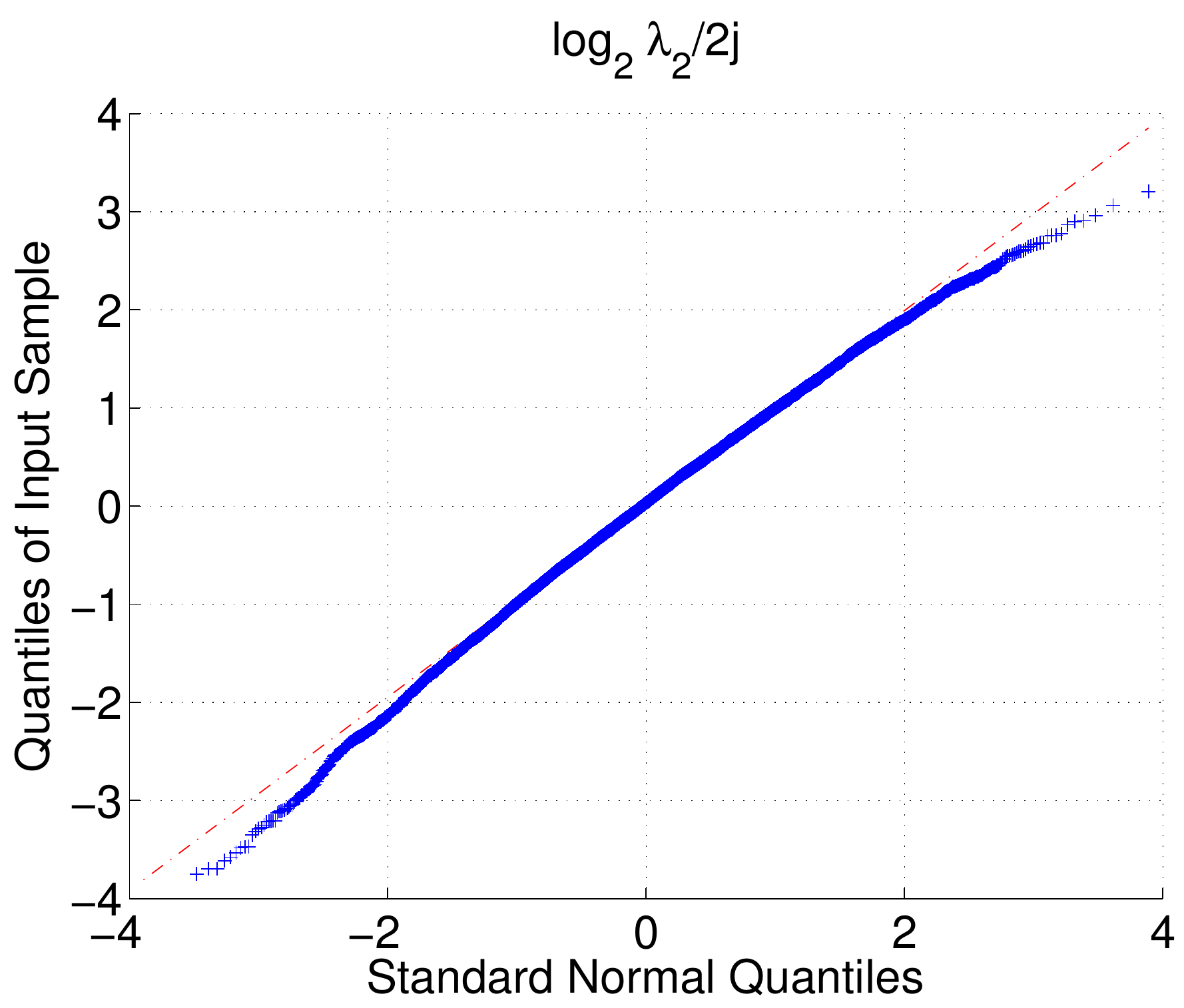} \hspace{3mm}
\includegraphics[height=40truemm,keepaspectratio]{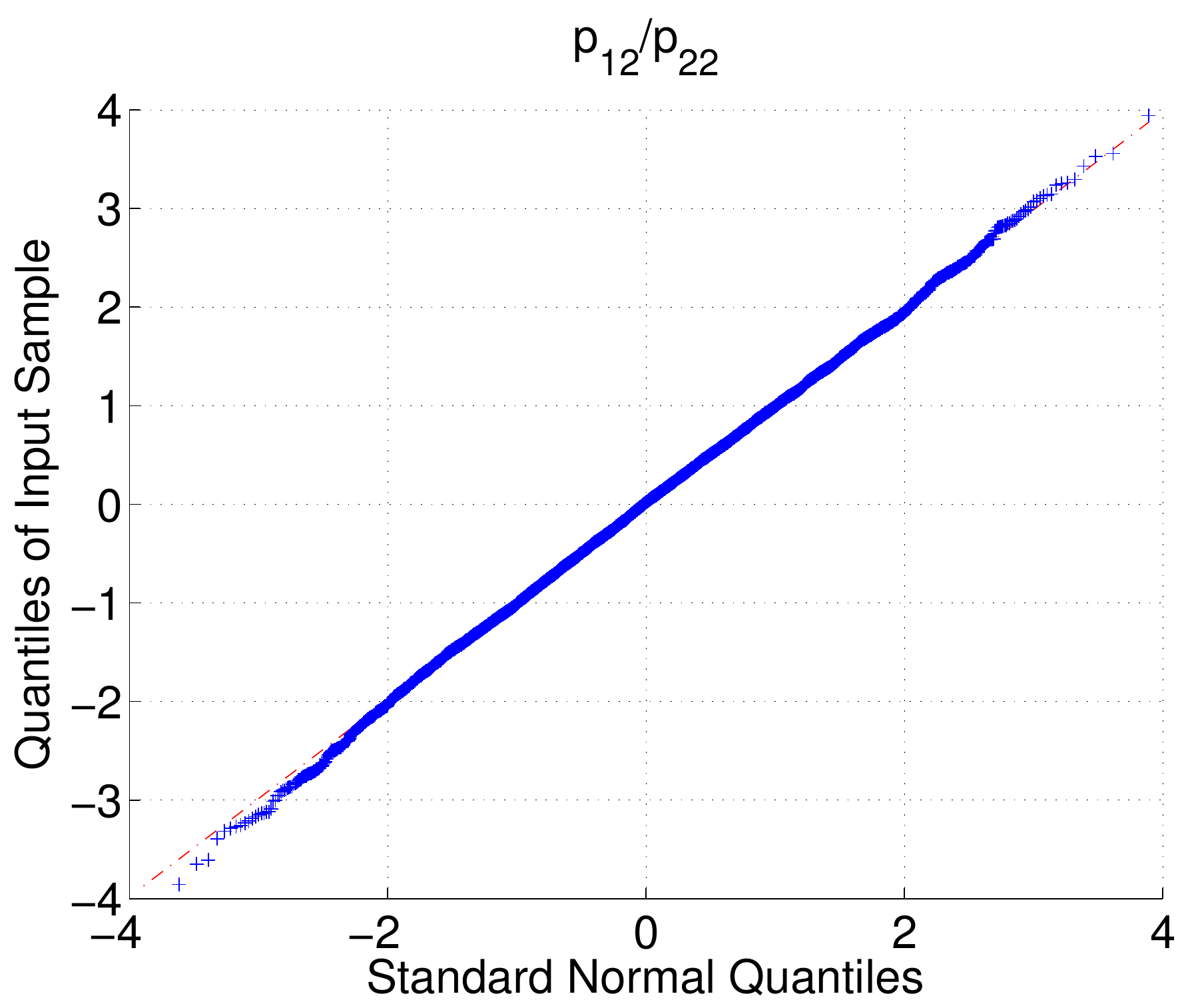}
}
\caption{\label{fig:figcc} {\bf  Estimation performance and asymptotic normality:} OFBM with $\gamma = 0.2 $ and $\beta= 0.7$. Top row: black solid lines with `o' show $ \hat E \widehat{\vartheta} \pm \sqrt{ \widehat{ \textnormal{ Var }}  \widehat{\vartheta} /n}$ for $\widehat{\vartheta} = \log_2 \lambda_1 (2^j)/ 2 j$,  $\log_2 \lambda_2(2^j) / 2 j$ and $p_{12}(2^j)/p_{22}(2^j)$ (target parameters: $h_1$ (left plots), $h_2$ (center plots), $p_{12}/p_{22}$ (right plots), respectively), red dashed lines corresponding to theoretical values. Bottom row, the corresponding qq-plots (against ${\cal N}(0,1)$ distributions) for $j=10$.
}
 \end{figure}

\begin{figure}[h]
\centerline{
\includegraphics[height=40truemm,keepaspectratio]{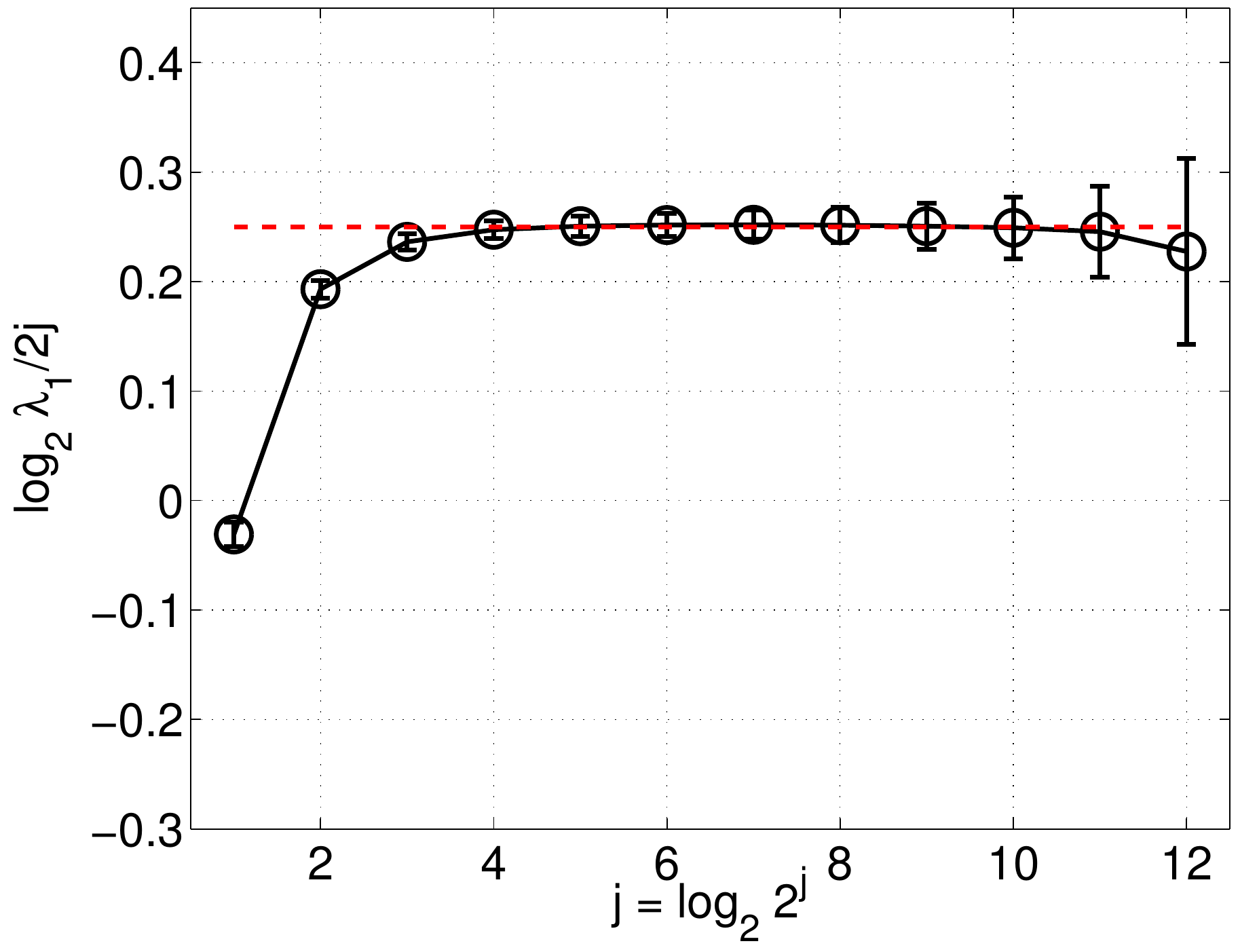}
\includegraphics[height=40truemm,keepaspectratio]{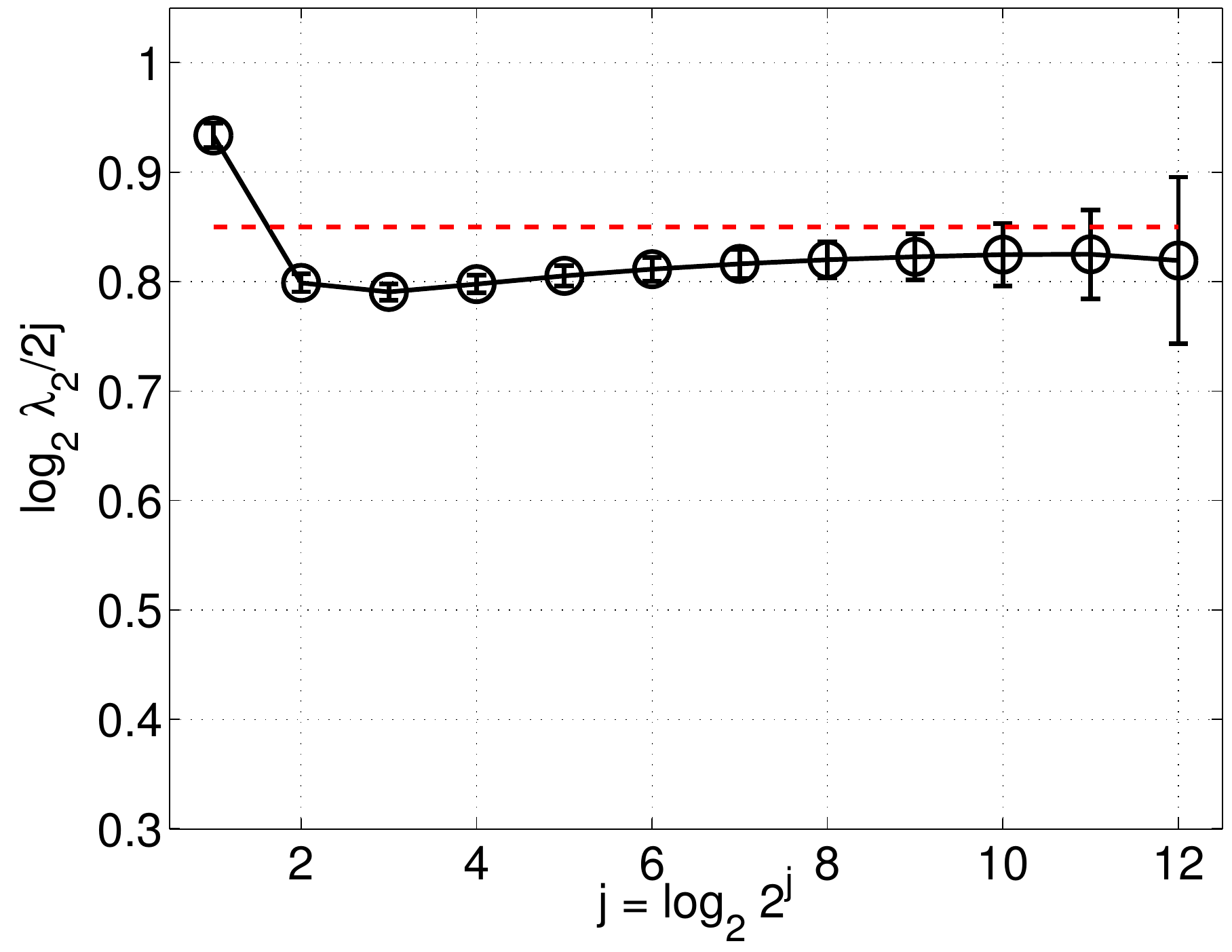}
\includegraphics[height=40truemm,keepaspectratio]{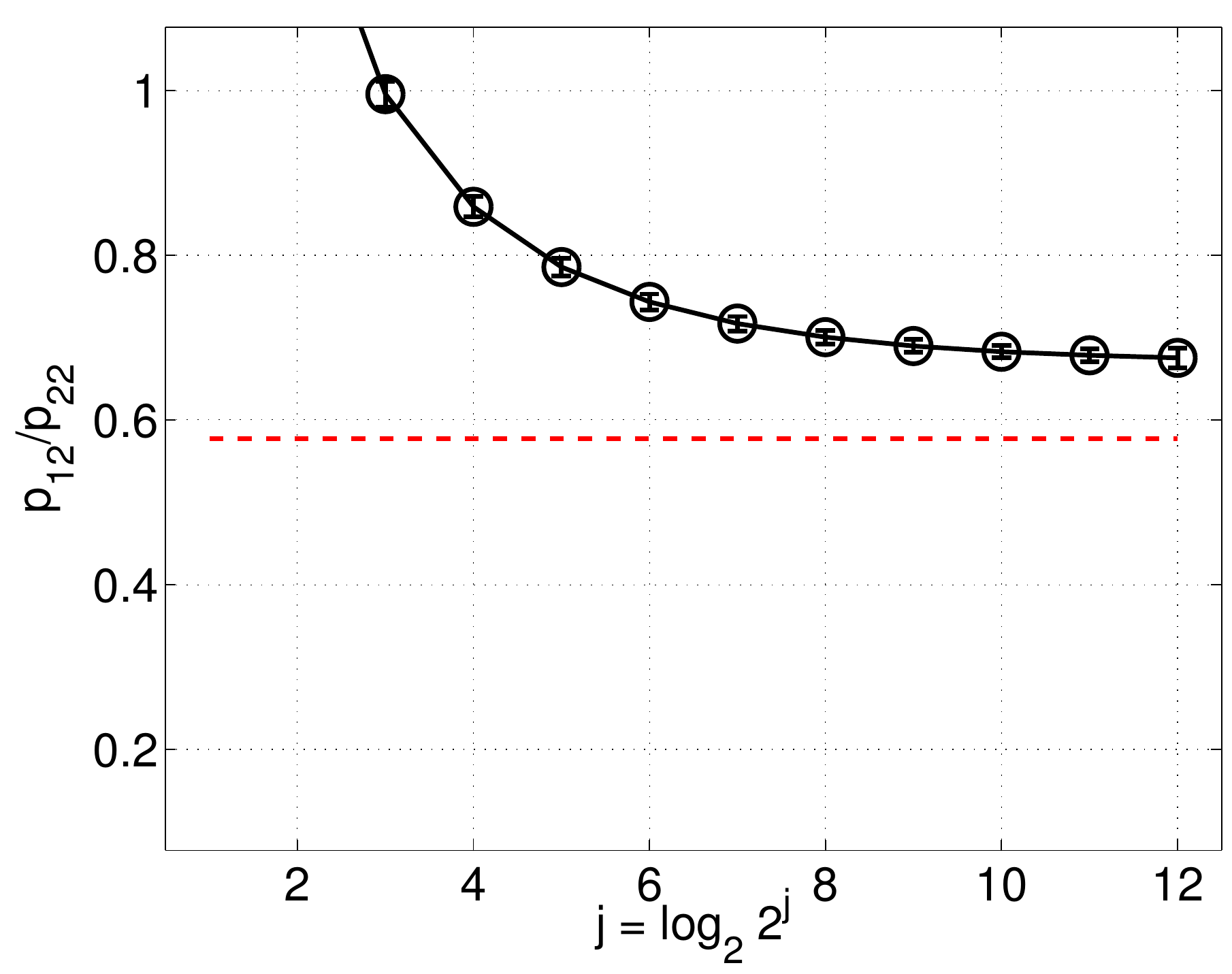}
}
\centerline{
\includegraphics[height=40truemm,keepaspectratio]{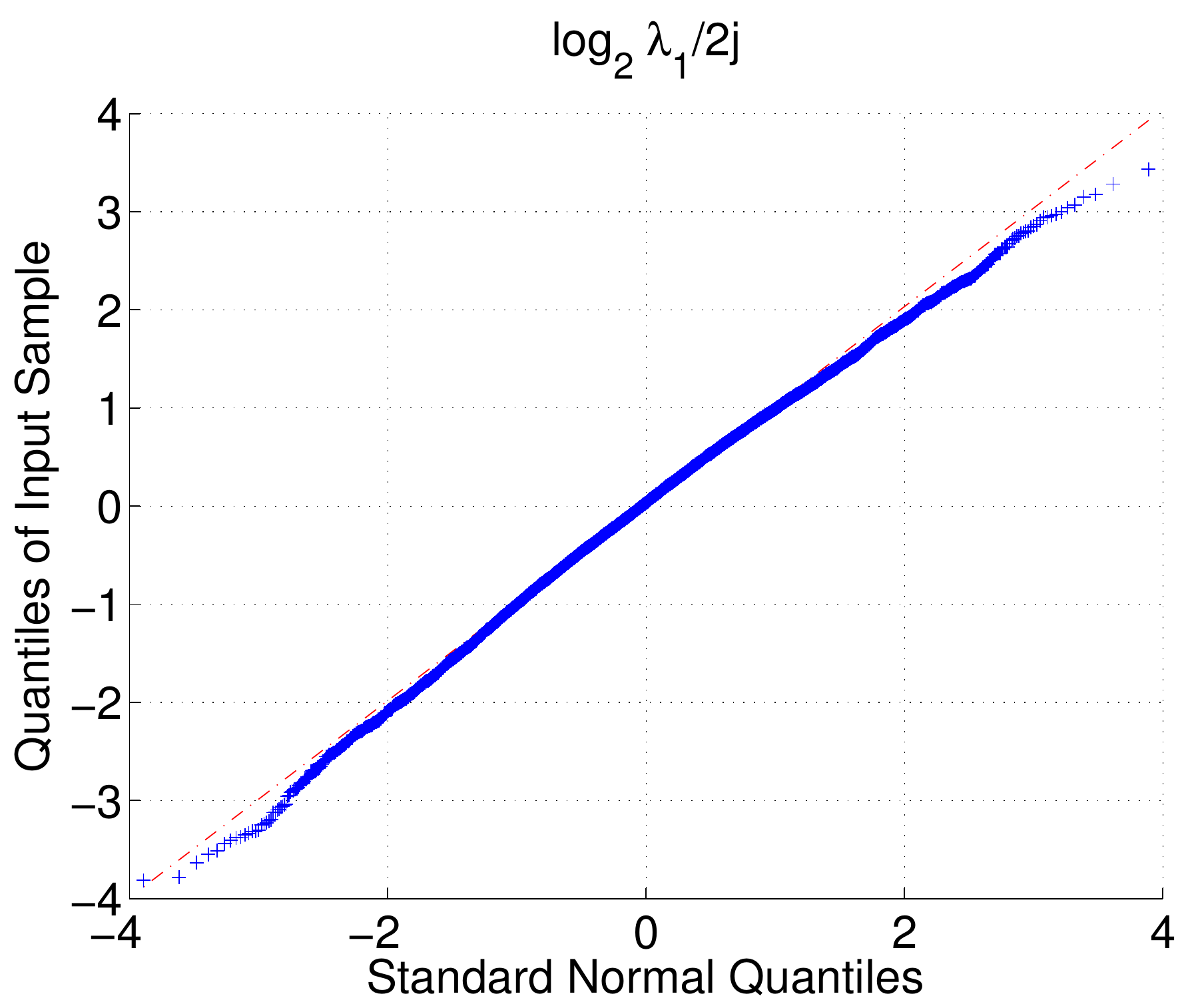} \hspace{3mm}
\includegraphics[height=40truemm,keepaspectratio]{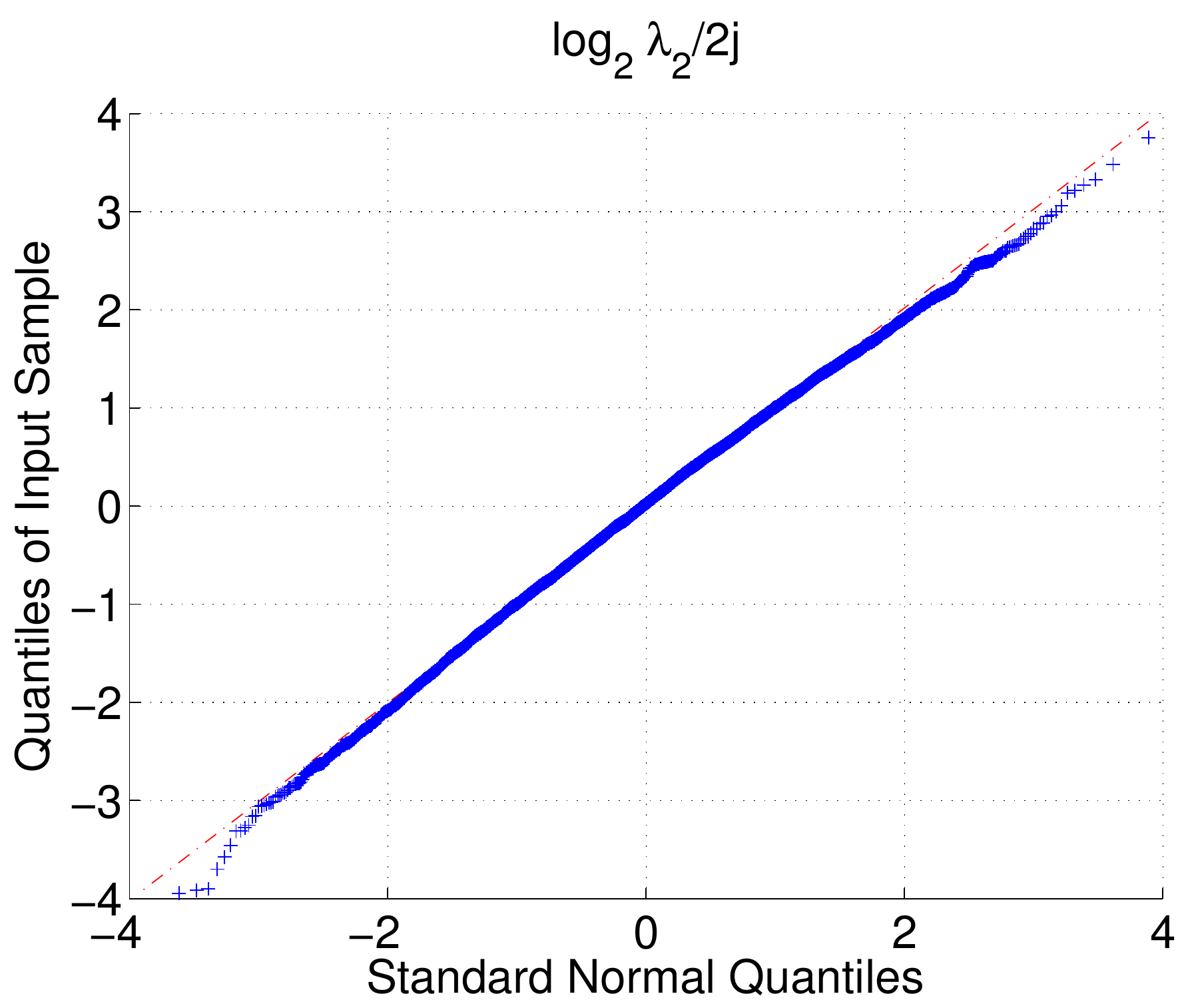} \hspace{3mm}
\includegraphics[height=40truemm,keepaspectratio]{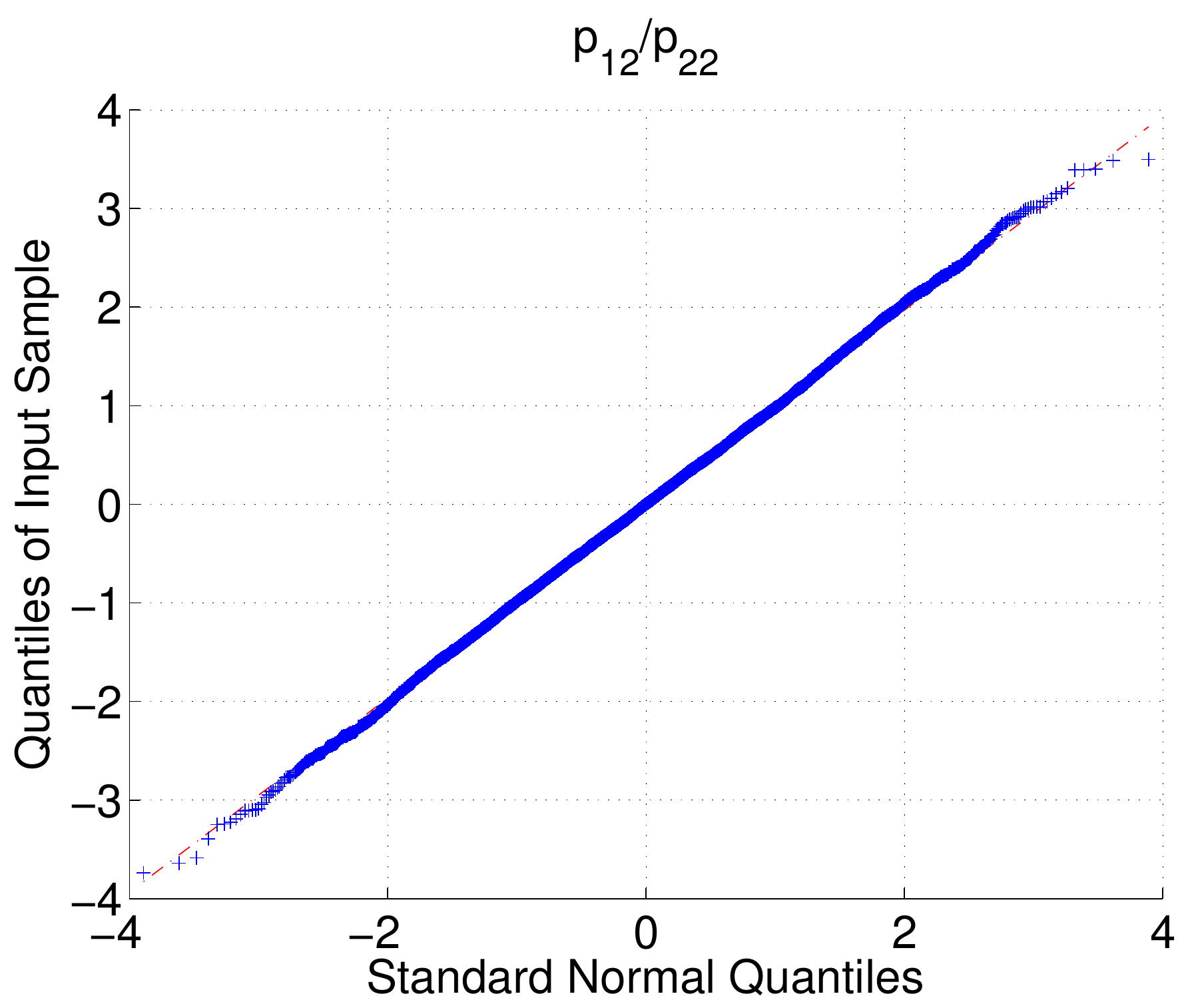}
}
\caption{\label{fig:figb} {\bf Estimation performance and asymptotic normality:}  OFBM with $\gamma = -\beta $ and $\beta/\sqrt{1+\beta^2} = \sin \pi/6$.
Top row: black solid lines with `o' show $ \hat E \widehat{\vartheta} \pm \sqrt{ \widehat{ \textnormal{ Var }}  \widehat{\vartheta} /n}$
for $\widehat{\vartheta} = \log_2 \lambda_1 (2^j)/ 2 j$,  $\log_2 \lambda_2(2^j) / 2 j $ and $p_{12}(2^j)/p_{22}(2^j)$ (target parameters: $h_1$ (left plots), $h_2$ (center plots), $p_{12}/p_{22}$ (right plots), respectively), red dashed lines correspond to theoretical values.
Bottom row, corresponding qq-plots (against ${\cal N}(0,1)$ distributions) for $ j=10 $.
}
 \end{figure}

 \begin{figure}[h]
\centerline{
\includegraphics[height=40truemm,keepaspectratio]{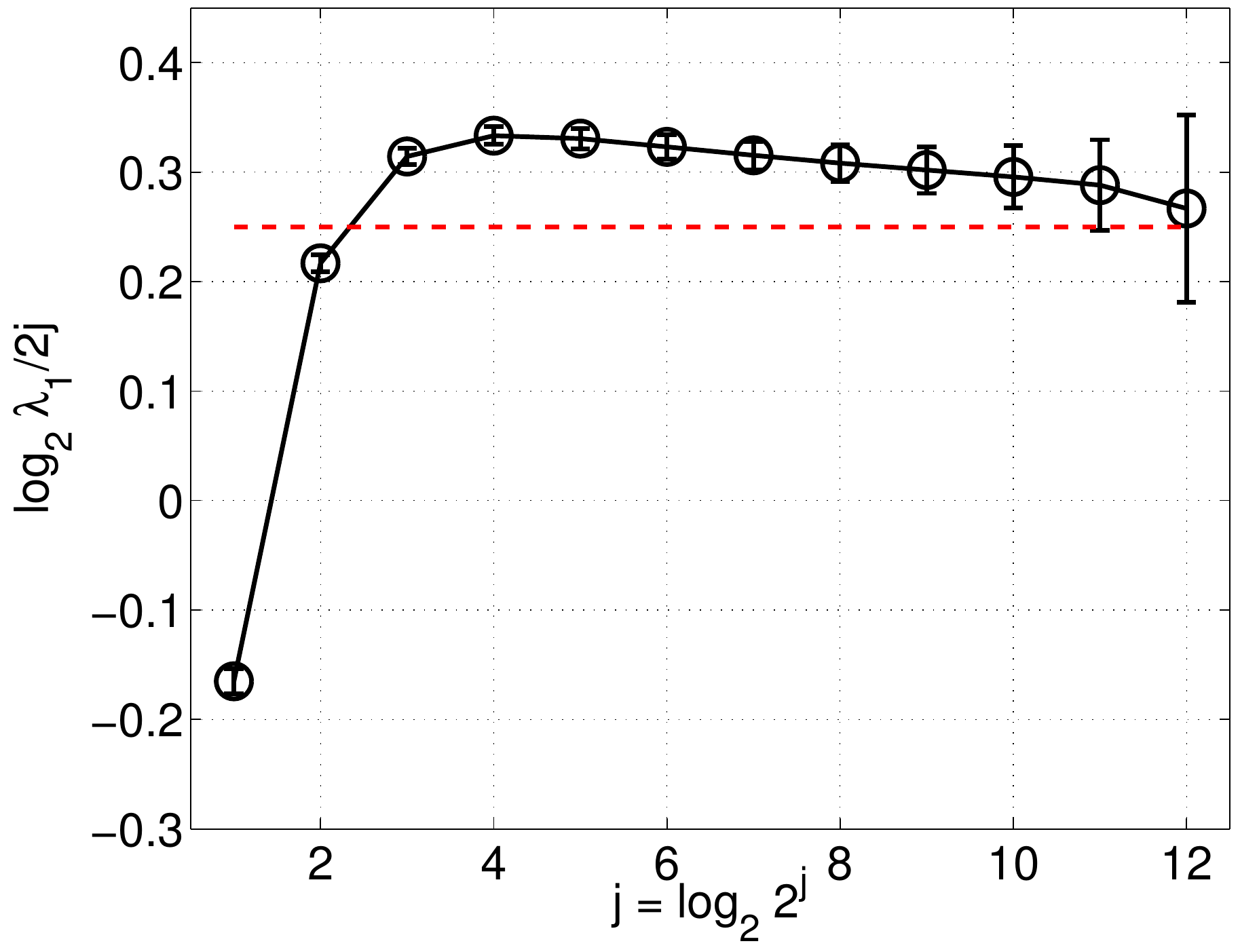}
\includegraphics[height=40truemm,keepaspectratio]{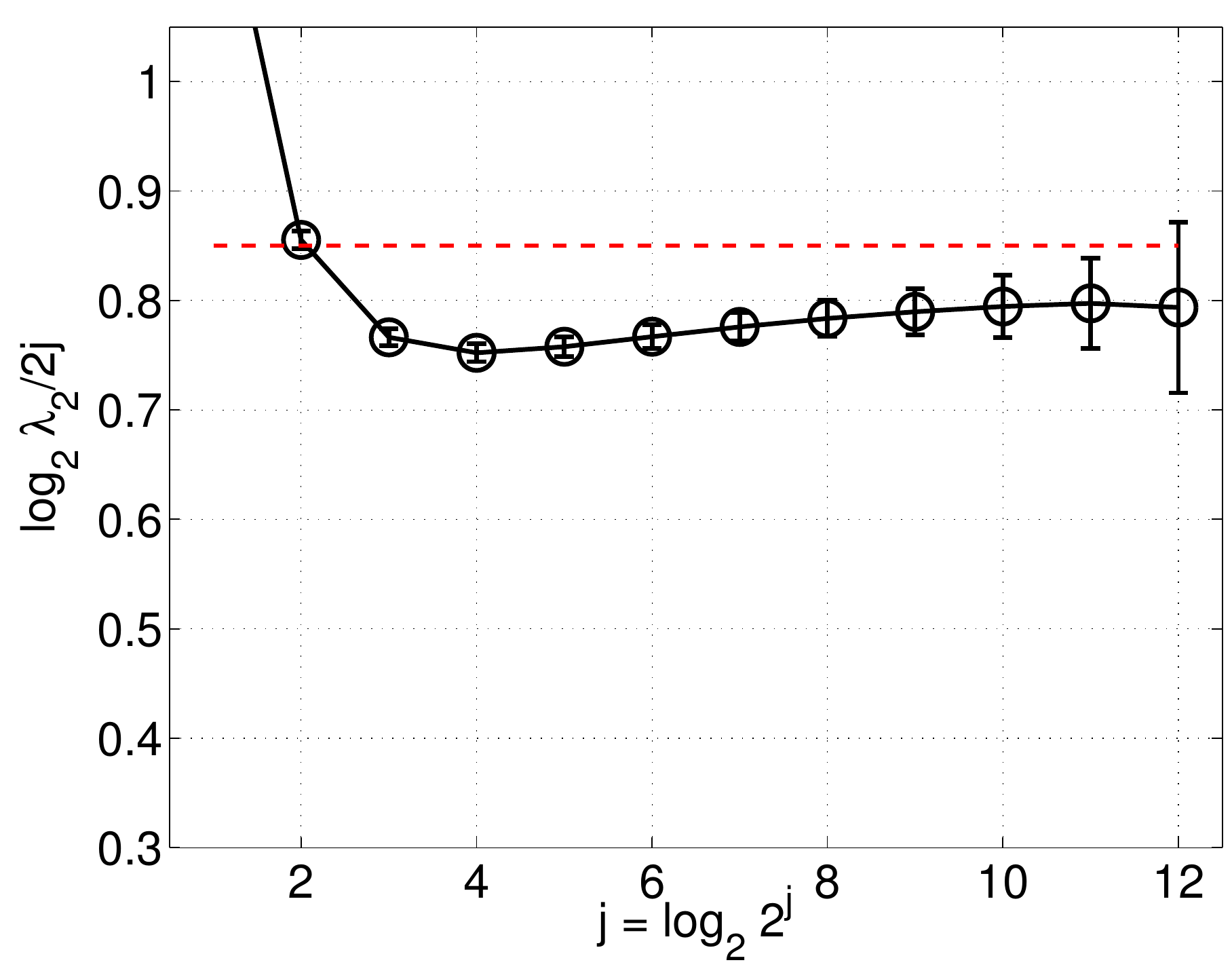}
\includegraphics[height=40truemm,keepaspectratio]{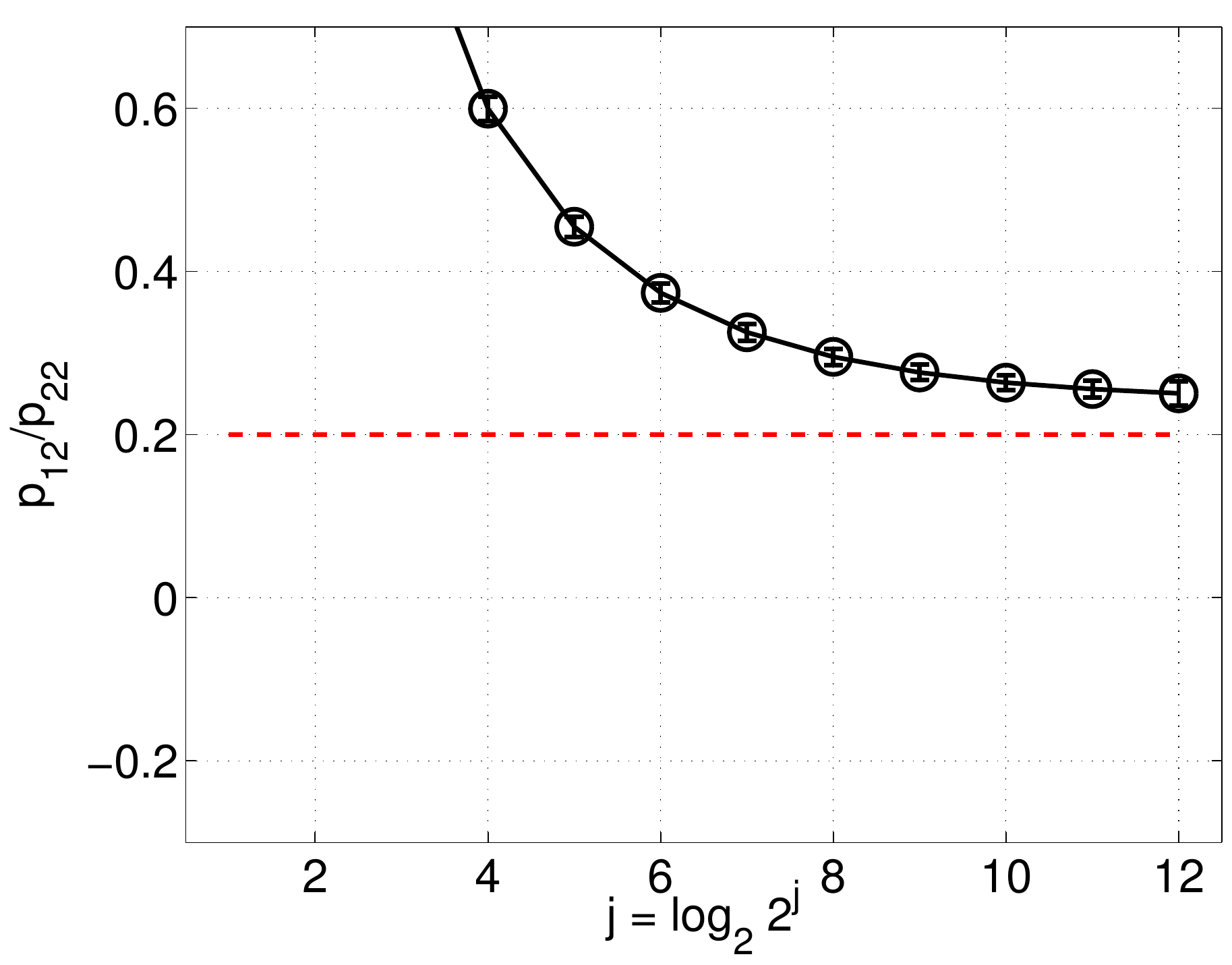}
}
\centerline{
\includegraphics[height=40truemm,keepaspectratio]{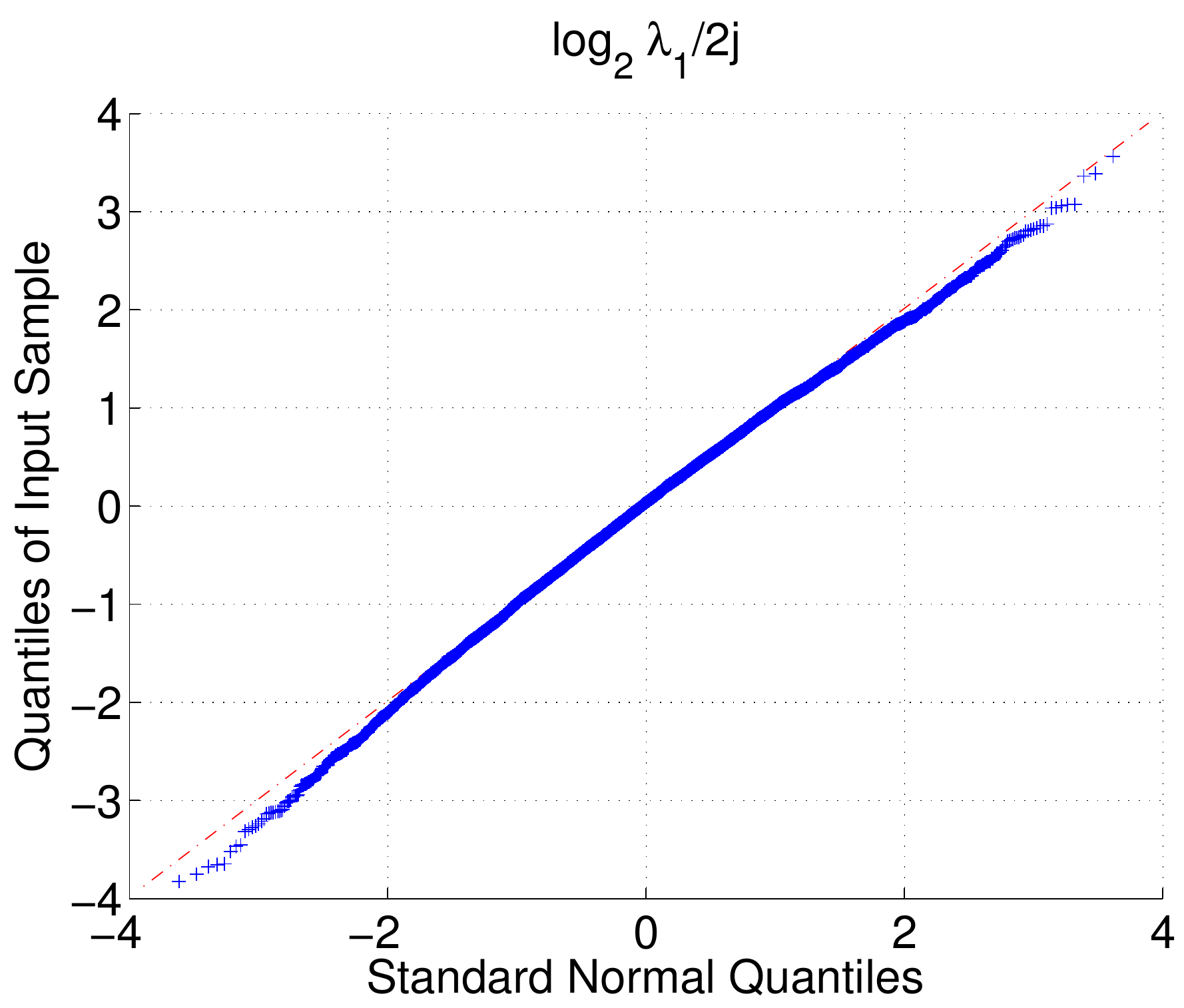} \hspace{3mm}
\includegraphics[height=40truemm,keepaspectratio]{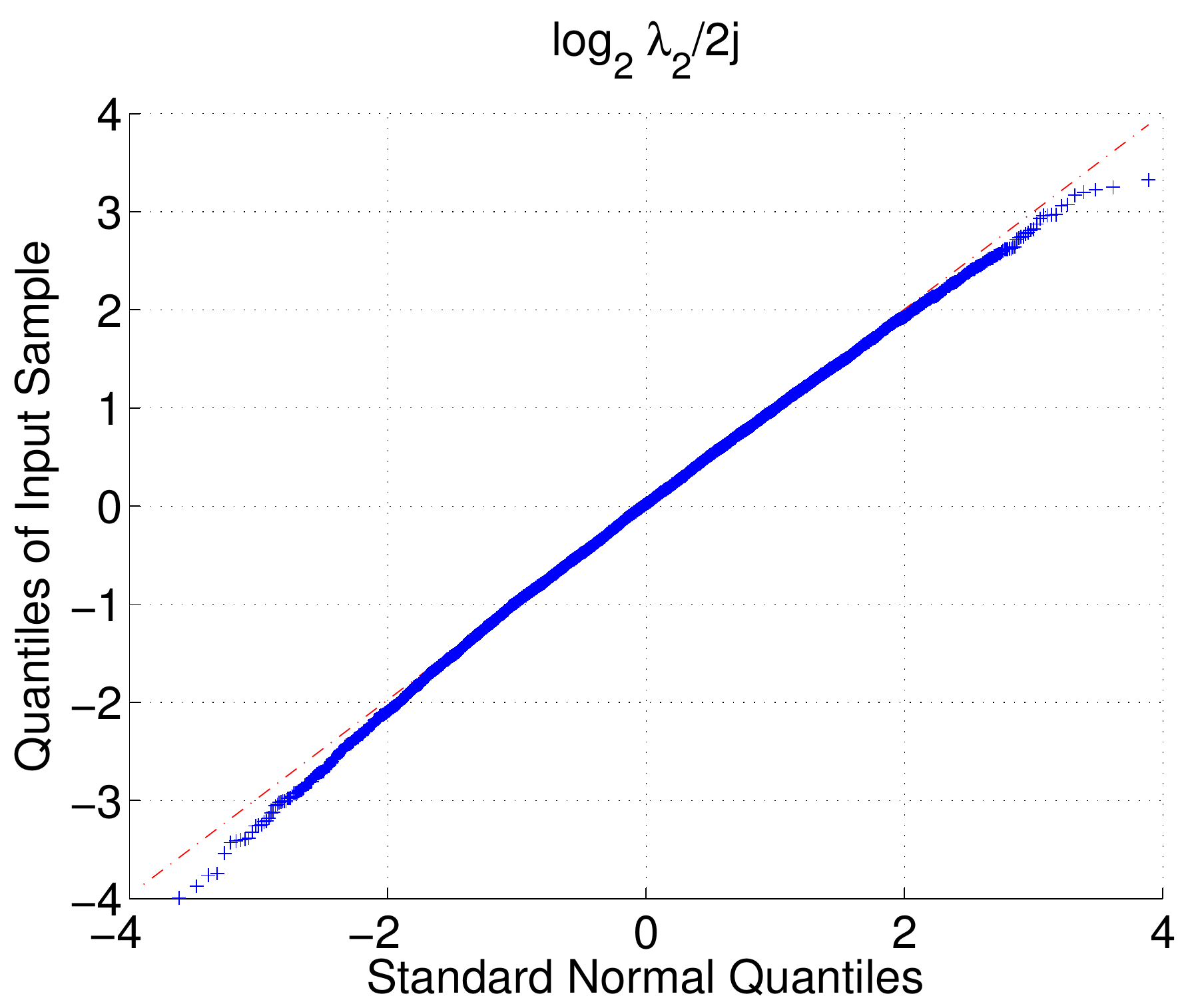} \hspace{3mm}
\includegraphics[height=40truemm,keepaspectratio]{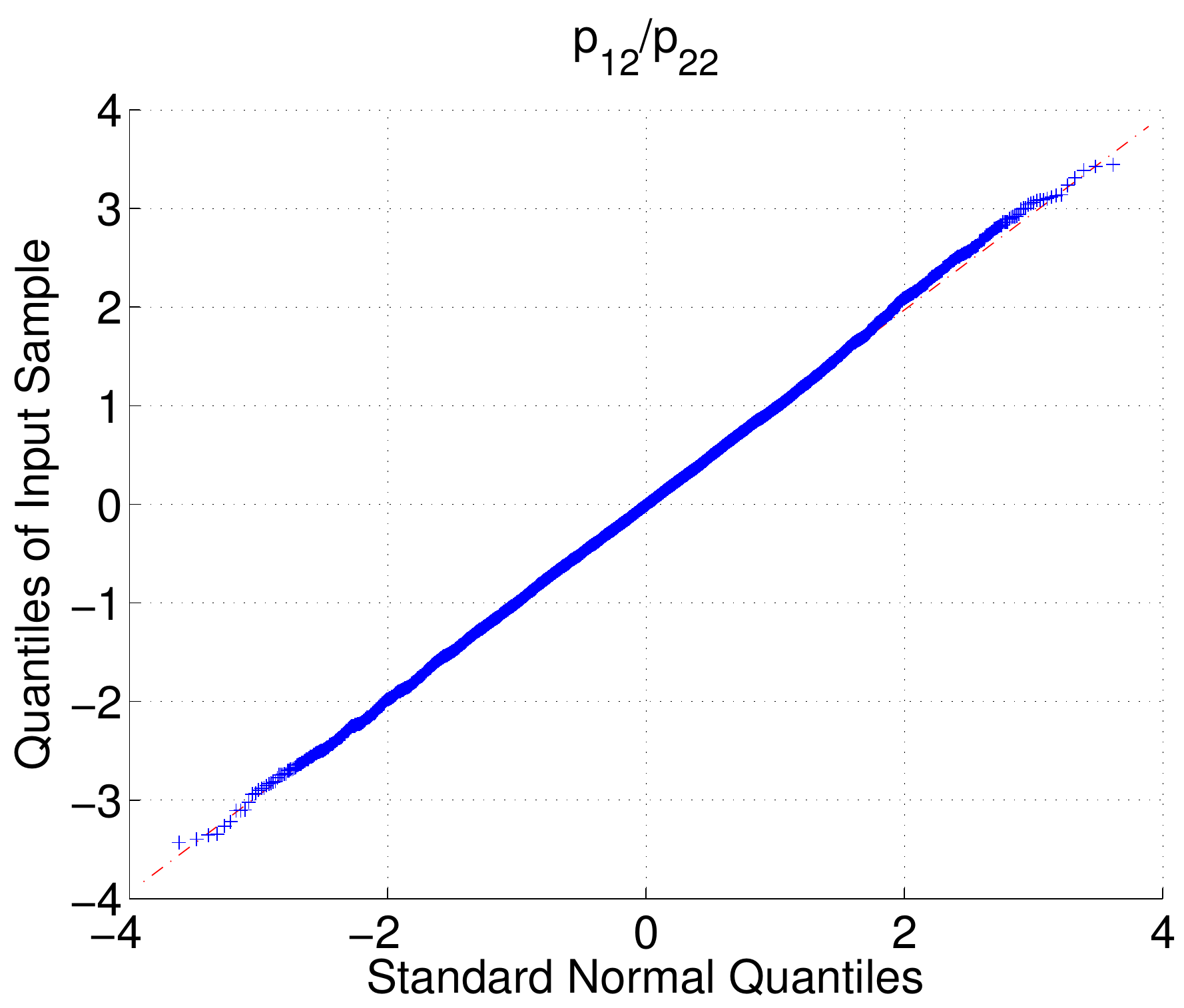}
}
\caption{\label{fig:figc} {\bf  Estimation performance and asymptotic normality:}  OFBM with $\gamma = 0 $ and $\beta= 0.2$.
Top row: black `o' in solid lines show $ \hat E \widehat{\vartheta} \pm \sqrt{ \widehat{ \textnormal{ Var }}  \widehat{\vartheta} /n}$
for $\widehat{\vartheta} = \log_2 \lambda_1 (2^j)/ 2 j$,  $\log_2 \lambda_2(2^j) / 2 j $ and $p_{12}(2^j)/p_{22}(2^j)$ (target parameters: $h_1$ (left plots), $h_2$ (center plots), $p_{12}/p_{22}$ (right plots), respectively), red dashed lines correspond to theoretical values.
Bottom row, corresponding qq-plots (against ${\cal N}(0,1)$ distributions) for $ j=10 $.
}
 \end{figure}

 \begin{figure}[h]
\centerline{
\includegraphics[height=40truemm,keepaspectratio]{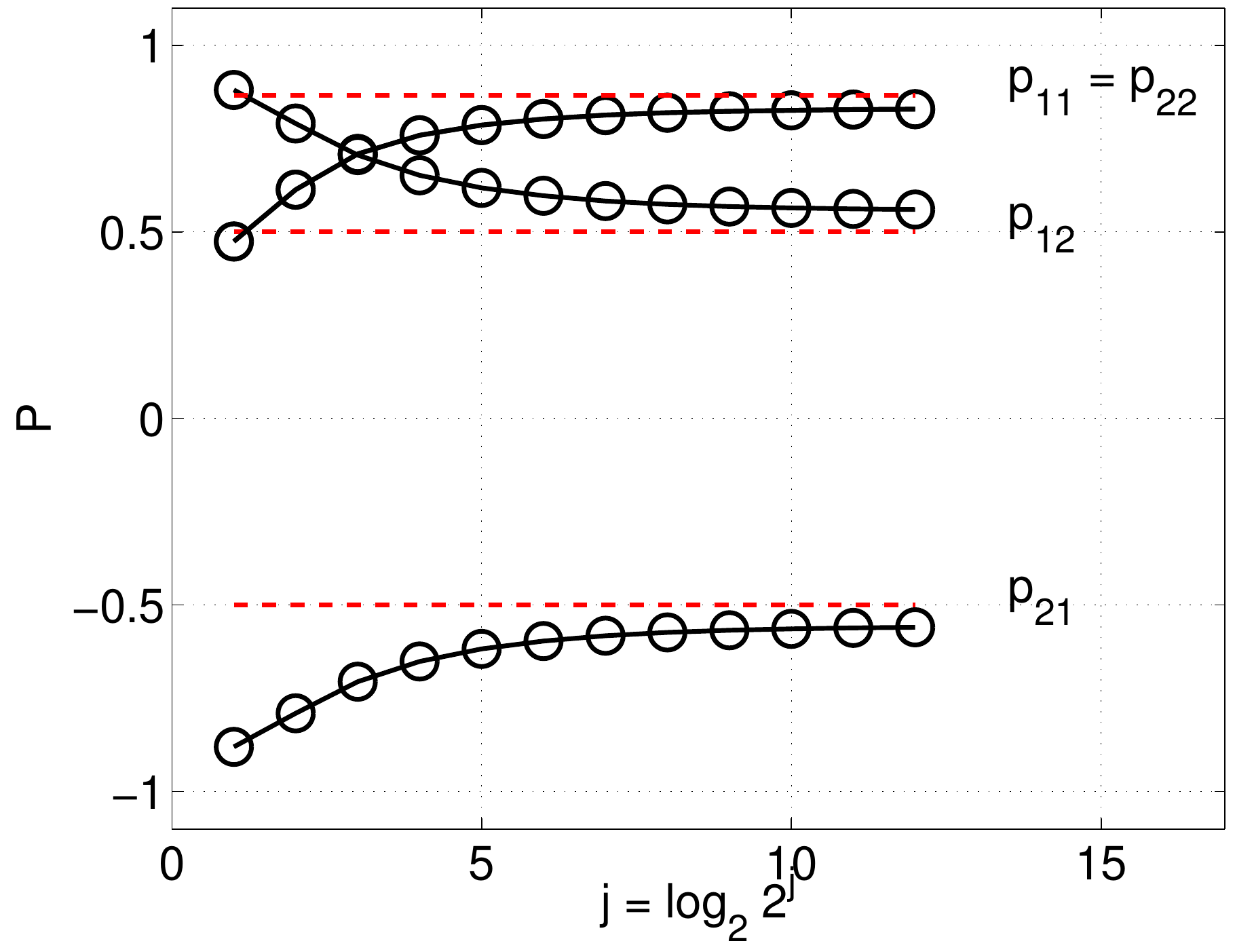}
}
\caption{\label{fig:figd} {\bf  Estimation performance when $P \in O(2)$:} OFBM with $\gamma = -\beta $ and $\beta/\sqrt{1+\beta^2} = \sin \pi/6$.
Estimates $\hat p_{k,l}(2^j) $, $k,l \in \{1, 2\}^2$ (black solid lines with `o')  of the entries of $P$, $p_{k,l} $ (red dashed line), from the eigenvectors of the $W(2^j)$. When $P \in O(2)$, both $P$ and $J_H$ can be estimated, and the matrix exponent $H$ is fully estimated.
}
\end{figure}

\subsection{Beyond the bivariate setting}

It is natural to ask whether in higher dimension the eigenvalues of the sample wavelet variance are good estimators of the Hurst eigenvalues. Figure \ref{fig:fige} shows eigenvalue-based estimation at work for $n=4$, with $\log_2 \lambda_p(2^j) /2j$, for $p=1, 2, 3, 4 $, averaged over 2,000 realizations of an OFBM with parameters
$$
P =  \left(\begin{array}{cccc}
0.90 &   -0.22 & -0.30 &   -0.22 \\
    0.43 &   0.45 &    0.63 &    0.46 \\
         0  & -0.85 &    0.40&    0.30 \\
         0   &      0  &  -0.59  & 0.81 \\
\end{array}\right)$$
and $ J_H =  \textnormal{diag}(0.20,0.40,0.70,0.90)$, $ N = 2^{16}$. The computational results indicate that the smallest and largest eigenvalues of $W(2^j)$ are still good estimators of the smallest and largest entries of $J_H$. Furthermore, there is evidence that the method produces reasonable estimates of all the intermediate-valued entries of $J_H$ based on the corresponding intermediate eigenvalues of $W(2^j)$. In regard to eigenvector estimation under the assumption $P \in O(4)$, unreported numerical studies suggest that the sample wavelet variance eigenvectors are also good estimators for $P$. 

 \begin{figure}[h]
\centerline{
\includegraphics[height=40truemm,keepaspectratio]{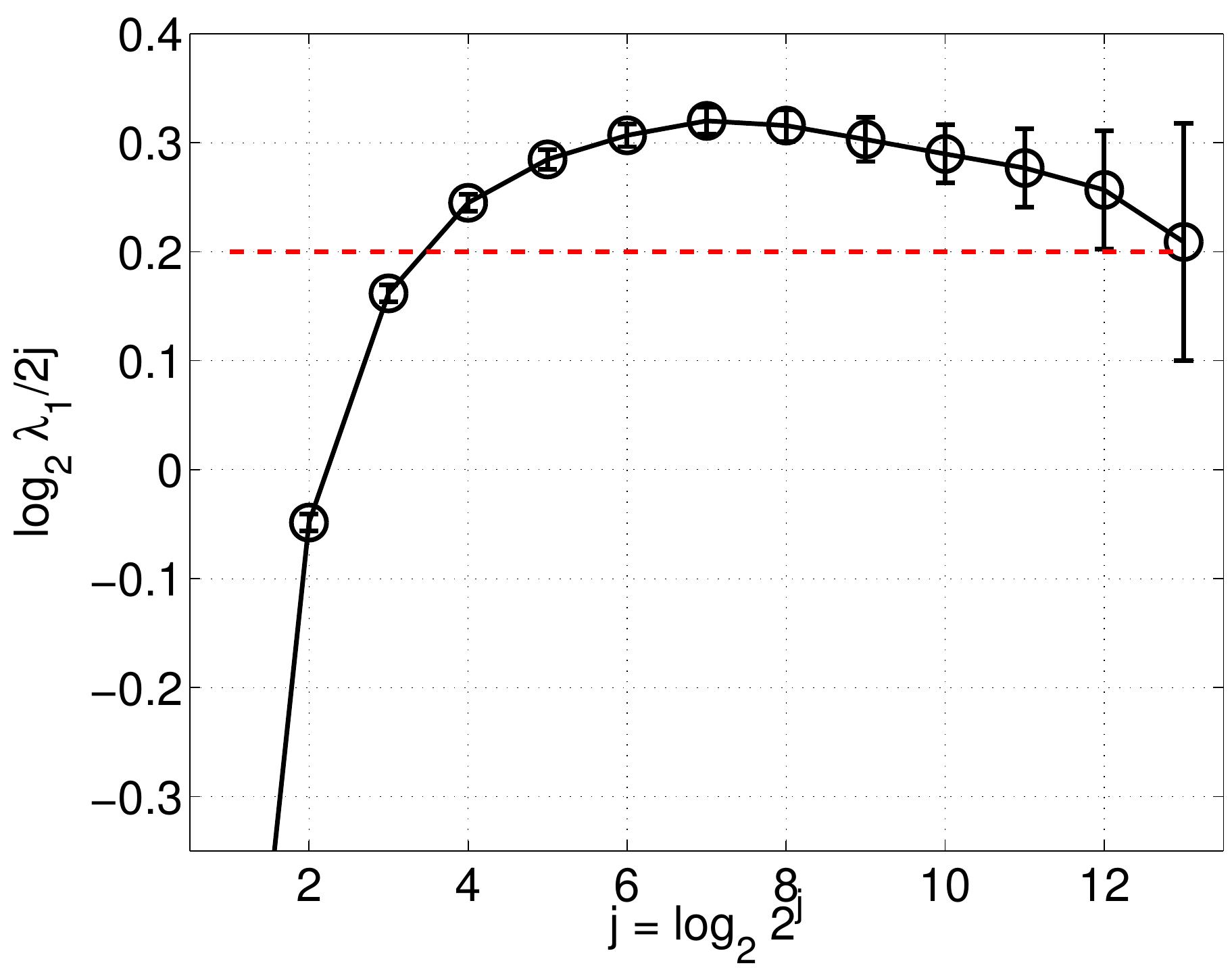}
\includegraphics[height=40truemm,keepaspectratio]{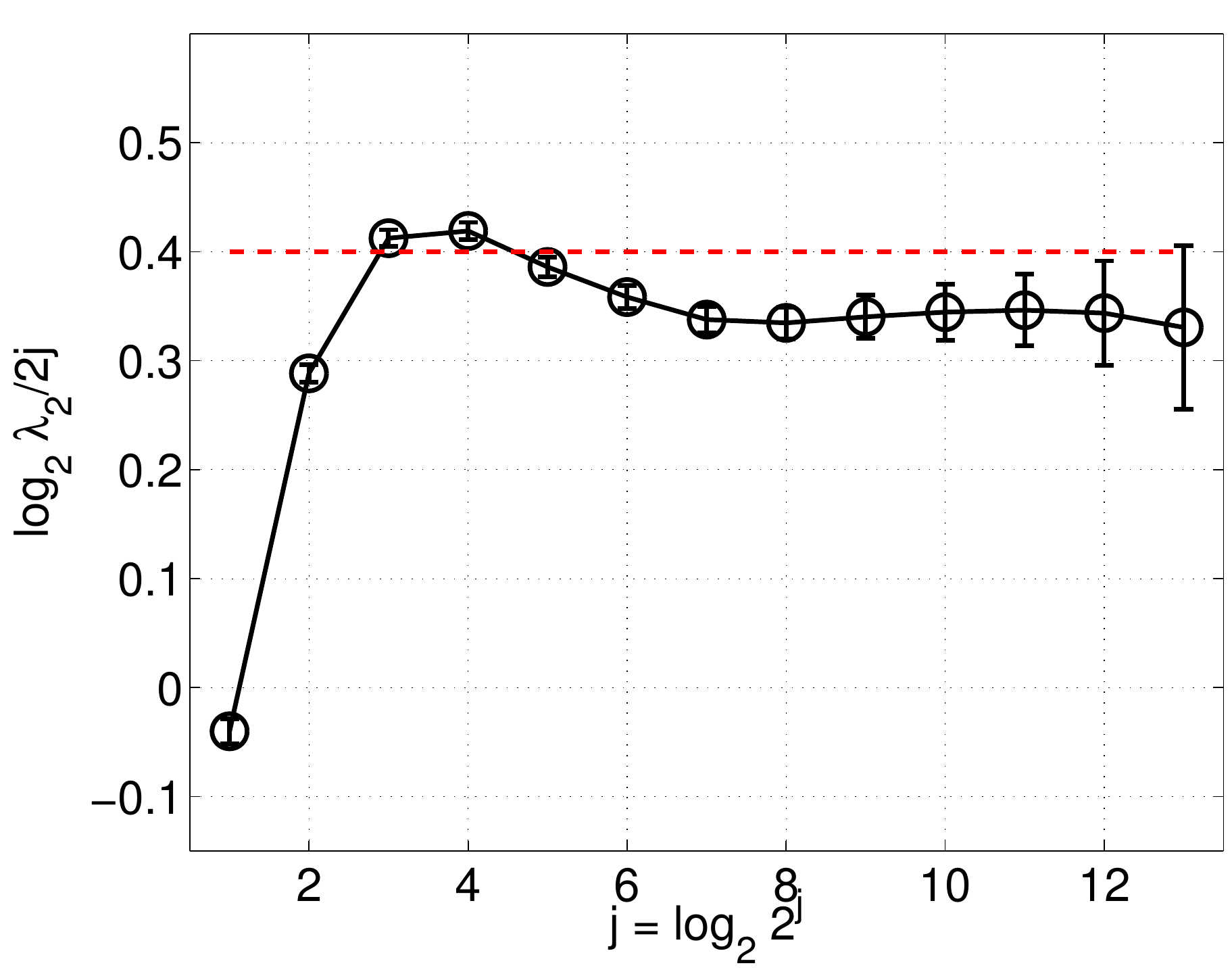}
\includegraphics[height=40truemm,keepaspectratio]{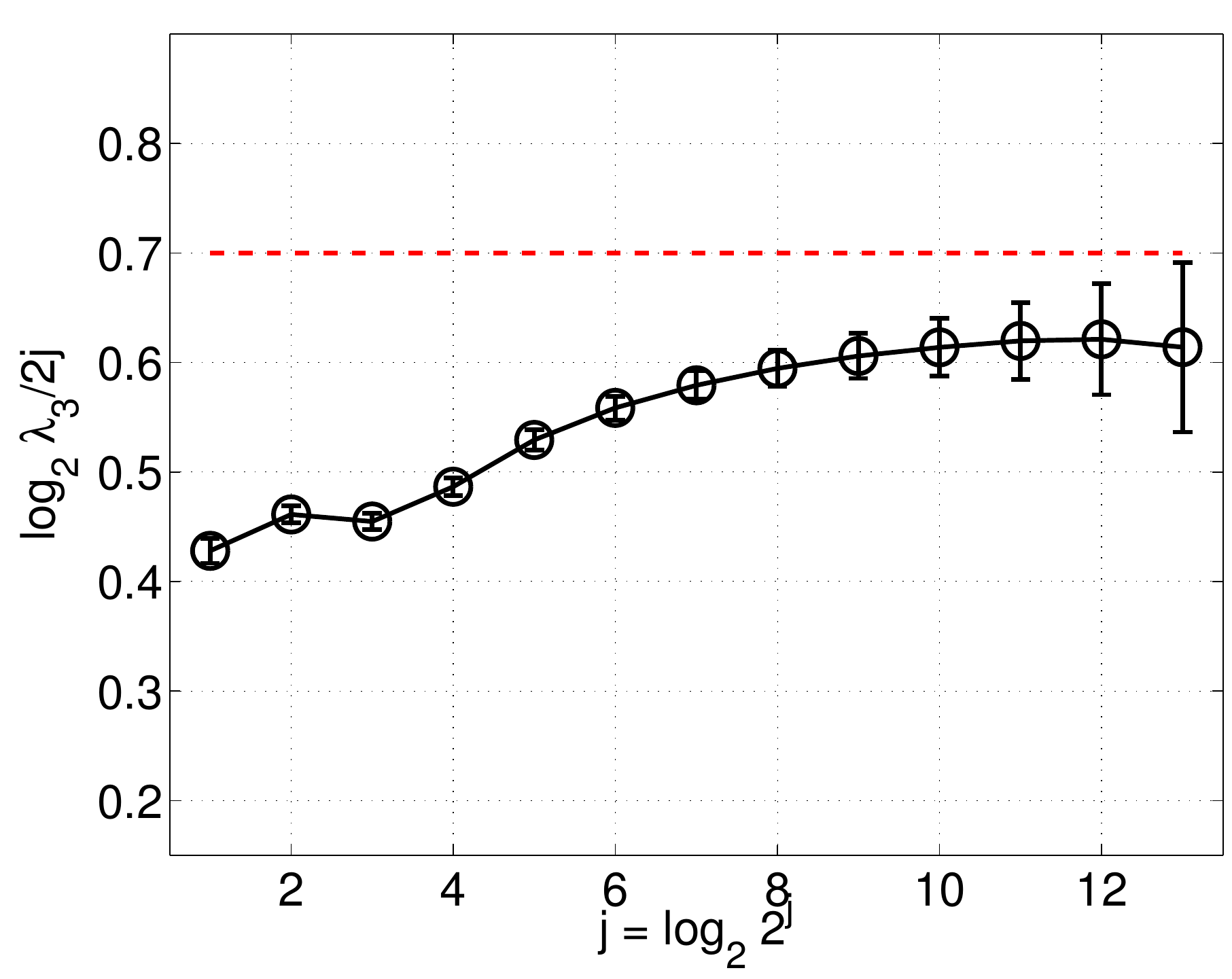}
\includegraphics[height=40truemm,keepaspectratio]{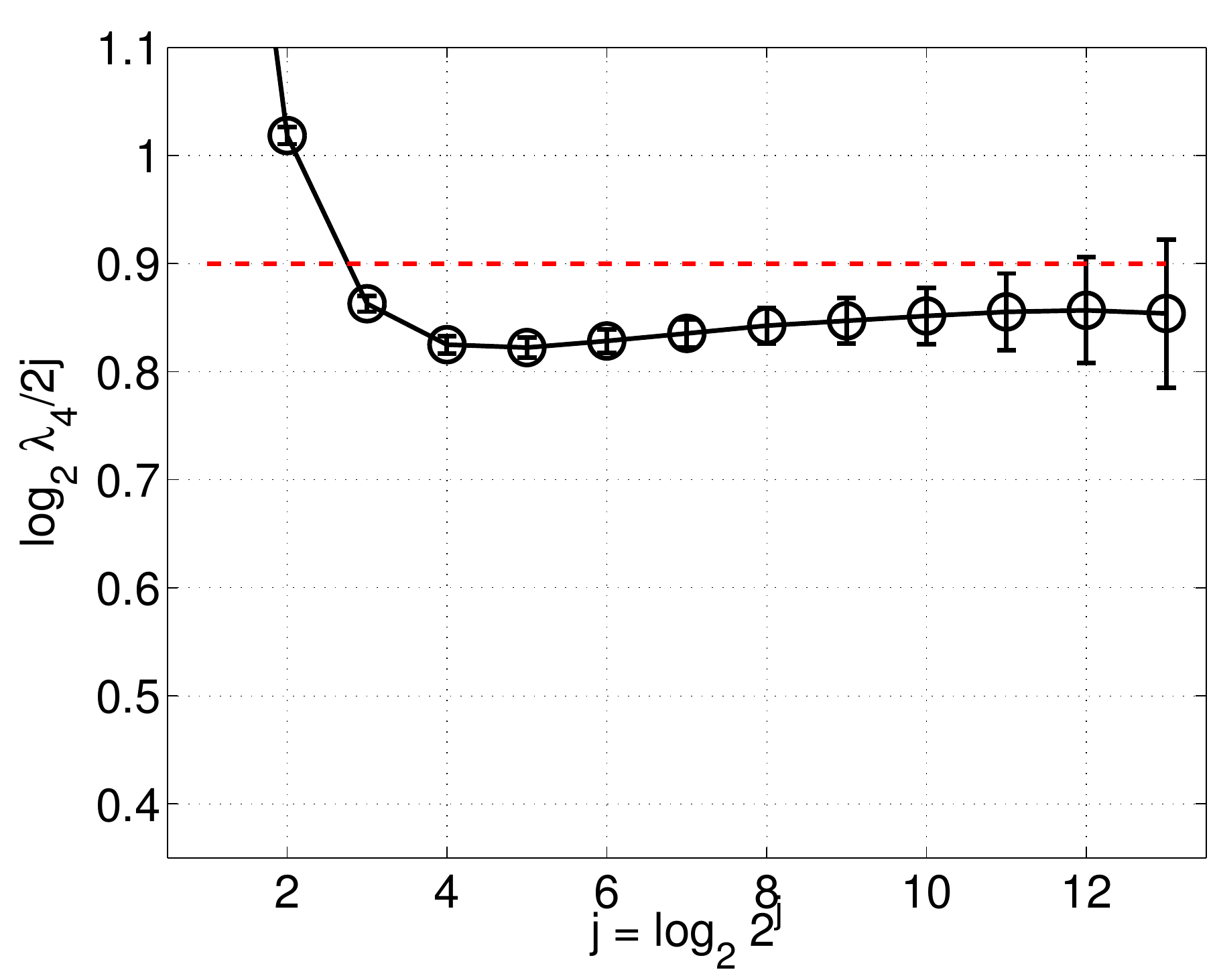}
}
\centerline{
\includegraphics[height=40truemm,keepaspectratio]{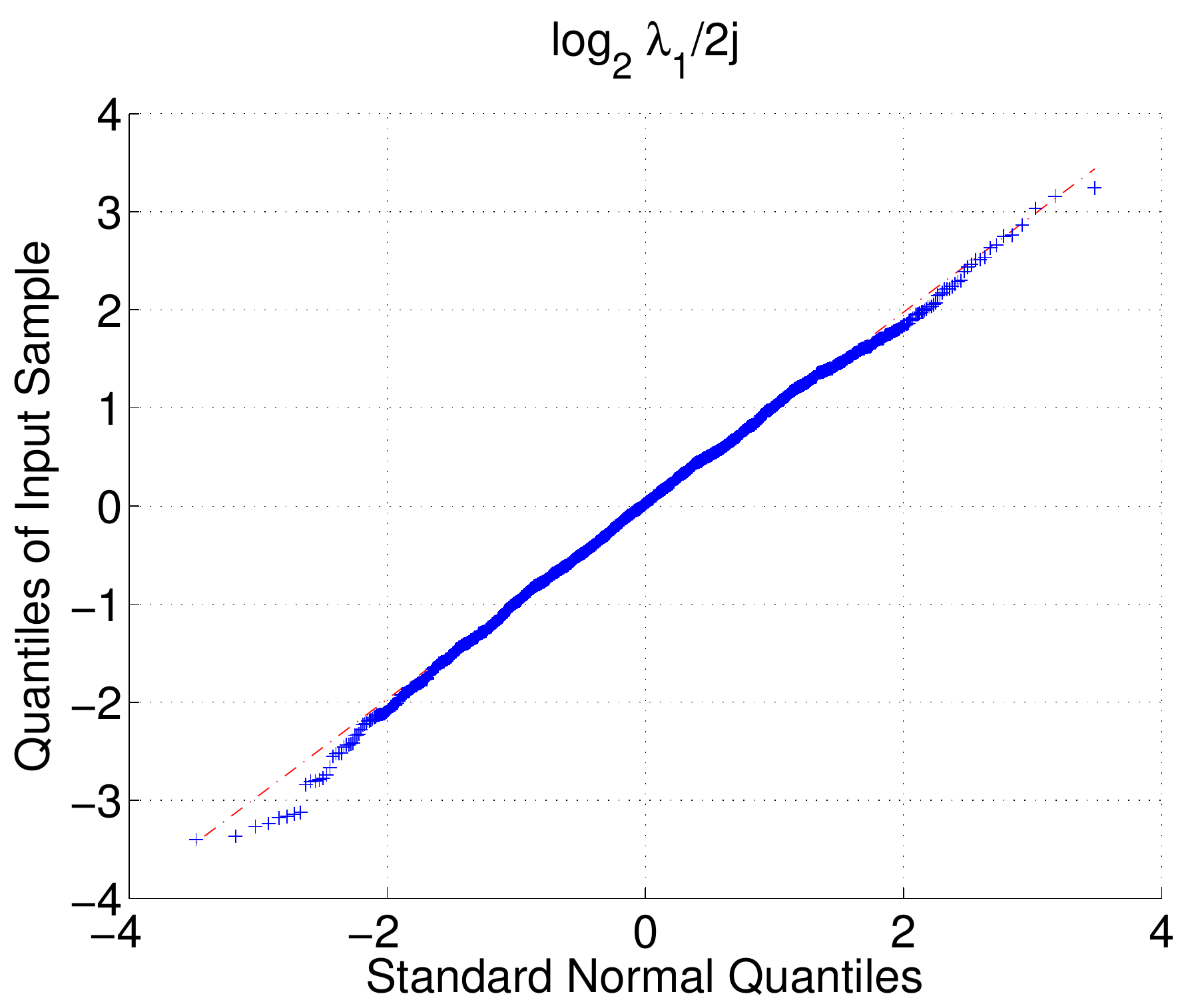} \hspace{3mm}
\includegraphics[height=40truemm,keepaspectratio]{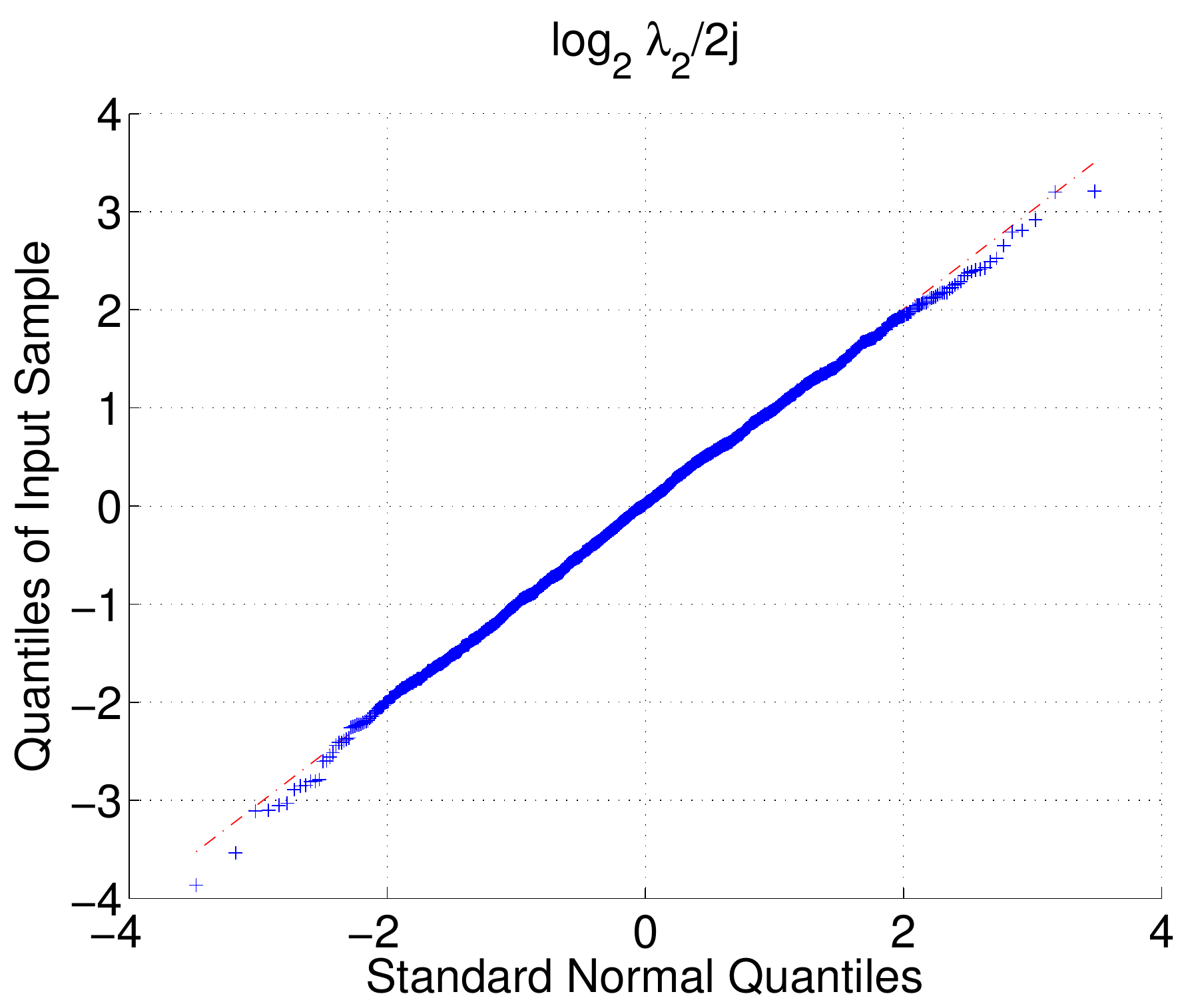} \hspace{3mm}
\includegraphics[height=40truemm,keepaspectratio]{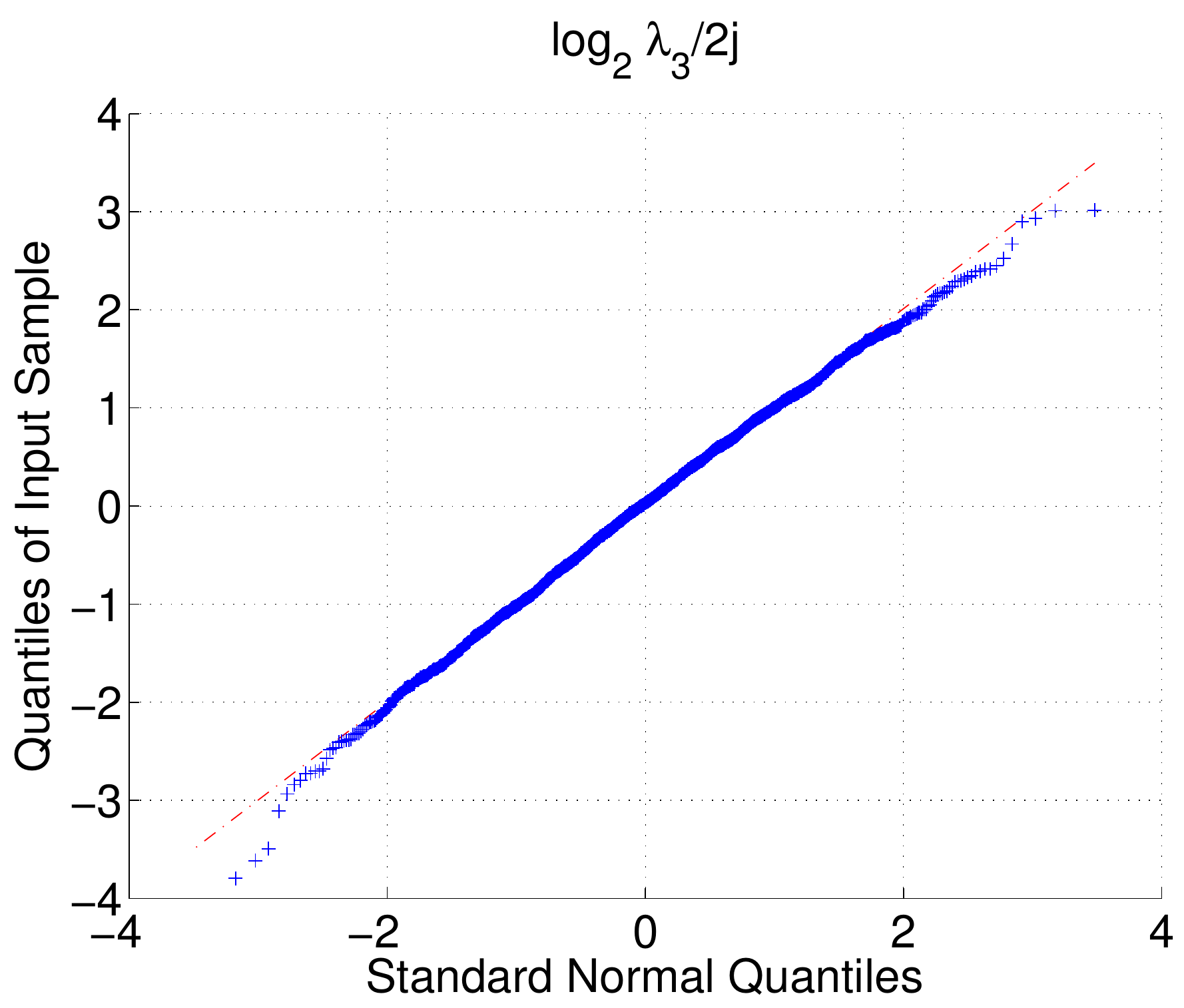} \hspace{3mm}
\includegraphics[height=40truemm,keepaspectratio]{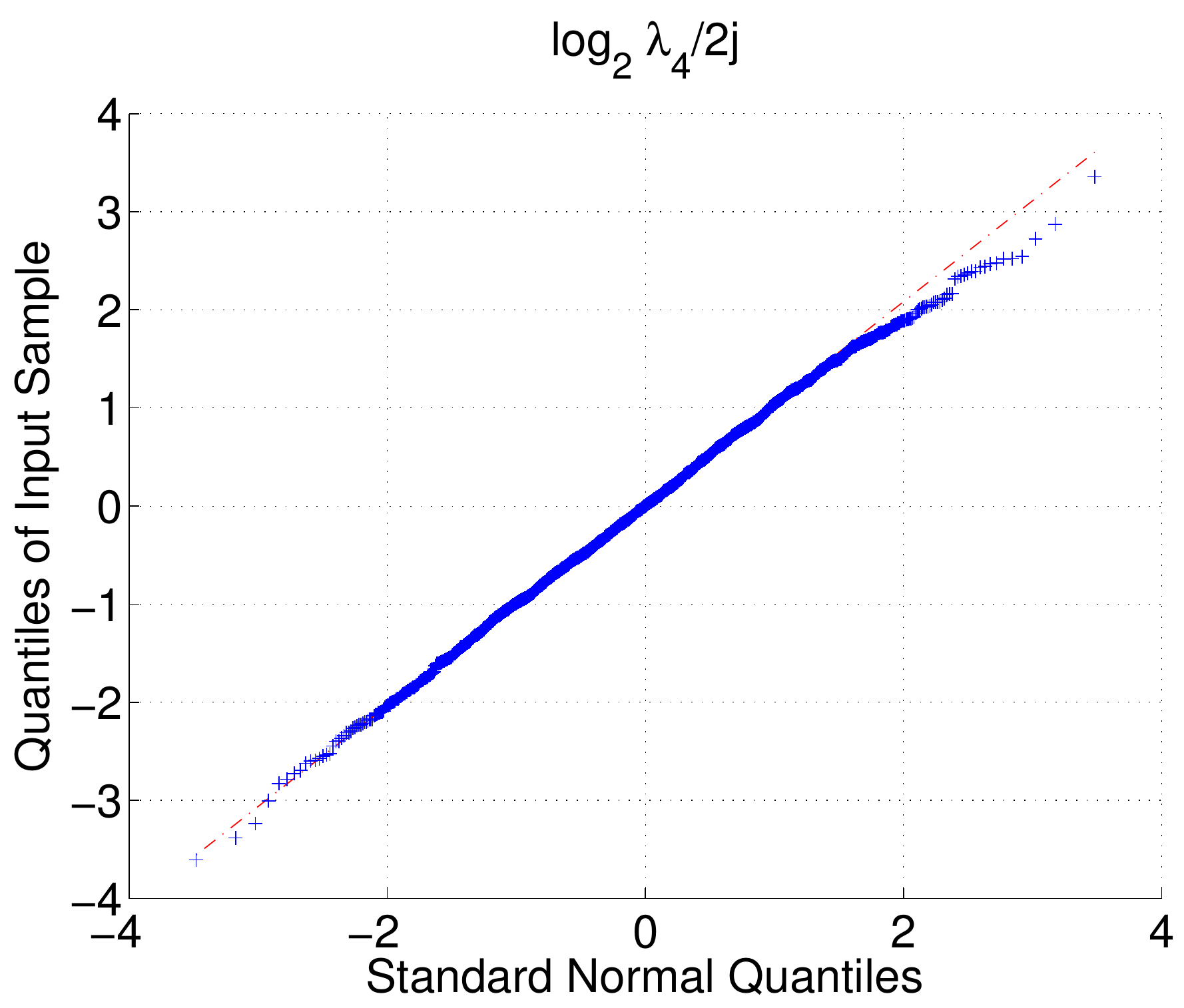}
}
\caption{\label{fig:fige} {\bf  Estimation performance and asymptotic normality:}  OFBM in dimension 4.
Top row: black solid lines with `o' show $ \hat E \widehat{\vartheta} \pm \sqrt{ \widehat{ \textnormal{ Var }}  \widehat{\vartheta} /n}$
for $\widehat{\vartheta} = \log_2 \lambda_1 (2^j)/ 2 j$,  $\log_2 \lambda_2(2^j) / 2 j $, $\log_2 \lambda_3(2^j) / 2 j $ and $\log_2 \lambda_4(2^j) / 2 j $ (target parameters: $h_1$, $h_2$, $h_3$, $h_4$, respectively, with the plots in the same order), red dashed lines correspond to theoretical values.
Bottom row, corresponding qq-plots (against ${\cal N}(0,1)$ distributions) for $j=10$.
}
 \end{figure}

 \section{Perspectives and open issues}\label{s:open}

OFBM constitutes the natural multivariate extension of the univariate FBM, allowing a different Hurst eigenvalue in each coordinate, in an arbitrary coordinate system. When the mixing matrix $P$ is non-diagonal, the problem of estimating $H$ becomes distinctively multivariate and is not easily amenable to approaches inspired in the univariate context. In this work, we propose a change of perspective from univariate-like, entry-wise scaling relations to the eigenstructure of the sample wavelet variance $W(2^j)$ across scales. In the bivariate setting, this methodology is mathematically shown to yield consistent and asymptotically normal estimates of the Hurst eigenvalues as well as of the eigenspace angle parameter $-p_{12}/p_{22}$. Consequently, the matrix Hurst parameter $H$ can be fully estimated when $P \in O(2)$. Large sample size simulation was used to illustrate and shed further light on the weak limits obtained. The research contained in this paper has lead to five open issues, currently under investigation: ($i$) the quantitative assessment of the performance of the estimators as a function of sample size; how much data is demanded by the difficult problem of estimating operator-scaling systems, especially in high dimension?; ($ii$) is there an advantage to using multiple scales in a regression, as in univariate wavelet-based estimation?; ($iii$) mathematical extensions to OFBM in higher dimension; ($iv$) the estimation of non-orthogonal coordinate systems; ($v$) applications in real data. In the near future, a {\sc Matlab} toolbox for the estimators proposed in this paper will be made publicly available.

\appendix

\section{Additional results}\label{s:additional}

{\sc Proof of Proposition \ref{p:wavelet_coefs_properties}}
We first show ($P1$). Since the covariance function $EB_H(s)B_H(t)^*$ is continuous, by Cram\'{e}r and Leadbetter \cite{cramer:leadbetter:1967}, p.86, it suffices to show that
\begin{equation}\label{e:integ_EBH(s)BH(t)_psi(s)_psi(t)dsdt}
\int_{\bbR} \int_{\bbR} \| EB_H(s)B_H(t)^* \|_{l^1} \hspace{1mm}|\psi(s)| |\psi(t)| \hspace{1mm}ds dt < \infty.
\end{equation}
In fact,
\begin{equation}\label{e:EBH(s)BH(t)_bound}
\| EB_H(s)B_H(t)^* \|_{l^1} = \sum^{n}_{i_1=1}\sum^{n}_{i_2=1}|E B_H(s)_{i_1}B_H(t)_{i_2}|
\leq  \Big( \sum^{n}_{i_1=1}\sqrt{E B_H(s)^2_{i_1}} \Big)\Big(\sum^{n}_{i_2=1}\sqrt{ E B_H(t)^2_{i_2}}\Big).
\end{equation}
However, for $t \in \bbR$, $\|EB_H(t)B_H(t)^*\|_{l^{\infty}} = \|t^{H}\Sigma t^{H^*}\|_{l^{\infty}} \leq C |t|^{2 \max \Re\textnormal{eig}(H) }$.
Therefore, \eqref{e:EBH(s)BH(t)_bound} is bounded by $C |t|^{\max \Re \hspace{0.5mm}\textnormal{eig}(H) } |s|^{\max \Re \hspace{0.5mm}\textnormal{eig}(H) }$.
By conditions \eqref{e:supp_psi=compact} and \eqref{e:psihat_is_slower_than_a_power_function}, \eqref{e:integ_EBH(s)BH(t)_psi(s)_psi(t)dsdt} holds. The fact that $E\|B_H(t)\|_{l^1} < \infty $ and the assumption \eqref{e:supp_psi=compact} yield
$$
E \int_{\bbR}|\psi(t)|\|B_H(2^j t + 2^j k)\|_{l^1} dt \leq C \int_{\bbR}|\psi(t)| \hspace{1mm}E\|B_H(t+k)\|_{l^1}dt
$$
$$
\leq  C E\|B_H(1)\| \int_{\bbR} |\psi(s-k)| \|s^H\|_{l^1}ds < \infty,
$$
where we used the change-of-variables $s = t+k$. By Fubini, this yields $(P1)$, as claimed.

In view of ($P1$), the properties ($P2$), ($P3$) and ($P5$) and can be established by arguments similar to those for the univariate case (see, for instance, the argument in Delbeke and Abry \cite{delbeke:abry:2000}, Theorem 3 for the former two, and Bardet \cite{bardet:2000}, Proposition II.1, for the latter). Note that $\int_{\bbR}e^{i2^jxt}\psi(t)dt = C \widehat{\psi}(2^{j}x)$ for some constant $C > 0$. Therefore, the property ($P6$) is a consequence of ($P3$). The property ($P4$) is a consequence of the harmonizable representation \eqref{e:OFBM_harmonizable}, and also of the conditions \eqref{e:N_psi}, \eqref{e:supp_psi=compact} and \eqref{e:psihat_is_slower_than_a_power_function}; the latter ensure that the integrand in \eqref{e:EW(j)_Fourier} is well-defined in $\bbR$. We now show ($P7$). Since $\psi \in \bbR$, then $\widehat{\psi}(-2^jx) = \overline{\widehat{\psi}(2^jx)}$. Thus, by a change-of-variables $y = -x$ over the integration domain $x < 0$ we can rewrite \eqref{e:EW(j)_Fourier} as
$$
EW(2^j) = C 2^j  \int^{\infty}_{0} x^{-D}  2 \Re(AA^*) x^{-D^*} \frac{|\widehat{\psi}(2^j x)|^2}{x^2} dx.
$$
Since $\textnormal{supp} \hspace{1mm}\widehat{\psi}(x) $ has positive Lebesgue measure, then \eqref{e:full-rank} yields $v^* EW(2^j) v > 0$, $v \in \bbC^n \backslash\{0\}$. $\Box$\\


\begin{lemma}\label{l:elementary_results}
Fix $j \in \bbN$. Under (OFBM1)--(4) and \eqref{e:a(nu)/J->infty}, let $EW(a(\nu)2^j)$ and $W(a(\nu)2^j)$ be as in Lemma \ref{l:EW(a(n)2j)_W(a(n)2j)_scaling}. Then, the following limits hold, as $\nu \rightarrow \infty$:
\begin{equation}\label{e:4(ac-b2)/a+c_4(ac-b2)/(a+c^2_p12_p22neq0}
\frac{4(ac-b^2)}{a+c} \sim \frac{4 \det EW(2^j)}{(p^{2}_{12}+p^{2}_{22})b_{22}(2^j)}a(\nu)^{2h_1}, \quad \frac{4(ac-b^2)}{(a+c)^2} \sim \frac{4 \det EW(2^j)}{((p^{2}_{12}+p^{2}_{22})b_{22}(2^j))^2}\frac{1}{a(\nu)^{2(h_2 - h_1)}}
\end{equation}
\begin{equation}\label{e:4(ac-b2)/a+c_4(ac-b2)/(a+c^2_p12_p22neq0_in_prob}
\frac{4(\widehat{a}\widehat{c}-\widehat{b}^2)}{\widehat{a}+\widehat{c}} \stackrel{P}\sim \frac{4 \det W(2^j)}{(p^2_{12}+p^{2}_{22})\widehat{b}_{22}(2^j)}a(\nu)^{2h_1}, \quad \frac{4(\widehat{a}\widehat{c}-\widehat{b}^2)}{(\widehat{a}+\widehat{c})^2} \stackrel{P}\sim \frac{4 \det W(2^j)}{((p^{2}_{12}+p^{2}_{22})\widehat{b}_{22}(2^j))^2}\frac{1}{a(\nu)^{2(h_2 - h_1)}},
\end{equation}
\begin{equation}\label{e:limit_derivative_squareroot}
\frac{\Big( 1 - \sqrt{1 - \frac{4(ac-b^2)}{(a+c)^2}}\Big)}{4(ac-b^2)/(a+c)^2} \rightarrow \frac{1}{2}, \quad
\frac{\Big( 1 - \sqrt{1 - \frac{4(\widehat{a}\widehat{c}-\widehat{b}^2)}{(\widehat{a}+\widehat{c})^2}}\Big)}{4(\widehat{a}\widehat{c}-\widehat{b}^2)/(\widehat{a}+\widehat{c})^2} \stackrel{P}\rightarrow \frac{1}{2},
\end{equation}
\begin{equation}\label{e:lambdaE1_asympt}
\lambda^{E}_1 \sim \frac{\det EW(2^j)}{(p^2_{12}+p^2_{22})b_{22}(2^j)}a(\nu)^{2h_1}, \quad \lambda_1 \stackrel{P}\sim \frac{\det W(2^j)}{(p^2_{12}+p^2_{22})\widehat{b}_{22}(2^j)}a(\nu)^{2h_1},
\end{equation}
\begin{equation}\label{e:lambdaE2_lambda2_asympt}
\lambda^{E}_2 \sim (p^{2}_{12}+p^2_{22})b_{22}(2^j)a(\nu)^{2h_2}, \quad \lambda_2 \stackrel{P}\sim (p^{2}_{12}+p^2_{22})\widehat{b}_{22}(2^j)a(\nu)^{2h_2}.
\end{equation}
\end{lemma}
\begin{proof}
In regard to $\lambda_1$ and $\lambda^{E}_1$, consider the relation \eqref{e:lambda1_explicit}. By Lemma \ref{l:EW(a(n)2j)_W(a(n)2j)_scaling},
$$
\frac{4(ac-b^2)}{a+c}
$$
\begin{equation}\label{e:4(ac-b^2)/a+c}
= \frac{4 a(\nu)^{2(h_1 + h_2)}\det EW(2^j)}{(p^{2}_{11}+p^{2}_{21})b_{11}a(\nu)^{2h_1} + 2 (p_{11}p_{12} + p_{21}p_{22})b_{12}a(\nu)^{h_1 + h_2}
+ (p^{2}_{12}+p^{2}_{22})b_{22}a(\nu)^{2h_2}},
\end{equation}
where the determinant is non-trivial due to the property $(P7)$. Again from Lemma \ref{l:EW(a(n)2j)_W(a(n)2j)_scaling} and by applying Theorem \ref{t:asymptotic_normality_wavecoef_fixed_scales}, an analogous expression holds for $4(\widehat{a}\widehat{c}-\widehat{b}^2)/(\widehat{a}+\widehat{c})$. Then, the expressions \eqref{e:4(ac-b2)/a+c_4(ac-b2)/(a+c^2_p12_p22neq0}, \eqref{e:4(ac-b2)/a+c_4(ac-b2)/(a+c^2_p12_p22neq0_in_prob}, \eqref{e:limit_derivative_squareroot}, \eqref{e:lambdaE1_asympt} follow.

An analogous reasoning applied to $\lambda_2$ and $\lambda^{E}_2$ leads to \eqref{e:lambdaE2_lambda2_asympt}, since the relation \eqref{e:lambda2_explicit}, Lemma \ref{l:EW(a(n)2j)_W(a(n)2j)_scaling} and Theorem \ref{t:asymptotic_normality_wavecoef_fixed_scales} show that
$\lambda^{E}_2 \sim a + c$, $\lambda_2 \stackrel{P}\sim \widehat{a} + \widehat{c}$. $\Box$\\
\end{proof}

\section{Asymptotic normality at fixed scales: proofs}\label{s:asympt_normality_fixed_scales}

This section contains the remaining proofs of the statements in Section \ref{s:wavelet_analysis}.\\

\noindent {\sc Proof of Proposition \ref{p:decay_Cov_wavelet_coefs}}: In view of \eqref{e:supp_psi=compact}, we can assume without loss of generality that
\begin{equation}\label{e:supp}
\textnormal{supp}(\psi) = [0,K], \quad K \geq 1.
\end{equation}
By \eqref{e:N_psi}, \eqref{e:time_revers}, \eqref{e:time_reversibility}, \eqref{e:wavelet_transform_after_change_var} and \eqref{e:supp}, the wavelet covariance can be reexpressed as
$$
E D(2^j,k)D(2^{j '},k')^* =  \int_{\bbR}\int_{\bbR}\psi(t)\psi(t') EB_H(2^j t+2^j k)B_H(2^{j'}t'+2^{j'}k')^*dtdt'
$$
\begin{equation}\label{e:wave_cov_reexpressed-1}
= -\frac{1}{2}\int_{\bbR}\int_{\bbR} \psi(t)\psi(t')
|2^j t- 2^{j'} t' + 2^j k - 2^{j'} k'|^{H}\Sigma |2^j t- 2^{j'} t' + 2^j k - 2^{j'} k'|^{H^*} dt dt'
\end{equation}
\begin{equation}\label{e:wave_cov_reexpressed}
= - \frac{1}{2} \hspace{1mm}|2^j k- 2^{j'}k'|^{H} \Big\{\int^K_{0}\int^K_{0}\psi(t)\psi(t') \hspace{1mm}\Big|1 + \frac{2^j t- 2^{j'}t'}{2^jk- 2^{j '}k'}\Big|^{H}\Sigma
\Big|1 + \frac{2^j t- 2^{j '}t'}{2^j k- 2^{j '}k'}\Big|^{H^*}dtdt'\Big\}
|2^j k- 2^{j'}k'|^{H^*}.
\end{equation}
Let
\begin{equation}\label{e:f(x)=fracquadform}
f(x) = x^{H}\Sigma x^{H^*} = \Big( f_{i_1,i_2}(x)\Big)_{i_1,i_2 = 1,\hdots,n}.
\end{equation}
Within the radius defined by \eqref{e:wave_param_range}, the integrand in \eqref{e:wave_cov_reexpressed} can be reexpressed analytically as
$$
\Big(1 + \frac{2^jt- 2^{j'}t'}{2^j k- 2^{j'}k'}\Big)^{H}\Sigma
\Big(1 + \frac{2^j t- 2^{j'}t'}{2^j k- 2^{j'}k'}\Big)^{H^*} = \sum^{\infty}_{r=0} \Big( \frac{f^{(r)}_{i_1,i_2}(1)}{r!} \Big( \frac{2^j t- 2^{j'}t'}{2^j k - 2^{j'}k'}\Big)^r \Big)_{i_1,i_2=1,\hdots,n}.
$$
Thus, again \eqref{e:N_psi} yields
$$
\int^K_{0}\int^K_{0}\psi(t)\psi(t') \hspace{1mm}\Big(1 + \frac{2^j t- 2^{j'}t'}{2^j k- 2^j k'}\Big)^{H}\Sigma\Big(1 + \frac{2^j t- 2^{j '}t'}{2^j k - 2^{j'}k'}\Big)^{H^*}dtdt'
$$
\begin{equation}\label{e:main_wavelet_transform_term_after_applying_Q_moments}
= \sum^{\infty}_{r=2N_{\psi}}\Big( \frac{f^{(r)}_{i_1, i_2}(1)}{r!}  \int^K_{0}\int^K_{0}\psi(t)\psi(t') \hspace{1mm}\Big( \frac{2^j t- 2^{j '}t'}{2^j k - 2^{j'}k'}\Big)^rdtdt'\Big)_{i_1,i_2 = 1,\hdots,n}.
\end{equation}
We now look at each term in the summation \eqref{e:main_wavelet_transform_term_after_applying_Q_moments}. By induction, up to the matrix constant $\Big(f^{(2N_{\psi}+l)}_{i_1,i_2}(1)/(2N_{\psi}+l)!\Big)_{i_1,i_2=1,\hdots,n}$ the term associated with the index $r = 2N_{\psi} + l$, $l \in \bbN$, can be written as
$$
\frac{1}{(2^{j}k - 2^{j'}k')^{2N_{\psi}}} \Big\{ \sum^{N_{\psi}+l}_{\nu=N_{\psi}} {2N_{\psi}+l \choose \nu} \frac{(2^{j})^{\nu}(-2^{j'})^{(2N_{\psi}+l)-\nu}}{(2^{j}k-2^{j'}k')^l}
\int^K_0 \int^K_0 \psi(t) \psi(t') t^{\nu}t'^{(2N_{\psi}+l)-\nu}dtdt'\Big\}
$$
$$
\stackrel{\tau = \nu - N_{\psi}}= \frac{(2^{j+j'})^{N_{\psi}}}{(2^{j}k - 2^{j'}k')^{2N_{\psi}}}
$$
\begin{equation}\label{e:sum_doubleinteg_psi_resid_after_induction}
\Big\{ \sum^{l}_{\tau=0} {2N_{\psi}+l \choose \tau + N_{\psi}} \frac{(2^{j})^{\tau}(-1)^{N_{\psi}+l-\tau}(2^{j'})^{l-\tau}}{(2^{j}k-2^{j'}k')^l}
\int^K_0 \int^K_0 \psi(t) \psi(t') t^{\tau + N_{\psi}}t'^{N_{\psi}+l-\tau}dtdt'\Big\}.
\end{equation}
In the expression \eqref{e:sum_doubleinteg_psi_resid_after_induction}, for any $l \in \bbN$ the largest binomial coefficient is
\begin{equation}\label{e:binom_coef_after_Stirling}
{2N_{\psi}+l \choose \lfloor l/2 \rfloor + N_{\psi}} \sim C \frac{2^l}{\sqrt{l}},
\end{equation}
where the asymptotic equivalence (as $l \rightarrow \infty$) follows from Stirling's formula. Moreover,
\begin{equation}\label{e:double_wavelet_integ_bound}
\Big|\int^K_0 \int^K_0 \psi(t) \psi(t') t^{\tau + N_{\psi}}t'^{N_{\psi}+l-\tau}dtdt' \Big| \leq C \frac{K^{2N_{\psi}+l+2}}{l}.
\end{equation}
Recall that we are taking the wavelet parameters in the range \eqref{e:wave_param_range}. Then, by \eqref{e:binom_coef_after_Stirling}, \eqref{e:double_wavelet_integ_bound} and Lemma \ref{l:fracquadform_deriv_bound} below, the expression \eqref{e:main_wavelet_transform_term_after_applying_Q_moments} is bounded from above by
$$
\frac{\|f^{(2N_{\psi})}(1)\|}{(2N_{\psi})!}
{2N_{\psi} \choose N_{\psi}}\int^{K}_{0}\int^{K}_{0}|\psi(t)\psi(t') | t^{N_{\psi}}(t')^{N_{\psi}}dt dt'
$$
$$
+
\sum^{\infty}_{l=1} \frac{\|f^{(2N_{\psi}+l)}(1)\|}{(2N_{\psi}+l)!} \Big| \frac{\max\{2^j,2^{j'}\}}{2^j k - 2^{j'}k'}\Big|^l \sum^{l}_{\tau=0}
{2N_{\psi} +l \choose \tau + N_{\psi}}\int^{K}_{0}\int^{K}_{0}|\psi(t)\psi(t') | t^{\tau + N_{\psi}}(t')^{N_{\psi}+l - \tau}dt dt'
$$
$$
\leq C + C' \sum^{\infty}_{l=1} \frac{1}{l}\Big( \frac{\delta}{K}\Big)^{l} \frac{2^l}{\sqrt{l}} \Big\{(l+1) \frac{K^{2N_{\psi}+l+2}}{l}\Big\}.
$$
Therefore, in view of \eqref{e:wave_param_range}, the series \eqref{e:main_wavelet_transform_term_after_applying_Q_moments} is, indeed, summable, and by \eqref{e:wave_cov_reexpressed} and \eqref{e:sum_doubleinteg_psi_resid_after_induction} the expression \eqref{e:decay_Cov_wavelet_coefs} holds.

In regard to \eqref{e:decay_Cov_wavelet_coefs_in_norm}, note that for $c \geq \delta^{-1}K\max\{2^j,2^{j'}\}$ we can construct the bound for the matrix exponential
$$
\|c^{H}\| \leq C \hspace{1mm} \|P\| \|P^{-1}\| \|c^{J_H}\|_{l^1} \leq C' \hspace{1mm}c^{\max_{h \in \textnormal{eig}(H)}\Re(h)}  \hspace{1mm}  |\log^{\kappa}(c)|.
$$
Now set $c = |2^j k - 2^{j'}k'|$ under the restriction \eqref{e:wave_param_range}. $\Box$\\

\begin{lemma}\label{l:fracquadform_deriv_bound}
Let $f$ be as in \eqref{e:f(x)=fracquadform}. Then,
\begin{equation}\label{e:fracquadform_deriv_bound}
\|f^{(r)}(1)\|/(r-1)! = O(1), \quad r \in \bbN.
\end{equation}
\end{lemma}
\begin{proof}
To prove \eqref{e:fracquadform_deriv_bound}, we will first establish the formula
\begin{equation}\label{e:f^(q)(x)}
f^{(r)}(x) = \sum^{r}_{k=0} {r \choose k} \Big\{\prod^{(r-1)-k}_{j_1=0}(H-j_1 I)\Big\}x^{H-(r-k)I}\Sigma x^{H^*-kI}\Big\{\prod^{k-1}_{j_2=0}(H^* -j_2 I)\Big\},
\end{equation}
where, for notational convenience, we set $\prod^{-1}_{j=0}(H-jI)\equiv I$. For $r = 1$, the formula holds. So, by induction, assume that it also holds for $r \in \bbN$. Then,
$$
f^{(r+1)}(x) = \sum^{r}_{k=0} {r \choose k} \Big\{\prod^{(r-1)-k}_{j_1=0}(H-j_1 I)\Big\}\{(H-(r-k)I)x^{H-(r-k)I-I}\Sigma x^{H^*-kI}
$$
$$
+x^{H-(r-k)I}\Sigma x^{H^*-kI-I}(H^*-kI)\}
\Big\{\prod^{k-1}_{j_2=0}(H^*-j_2 I)\Big\}
$$
$$
= {r+1 \choose 0} \Big\{\prod^{r}_{j=0}(H-j I)\Big\} x^{H-(r+1)I}\Sigma x^{H^*}
$$
$$
+ \sum^{r}_{k=1} \Big[ {r \choose k} + {r \choose k-1} \Big] \Big\{\prod^{r-k}_{j_1=0}(H-j_1 I)\Big\} x^{H-((r+1) -k)I}\Sigma x^{H^*-kI}\Big\{\prod^{k-1}_{j_2=0}(H^*-j_2 I)\Big\}
$$
$$
+ {r+1 \choose r+1} x^{H}\Sigma x^{H^*-(r+1)I}\Big\{\prod^{r}_{j=0}(H^*-j I)\Big\}
$$
$$
= \sum^{r+1}_{k=0} {r+1 \choose k} \Big\{\prod^{r-k}_{j_1=0}(H-j_1 I)\Big\}x^{H-((r+1)-k)I}\Sigma x^{H^*-kI}\Big\{\prod^{k-1}_{j_2=0}(H^*-j_2 I)\Big\}.
$$
This establishes \eqref{e:f^(q)(x)}.

Recall the notation \eqref{e:eigen-assumption} for the eigenvalues of $H$, and let $h_{\min} := \min \Re \hspace{0.5mm} \textnormal{eig}(H) \in (0,1)$. For any $J \in \bbN \cup \{0\}$, the equivalence of matrix norms yields the bound
\begin{equation}\label{e:prod_(H-jI)_norm_bound}
\Big\|\prod^{J}_{j=0}(H-jI) \Big\| \leq C \prod^{J}_{j=0}\|J_H-jI\|_{l^1} = C \prod^{J}_{j=0}\Big\{\sum^{n}_{k=1}|j - h_k| + c \Big\} \leq C'\prod^{J}_{j=0}|j - h_{\min}|,
\end{equation}
where $c \geq 0$ accounts for the number of off-diagonal 1s in the Jordan form $J_H$. Now set $x = 1$ in the expression \eqref{e:f^(q)(x)}. By \eqref{e:prod_(H-jI)_norm_bound},
$$
\|f^{(r)}(1)\| = |(-1)^{r}| \Big\| \sum^{r}_{k=0} {r \choose k} \Big\{\prod^{(r-1)-k}_{j_1=0}(j_1 I - H)\Big\}\Sigma \Big\{\prod^{k-1}_{j_2=0}(j_2 I- H^*)\Big\} \Big\|
$$
$$
\leq C \sum^{r}_{k=0} {r \choose k} \Big\{\prod^{(r-1)-k}_{j_1=0}|j_1 - h_{\min}|\Big\}\Big\{\prod^{k-1}_{j_2=0}|j_2 - h_{\min}|\Big\} = C |f^{(r)}_{h_{\min}}(1)|,
$$
where $f_{h_{\min}}(x)$ denotes the function $f$ when $n = 1$ (scalar), the Hurst parameter is $h_{\min}$ and $\Sigma = 1$. Since $|f^{(r)}_{h_{\min}}(1)| = 2 h_{\min} \prod^{r-1}_{j=1} (j - 2 h_{\min})$, then \eqref{e:fracquadform_deriv_bound} holds. $\Box$\\
\end{proof}

For the next results, we assume that $a_{j}, a_{j'} \in \bbN$.

Let $\{\phi_{k}\}_{k \in \bbZ}$ be a sequence of real numbers. We are interested in calculating the limit
\begin{equation}\label{e:average_double_sum_Phi}
\lim_{\nu \rightarrow \infty}\frac{1}{\nu}\sum^{K_j}_{k=1}\sum^{K_{j'}}_{k'=1} \phi_{a_{j}k - a_{j'}k'}.
\end{equation}
Let
\begin{equation}\label{e:R=indices}
{\mathcal R} := \{a_j \bbN - a_{j'}\bbN\}
\end{equation}
be the set of indices of $\phi_{\cdot}$. The next lemma provides the growth rate of the number of terms in the summation \eqref{e:average_double_sum_Phi}. Before we state and show it, note that to every $r \in {\mathcal R}$ we can associate the affine level curve
\begin{equation}\label{e:k'}
k'(k) = \frac{a_j}{a_{j'}}k - \frac{r}{a_{j'}} , \quad k \in \bbR,
\end{equation}
which contains all the pairs $(k,k')\in \bbN^2$ satisfying $a_j k - a_{j'}k' = r$ in the chosen range.
\begin{lemma}\label{l:increase_in_number_of_pairs}
Let $a_j, a_{j'},\nu \in \bbN$ and $r \in {\mathcal R}$ (see \eqref{e:R=indices}). Let $k'(\cdot)$ be the function defined in \eqref{e:k'}. Then, the number $\xi_r(\nu)$ of pairs $(k,k') \in \bbZ^2$, $1 \leq k \leq a_{j'}\nu$, $1 \leq k' \leq a_{j}\nu$, in the affine level curve \eqref{e:k'} associated with $r$ satisfies
\begin{equation}\label{e:increase_in_number_of_pairs}
\lim_{\nu \rightarrow \infty} \frac{\xi_r(\nu)}{\nu} = \gcd(a_j,a_{j'}).
\end{equation}
\end{lemma}
\begin{proof}
We first show that $a_j k - a_{j'}k' = r$ has a solution $(k_*,k'(k_*)) \in \bbZ^2$, and that the set ${\mathcal A}$ of such solutions has the form
\begin{equation}\label{e:k**=k*+aj'N_set_of_solutions}
{\mathcal A} = \Big\{(k_{**},k'(k_{**})) \in \bbZ^2: k_{**} = k_* + \frac{a_{j'}}{\gcd(a_j,a_{j'})} \bbZ \Big\}.
\end{equation}
By Theorem 1.8 in Jones and Jones \cite{jones:jones:1998}, p.10, we can conveniently reexpress the set ${\mathcal R}$ in \eqref{e:R=indices} as
\begin{equation}\label{e:R=gcd(aj,aj')Z}
{\mathcal R} = \gcd(a_j,a_{j'}) \bbZ.
\end{equation}
In other words, for any fixed $w \in \bbZ$, there is a solution pair $(k_*,k'(k_*)) \in \bbN^2$ for the equation
\begin{equation}\label{e:ajk-aj'k'=wgcd(aj,aj')}
a_j k - a_{j'}k' = \gcd(a_j,a_{j'}) w.
\end{equation}
For notational simplicity, rewrite the solution pair as $(k_*,k'_*)$, and consider another solution $(k_{**},k'_{**}) \in \bbZ^2$. By plugging the two solution pairs into \eqref{e:ajk-aj'k'=wgcd(aj,aj')} we obtain
\begin{equation}\label{e:k'**-k'*=aj/aj'(k**-k*)}
k'_{**} - k'_* = \frac{a_j /\gcd(a_j,a_{j'})}{a_{j'}/\gcd(a_j,a_{j'})} (k_{**} - k_*).
\end{equation}
By premultiplying \eqref{e:k'**-k'*=aj/aj'(k**-k*)} by $-1$ if necessary, we can assume that $k_{**} - k_*, k'_{**} - k'_*\in \bbN$. Since $a_j/\gcd(a_j,a_{j'}), a_{j'}/\gcd(a_j,a_{j'})$ are coprime, then $k_{**} - k_*= \frac{a_{j'}}{\gcd(a_j,a_{j'})} z$ for some $z \in \bbN$. Therefore, $(k_{**},k'_{**})$ has the form described on the right-hand side of \eqref{e:k**=k*+aj'N_set_of_solutions}. Conversely, for a given solution $(k_*,k'(k_*))$, every pair $(k_{**},k'(k_{**}))$ of the form expressed on the right-hand side of \eqref{e:k**=k*+aj'N_set_of_solutions} is also a solution (in $\bbZ^2$) for \eqref{e:ajk-aj'k'=wgcd(aj,aj')}. This establishes \eqref{e:k**=k*+aj'N_set_of_solutions}.

To show \eqref{e:increase_in_number_of_pairs}, for any sufficiently large $\nu$, let $k_0 \in \{1,\hdots,a_{j'}\nu\}$ be the smallest number such that $(k_0,k'(k_0)) \in \bbN^2$ solves \eqref{e:ajk-aj'k'=wgcd(aj,aj')}. Let $\bbR \ni x = \gcd(a_j,a_{j'}) (\nu - k_0/a_{j'})$. Then, by \eqref{e:k**=k*+aj'N_set_of_solutions}, $(\lfloor x \rfloor,k'(\lfloor x \rfloor) )$ is the rightmost solution for \eqref{e:ajk-aj'k'=wgcd(aj,aj')} within the first-entry range $k = 1,\hdots,a_{j'}\nu$. Moreover,
$\lfloor x \rfloor \nu^{-1} \rightarrow \gcd(a_j,a_{j'})$, $\nu \rightarrow \infty$. This establishes \eqref{e:increase_in_number_of_pairs}. $\Box$\\
\end{proof}

\begin{lemma}\label{l:convergence_average_ajk-aj'k'}
Let $\{\phi_{\cdot}\} \in \bbR$ be a sequence such that $\sum^{\infty}_{z = -\infty} |\phi_{z \gcd(a_j,a_{j'})}| < \infty$.
Then,
\begin{equation}\label{e:convergence_average_ajk-aj'k'}
\frac{1}{\nu}\sum^{a_{j'}\nu}_{k=1}\sum^{a_{j}\nu}_{k'=1}\phi_{a_{j}k - a_{j'}k'}\rightarrow \gcd(a_{j},a_{j'}) \sum^{\infty}_{z = - \infty} \phi_{z \gcd(a_j,a_{j'})}, \quad \nu \rightarrow \infty.
\end{equation}
\end{lemma}
\begin{proof}
By expression \eqref{e:R=gcd(aj,aj')Z} in the proof of Lemma \ref{l:increase_in_number_of_pairs}, $\nu^{-1}\sum^{a_{j'}\nu}_{k=1}\sum^{a_{j}\nu}_{k'=1}\phi_{a_{j}k - a_{j'}k'}= \hspace{1mm}\nu^{-1}\sum_{r \in {\mathcal R} \cap B_{j,j'}(\nu)} \xi_{r}(\nu) \phi_{r}$, where $B_{j,j'}(\nu)$ is the range for $r$ such that $1 \leq k \leq a_{j'}\nu$ and $1 \leq k' \leq a_{j}\nu$. Moreover, by \eqref{e:R=gcd(aj,aj')Z}, $\cup_{\nu \in \bbN}B_{j,j'}(\nu) = \gcd(a_j,a_{j'})\bbZ$. Then, expression \eqref{e:convergence_average_ajk-aj'k'} is a consequence of dominated convergence and \eqref{e:increase_in_number_of_pairs}. $\Box$\\
\end{proof}

\noindent {\sc Proof of Proposition \ref{p:4th_moments_wavecoef}}: The statement ($ii$) is a direct consequence of ($i$), so we only prove the latter. It suffices to consider the subsequence $\nu = 2^j 2^{j'} \nu_*$, $\nu_* \rightarrow \infty$. Then, $K_{j} = 2^{j'}\nu^*$, $K_{j'} = 2^{j}\nu_*$, and $\sqrt{K_{j}}\sqrt{K_{j'}} K^{-1}_{j} K^{-1}_{j'} = 2^{-(j+j')/2}\hspace{0.5mm}\nu^{-1}_*$. From the expression \eqref{e:wave_cov_reexpressed-1},
$$
\textnormal{Cov}(D(2^j,k),D(2^{j'},k')) =: \Phi_{2^{j}k - 2^{j'}k'}.
$$
Let $\Xi_{2^{j}k - 2^{j'}k'} = \Phi_{2^{j}k - 2^{j'}k'} \otimes \Phi_{2^{j}k - 2^{j'}k'}$. By \eqref{e:R=gcd(aj,aj')Z}, the range of indices spanned by $2^j k - 2^{j'}k'$ is $\bbZ \hspace{0.5mm}\gcd(2^j,2^{j'})$. Thus, we would like to show that
\begin{equation}\label{e:sum_norms_Kron_wavecoef_is_abssummable}
\sum^{\infty}_{z = - \infty} \| \Xi_{z \gcd(2^j,2^{j'})  }\| < \infty.
\end{equation}
Since the $l^1$ matrix norm (see \eqref{e:|A|p}) is sub-multiplicative,
$
\|\Xi_{z \gcd(2^j,2^{j'})  } \|_{l^1} =
\| \textnormal{vec}(\Phi_{z \gcd(2^j,2^{j'})  }) \textnormal{vec}(\Phi_{z \gcd(2^j,2^{j'})  })^* \|_{l^1}
\leq \|\textnormal{vec}(\Phi_{z \gcd(2^j,2^{j'})  })\|^2_{l^1}.
$
Moreover, Proposition \ref{p:decay_Cov_wavelet_coefs} implies that the sequence $\{\norm{ \textnormal{vec}(\Phi_{z \gcd(2^j,2^{j'})  })}_{l^1}\}_{z \in \bbZ}$ is summable (with $N_{\psi} \geq 2$; see \eqref{e:N_psi}); hence, \eqref{e:sum_norms_Kron_wavecoef_is_abssummable} holds. The expression (\ref{e:limiting_kron}) is now a consequence of Lemma \ref{l:convergence_average_ajk-aj'k'}. $\Box$\\

\noindent {\sc Proof of Theorem \ref{t:asymptotic_normality_wavecoef_fixed_scales}}
The argument is reminiscent of those in Bardet \cite{bardet:2000}, pp.510-513, Bardet \cite{bardet:2002}, p.997, and Istas and Lang \cite{istas:lang:1997}, Lemma 2. For notational simplicity, we will restrict ourselves to the bivariate context ($n = 2$). The argument for general $n$ can be worked out by a simple adaptation.

The proof is by means of the Cram\'{e}r-Wold device. Form the vector of wavelet coefficients
$$
Y = (d_1(2^{j_1},1),d_2(2^{j_1},1), \hdots, d_1(2^{j_1},K_{j_1}),d_2(2^{j_1},K_{j_1});\hdots;
$$
$$
d_1(2^{j_m},1),d_2(2^{j_m},1), \hdots, d_1(2^{j_m},K_{j_m}),d_2(2^{j_m},K_{j_m})) \in \bbR^{\Upsilon(\nu)},
$$
where $\Upsilon(\nu) := 2 \hspace{1mm}\sum^{j_m}_{j=j_1}K_j $. Notice that $m$, $j_1, \hdots, j_m$ are fixed, but each $K_j$ goes to infinity with $\nu$. Let
\begin{equation}\label{e:alphavec}
{\boldsymbol \alpha} = ({\boldsymbol \alpha}_{j_1},\hdots,{\boldsymbol \alpha}_{j_m}) \in \bbR^{3m},
\end{equation}
where
$$
{\boldsymbol \alpha}_{j} = (\alpha_{j,1},\alpha_{j,12},\alpha_{j,2})^* \in \bbR^3, \quad j = j_1,\hdots,j_m.
$$
Now form the block-diagonal matrix
\begin{equation}\label{e:block-diagonal_matrix_D}
D = \textnormal{diag}\Big( \underbrace{\frac{1}{K_{j_1}}\sqrt{\frac{1}{2^{j_1}}}\Omega_{j_1}, \hdots, \frac{1}{K_{j_1}}\sqrt{\frac{1}{2^{j_1}}}\Omega_{j_1}}_\text{$K_{j_1}$}; \hdots;
\underbrace{\frac{1}{K_{j_m}}\sqrt{\frac{1}{2^{j_m}}}\Omega_{j_m}, \hdots, \frac{1}{K_{j_m}}\sqrt{\frac{1}{2^{j_m}}}\Omega_{j_m}}_\text{$K_{j_m}$}\Big),
\end{equation}
where
$$
\Omega_{j}
= \left(\begin{array}{cc}
\alpha_{j,1} & \alpha_{j,12}/2\\
\alpha_{j,12}/2 & \alpha_{j,2}
\end{array}\right), \quad j = j_1,\hdots, j_m.
$$
We would like to establish the limiting distribution of the statistic
$$
T_{\nu} = \sum^{j_m}_{j=j_1} \frac{{\boldsymbol \alpha}^*_{j}}{\sqrt{2^j}} \vecoper_{{\mathcal S}}W(2^j) = Y^* D Y,
$$
where it suffices to consider ${\boldsymbol \alpha}$ in \eqref{e:alphavec} such that
\begin{equation}\label{e:only_alphas_for_which_the_limit_is_nontrivial}
{\boldsymbol \alpha}^* \Sigma(H,AA^*) {\boldsymbol \alpha} > 0
\end{equation}
(see Brockwell and Davis \cite{brockwell:davis:1991}, pp.\ 211 and 214). The matrix $\Sigma(H,AA^*)$ in \eqref{e:only_alphas_for_which_the_limit_is_nontrivial} is obtained from Proposition \ref{p:4th_moments_wavecoef}, and can be written in block form as
\begin{equation}\label{e:Sigma(H,AA*)_in_block_form}
\Sigma(H,AA^*) = (G_{jj'})_{j,j'=j_1,\hdots,j_m},
\end{equation}
corresponding to block entries of the vector ${\boldsymbol \alpha} = (\alpha_{j_1},\hdots,\alpha_{j_m})^*$. Let $\Gamma = \Cov(Y,Y)$ and consider the spectral decomposition $\Gamma^{1/2}D\Gamma^{1/2} = O \Lambda O^*$, where $\Lambda$ is diagonal with real, and not necessarily positive, eigenvalues $\lambda$ and $O$ is an orthogonal matrix. Now let $Z \sim N(0,I_{\Upsilon(\nu)})$. Then,
$$
T_{\nu} \stackrel{d}= Z^* \Gamma^{1/2} D \Gamma^{1/2} Z = Z^* O \Lambda  O^* Z \stackrel{d}=Z^* \Lambda Z =: \sum^{\Upsilon(\nu)}_{i=1}\lambda_i(\nu) Z^{2}_i.
$$
Assume for the moment that
\begin{equation}\label{e:max_eig=o(1/(2^J/2))}
\max_{i=1,\hdots,\Upsilon(\nu)}|\lambda_i(\nu)| = o\Big(\frac{1}{\nu^{1/2}} \Big).
\end{equation}
By \eqref{e:only_alphas_for_which_the_limit_is_nontrivial} and Proposition \ref{p:4th_moments_wavecoef},
$$
\nu \Var(T_{\nu}) = \sum^{j_m}_{j=j_1}\sum^{j_m}_{j'=j_1}{\boldsymbol \alpha}^*_{j} \Big\{ \sqrt{\frac{\nu}{2^j}} \sqrt{\frac{\nu}{2^{j'}}}\Cov(\vecoper_{{\mathcal S}}W(2^j),
\vecoper_{{\mathcal S}}W(2^{j'})) \Big\}  {\boldsymbol \alpha}_{j'}
\rightarrow \sum^{j_m}_{j=j_1}\sum^{j_m}_{j'=j_1}{\boldsymbol \alpha}^*_{j} \hspace{1mm}G_{jj'} \hspace{1mm}  {\boldsymbol \alpha}_{j'} > 0.
$$
Therefore, there exists a constant $C > 0$ such that, for large enough $\nu$, $\nu \Var(T_{\nu}) \geq C > 0$. In view of condition \eqref{e:max_eig=o(1/(2^J/2))},
$$
\frac{\max_{i=1,\hdots,\Upsilon(\nu)}|\lambda_{i}(\nu)|}{\sqrt{\Var(T_{\nu})}} \leq C' \nu^{1/2} \max_{i=1,\hdots,\Upsilon(\nu)}|\lambda_{i}(\nu)| \rightarrow 0, \quad \nu \rightarrow \infty.
$$
The claim \eqref{e:asymptotic_normality_wavecoef_fixed_scales} is now a consequence of Lemma \ref{l:Vn/sigma_n->N(0,1)}.

So, we need to show \eqref{e:max_eig=o(1/(2^J/2))}. The first step is to establish the bound
\begin{equation}\label{e:maxeig_Gamma1/2DGamma1/2=<maxblockD*maxeigGamma}
\sup_{\textbf{u} \in S^{\Upsilon(\nu) - 1}} |\textbf{u}^* \Gamma^{1/2} D \Gamma^{1/2}\textbf{u} | \leq C
\max_{j=j_1,\hdots,j_m}\frac{1}{K_{j}} \|\Omega_{j}\|_{l^1} \hspace{1mm} \sup_{\textbf{u} \in S^{\Upsilon(\nu) - 1}} \textbf{u}^* \Gamma \textbf{u}.
\end{equation}
Let $\textbf{u} \in S^{\Upsilon(\nu) - 1}$ and let $\textbf{v} =  \Gamma^{1/2}\textbf{u} $. We can break up the vector $\textbf{v}$ into two-dimensional subvectors $v_{\cdot,\cdot}$ to reexpress
$$
\textbf{v} = (v_{j_1,1}, \hdots, v_{j_1,K_{j_1}}; \hdots; v_{j_m,1}, \hdots, v_{j_m,K_{j_m}})^*.
$$
Based on the block-diagonal structure of $D$ expressed in \eqref{e:block-diagonal_matrix_D},
$$
|\textbf{u}^* \Gamma^{1/2}D\Gamma^{1/2}\textbf{u}| =| \textbf{v}^* D \textbf{v} |
=  \Big|\sum^{j_m}_{j=j_1}\sum^{K_j}_{l=1}v^*_{j,l} \hspace{1mm}\frac{\Omega_{j}}{K_j \sqrt{2^j}} \hspace{1mm} v_{j,l}  \Big|
\leq C \sum^{j_m}_{j=j_1}\sum^{K_j}_{l=1} \frac{1}{K_j \sqrt{2^j}}\|\Omega_{j}\|_{l^1} \hspace{1mm} \|v_{j,l}\|^2
$$
\begin{equation}\label{e:maxeig_Gamma1/2DGamma1/2=<maxblockD*maxeigGamma_proof}
\leq C \Big( \max_{j=j_1,\hdots,j_m}\frac{1}{K_j\sqrt{2^j}}\|\Omega_{j}\|_{l^1} \Big)\hspace{1mm}  \sum^{j_m}_{j=j_1}\sum^{K_j}_{l=1} \|v_{j,l}\|^2
= C \Big( \max_{j=j_1,\hdots,j_m}\frac{1}{K_j \sqrt{2^j}}\|\Omega_{j}\|_{l^1} \Big) \textbf{u}^* \Gamma \textbf{u},
\end{equation}
where the constant $C$ comes from a change of matrix norms (see \eqref{e:|A|p}) and only depends on the fixed dimension $n=2$. By taking $\sup_{\textbf{u} \in S^{\Upsilon(\nu) - 1}}$ on both sides of \eqref{e:maxeig_Gamma1/2DGamma1/2=<maxblockD*maxeigGamma_proof}, we arrive at \eqref{e:maxeig_Gamma1/2DGamma1/2=<maxblockD*maxeigGamma}.

The second step towards showing \eqref{e:max_eig=o(1/(2^J/2))} consists of analyzing the asymptotic behavior of the right-hand side of \eqref{e:maxeig_Gamma1/2DGamma1/2=<maxblockD*maxeigGamma}, as $\nu \rightarrow \infty$. Note that
\begin{equation}\label{e:max1/K_asympt_1/2J}
\max_{j=j_1,\hdots,j_m}\frac{1}{K_j} \|\Omega_{j}\|_{l^1}  \sim C \frac{1}{\nu}, \quad \nu \rightarrow \infty.
\end{equation}
In view of \eqref{e:max1/K_asympt_1/2J}, for establishing \eqref{e:max_eig=o(1/(2^J/2))} it suffices to show that $\sup_{\textbf{u} \in S^{\Upsilon(\nu) - 1}} \textbf{u}^* \Gamma \textbf{u}$ is bounded. So, rewrite $Y = ( \hspace{0.5mm}Y_i \hspace{0.5mm})_{i=1,\hdots,\Upsilon(\nu)}$. Since
\begin{equation}\label{e:maxeigGamma=<sup_sum|Cov|}
\sup_{\textbf{u} \in S^{\Upsilon(\nu) - 1}} \textbf{u}^* \Gamma \textbf{u} \leq \max_{i_1 = 1,\hdots, \Upsilon(\nu)}\sum^{\Upsilon(\nu)}_{i_2 = 1}|\Cov(Y_{i_1},Y_{i_2})|
\end{equation}
(see Lemma 1 in Bardet \cite{bardet:2000}, p.509), we only need to show that the right-hand side of \eqref{e:maxeigGamma=<sup_sum|Cov|} is bounded. Fix $0 < \delta < 1/2$ and without loss of generality assume that $\textnormal{length}(K) = 1$. We turn back to wavelet and dimensionality parameters (indices) to reexpress, and then bound, the right-hand side of \eqref{e:maxeigGamma=<sup_sum|Cov|} as
$$
\max_{j=j_1,\hdots,j_m} \max_{k=1,\hdots,K_j}\max_{i=1,2} \sum^{j_m}_{j'=j_1}\sum^{K_{j'}}_{k'=1}\sum^{2}_{i'=1} |\Cov(d_i(2^j,k),d_{i'}(2^{j'},k'))|
$$
$$
\leq 2m \max_{j,j'=j_1,\hdots,j_m} \max_{k=1,\hdots,K_j}\max_{i,i'=1,2} \sum^{K_{j'}}_{k'=1} |\Cov(d_i(2^j,k),d_{i'}(2^{j'},k'))|
$$
$$
= 2m \max_{j,j'=j_1,\hdots,j_m} \max_{k=1,\hdots,K_j}\max_{i,i'=1,2}
$$
\begin{equation}\label{e:cov_terms=distant_apart+close_together}
\sum^{K_{j'}}_{k'=1} 1_{ \Big\{\frac{\max\{2^j,2^{j'}\}}{|2^j k - 2^{j'}k'|} > \delta \Big\}}+ 1_{ \Big\{\frac{\max\{2^j,2^{j'}\}}{|2^j k - 2^{j'}k'|}\leq \delta \Big\}}|\Cov(d_i(2^j,k),d_{i'}(2^{j'},k'))|.
\end{equation}
One can interpret the bound in \eqref{e:cov_terms=distant_apart+close_together} as dividing up the covariance terms into those that, parameter-wise, are either far apart or close together (see \eqref{e:wave_param_range}).

We now develop a bound for the first summation term in \eqref{e:cov_terms=distant_apart+close_together}. By the property $(P2)$,
$$
|\Cov(d_{i}(2^j,k),d_{i'}(2^{j'},k'))| \leq \max_{i,i'=1,2}\max_{j,j'=j_1,\hdots,j_m}\sqrt{\Var \hspace{1mm}d_{i}(2^j,0)}\sqrt{\Var \hspace{1mm}d_{i'}(2^{j'},0)} \leq C,
$$
i.e., the covariance terms have a common bound. Moreover, in regard to the associated indicators in \eqref{e:cov_terms=distant_apart+close_together}, when $j' \geq j$, $\#\{k': |2^{j-j'}k-k' | < \delta^{-1} \} \leq 2 \delta^{-1} + 1$. Alternatively, when $j' < j$, $\#\{k': |2^{j-j'}k-k' | < 2^{j-j'} \delta^{-1} \} \leq 2 \hspace{1mm}2^{j-j'} \delta^{-1} + 1$. Therefore, the first summation term in \eqref{e:cov_terms=distant_apart+close_together} comprises finitely many terms and is bounded by a constant, irrespective of $\nu$.

To bound the second summation term in \eqref{e:cov_terms=distant_apart+close_together}, since $|\Cov(d_i(2^j,k),d_{i'}(2^{j'},k'))| \leq \|\Cov(D(2^j,k),D(2^{j'},k'))\|_{l^1}$, the bound \eqref{e:decay_Cov_wavelet_coefs_in_norm} implies that
$$
\sum^{K_{j'}}_{k'=1} 1_{ \Big\{\frac{\max\{2^j,2^{j'}\}}{|2^j k - 2^{j'}k'|}\leq \delta \Big\}} |\Cov(d_i(2^j,k),d_{i'}(2^{j'},k'))|
$$
\begin{equation}\label{e:cov_terms_close_together_converge}
\leq C \sum^{K_{j'}}_{k'=1} 1_{ \Big\{\frac{\max\{2^j,2^{j'}\}}{|2^j k - 2^{j'}k'|}\leq \delta \Big\}} \frac{|\log^{\kappa} |2^j k- 2^{j'}k'||}{|2^j k- 2^{j'}k'|^{2N_{\psi}-2}}\hspace{1mm}
\leq C \sum_{z \neq 0} \frac{1}{|z|^{2N_{\psi}-2}}\hspace{1mm}
\log^{\kappa} |z| < \infty,
\end{equation}
where $2N_{\psi} - 2 > 1$ by \eqref{e:N_psi} and $C$ does not depend on $k$, $k'$. Consequently, \eqref{e:cov_terms=distant_apart+close_together} is bounded, and so is \eqref{e:maxeigGamma=<sup_sum|Cov|}, as we wished to show. This establishes \eqref{e:max_eig=o(1/(2^J/2))}, and as a result, also \eqref{e:asymptotic_normality_wavecoef_fixed_scales}. $\Box$\\

The next lemma, stated without proof, is used in the proof of Theorem \ref{t:asymptotic_normality_wavecoef_fixed_scales}. It is a simple consequence of the Lindeberg central limit theorem.
\begin{lemma}\label{l:Vn/sigma_n->N(0,1)}
Let $\{W_{j,n}\}_{j=1,\hdots,n}$, $n \in \bbN$, be an array of i.i.d.\ random variables such that $EW_{j,k} = 0$ and $EW^2_{j,k} < \infty$. Let $\{\lambda_{j,n}\}_{j=1,\hdots,n}$, $n \in \bbN$, be an associated array of constants $\lambda_{j,n} \in \bbR$. Define the statistic
$V_{n} = \sum^{n}_{j=1}\lambda_{j,n} W_{j,n}$ and its variance $\sigma^2_n = \Var(V_n)$. If
$\max_{j=1,\hdots,n}|\lambda_{j,n}| = o(\sigma_n)$, then $\frac{V_{n}}{\sigma_n} \stackrel{d}\rightarrow N(0,1)$.
\end{lemma}

\section{Estimation based on the discretized wavelet transform}\label{s:discretized_wavelet}

In this section, we address the issue raised in Remark \ref{r:discretized_wavelet_transform}. Instead of a continuous time OFBM path $\{B_H(t)\}_{t \in \bbR}$,  we assume that only a discrete OFBM sample
\begin{equation}\label{e:BH(k)}
\{B_H(k)\}_{k \in \bbZ}
\end{equation}
is available. The main goal is to develop the asymptotic distribution of the Hurst matrix estimators starting from the so-called discretized wavelet coefficients (as defined in \eqref{e:approxD_2j,k} below). The mathematical construction in this section is carried out along the same lines of that in Stoev et al.\ \cite{stoev:pipiras:taqqu:2002} (notably Proposition 2.4 and Theorem 3.2).\\

We suppose the wavelet approximation coefficients stem from Mallat's pyramidal algorithm, under a multiresolution analysis of $L^2(\bbR)$ (MRA; see Mallat \cite{mallat:1999}, chapter 7). Accordingly, we need to replace ($W2$) with the following more restrictive condition.

\medskip

\noindent {\sc Assumption ($W2'$)}:
\begin{align*}
& \textnormal{the functions $\varphi$ (a bounded scaling function) and $\psi$ correspond to a MRA of $L^2(\bbR)$,} \\
& \textnormal{and $\textnormal{supp}(\varphi)$ and $\textnormal{supp}(\psi)$ are compact intervals.}
\end{align*}


\noindent All through this section, we assume that ($W1$), ($W2'$) and ($W3$) hold. Given \eqref{e:BH(k)}, we initialize the algorithm with the vector-valued sequence
$$
\bbR^n \ni \widetilde{a}_{0,k} := a_{\varphi}B_H(k), \quad k \in \bbZ, \quad a_{\varphi} := \int_{\bbR}\varphi(t) dt,
$$
also called the approximation coefficients at scale $2^0 = 1$. At coarser scales $2^j$, Mallat's algorithm is characterized by the iterative procedure
$$
\widetilde{a}_{j+1,k} = \sum_{k'\in \bbZ}h_{k'- 2k}\widetilde{a}_{j,k'}, \quad \widetilde{d}_{j+1,k} = \sum_{k'\in \bbZ}g_{k'- 2k}\widetilde{a}_{j,k'}, \quad j \in \bbN, \quad k \in \bbZ,
$$
where the filter sequences $\{h_k\}_{k \in \bbZ}$, $\{g_k\}_{k \in \bbZ}$ are called low- and high-pass MRA filters, respectively. Due to ($W2'$), only a finite number of filter terms is non-zero, which is convenient for computational purposes (see Daubechies \cite{daubechies:1992}, chapter 6).

For the subsequent developments, we will need to strengthen the condition (OFBM1).

\medskip

\noindent {\sc Assumption (OFBM1$'$)}: the eigenvalues $h_k$ of the matrix exponent $H$ satisfy

\begin{equation}\label{e:eigen-assumption_stronger}
    \Re(h_k) \in (0,1)\backslash \{1/2\},\quad k=1,\ldots,n.
\end{equation}

\medskip

\noindent Let $D$ be the shifted exponent matrix \eqref{e:D}. Under \eqref{e:eigen-assumption_stronger}, Theorem 3.2 in Didier and Pipiras \cite{didier:pipiras:2011} yields the time domain representation
\begin{equation}\label{e:OFBM_timedomain}
\{B_H(t)\}_{t \in \bbR} \stackrel{{\mathcal L}}= \Big\{ \int_{\bbR} g_H(t,u)B(du) \Big\}_{t \in \bbR},
\end{equation}
where $B(du)$ is Brownian random measure such that $EB(du)B(du)^* = du$ and
\begin{equation}\label{e:gH(t,u)}
g_H(t,u) = \{(t - u)^{D}_{+}-(-u)^{D}_{+}\} M_+ + \{(t - u)^{D}_{-}-(-u)^{D}_{-}\} M_- \in M(n,\bbR)
\end{equation}
for $M_+,M_- \in M(n,\bbR)$ (in particular, the theorem describes the relation between the matrix parameters $M_+$, $M_-$ and $A$ in \eqref{e:OFBM_harmonizable}). The next result draws upon \eqref{e:OFBM_timedomain} to provide an integral representation for the normalized discretized wavelet coefficients
\begin{equation}\label{e:approxD_2j,k}
\bbR^{n} \ni \widetilde{D}(2^j,k) := 2^{-j/2} \widetilde{d}_{j,k}.
\end{equation}
For notational simplicity, we define $\varphi_i(u) = \varphi(u - i)$, $i \in \bbZ$.
\begin{lemma}
Fix $j \in \bbN$ and $k \in \bbZ$. Under the assumptions (OFBM1$'$) and (OFBM2), let $\widetilde{D}(2^j,k)$ be the discretized wavelet coefficient \eqref{e:approxD_2j,k}. Then,
\begin{equation}\label{e:approxD_2j,k_expression}
\widetilde{D}(2^j,k) \stackrel{d}=  \int_{\bbR} 2^{jD}\Big\{ \int_{\bbR}\psi(t) \sum^{\infty}_{i= - \infty} a_{\varphi} \varphi_{i}(2^{j}t)g_{H}\Big(\frac{i}{2^j},
\frac{s}{2^j}-k \Big) dt \Big\} B(ds).
\end{equation}
\end{lemma}
\begin{proof}
The assumed conditions on $\varphi$ and $\psi$ and the expression \eqref{e:gH(t,u)} imply that the integral on the right-hand side of \eqref{e:approxD_2j,k_expression} is well-defined in the mean-square sense. Therefore, \eqref{e:approxD_2j,k_expression} can be shown by a direct adaptation of the proof of Lemma 6.1 in Stoev et al.\ \cite{stoev:pipiras:taqqu:2002}, which also makes use of Lemma 6.2 in the latter reference. $\Box$\\
\end{proof}

For the proofs of Proposition \ref{p:D-tildeD_is_bounded} and Theorem \ref{t:asympt_log_a(nu)_discrete} below, recall the matrix norm notation \eqref{e:|A|p}. Then, for a random vector ${\mathbf X}$, we can define the norm $\|{\mathbf X}\|_{L^q(P)} = (E\|{\mathbf X}\|^q_{l^q})^{1/q}$, $q \in \bbN$.

The next proposition establishes that the mean-square distance between the discretized and regular wavelet coefficients (expressions \eqref{e:approxD_2j,k} and \eqref{wavelet_transform}, respectively) is bounded by a constant that does not depend on the octave $j$ or the shift $k$.
\begin{proposition}\label{p:D-tildeD_is_bounded}
Under the assumptions (OFBM1$'$) and (OFBM2), let $\widetilde{D}(2^j,k)$ be the approximation coefficient \eqref{e:approxD_2j,k}. Then, there is a constant $C(H,\psi,\varphi) \geq 0$, not depending on $j$ or $k$, such that
$$
\|D(2^j,k)- \widetilde{D}(2^j,k)\|_{L^2(P)} \leq C(H,\psi,\varphi).
$$
\end{proposition}
\begin{proof}
Without loss of generality, we can assume that the interval $[-K,K]$, $K > 0$, contains the support of both functions $\psi$ and $\varphi$. The relations \eqref{e:OFBM_timedomain}, \eqref{e:N_psi} and \eqref{e:integ_EBH(s)BH(t)_psi(s)_psi(t)dsdt} yield the integral representation $D(2^j,k) \stackrel{{\mathcal L}}= \int_{\bbR} 2^{jD} \{\int^{K}_{-K} \psi(t) g_H(t,\frac{s}{2^j}-k)dt \} B(ds)$. So, let
$$
f_{H}(s) = 2^{jD} \int^{K}_{-K} \psi(t) \Big[g_{H}\Big(t, \frac{s}{2^j} -k \Big)
- \sum^{\infty}_{i= -\infty} a_{\varphi}\varphi_{i}(2^jt) g_{H}\Big(\frac{i}{2^j}, \frac{s}{2^j} -k \Big) \Big] dt \in M(n,\bbR).
$$
Then, by expression (\ref{e:approxD_2j,k_expression}), we can write
\begin{equation}\label{e:bound_D-tildeD}
\norm{D(2^j,k)- \widetilde{D}(2^j,k)}_{L^2(P)} = \sqrt{\textnormal{tr}\Big\{\int_{\bbR}f_{H}(s)f_{H}(s)^* ds\Big\}}
= \sqrt{\Big\{\int_{\bbR}\|f_{H}(s)\|^{2}_{l^2} ds\Big\}}.
\end{equation}
However, for almost every $s \in \bbR$,
$$
\|f_{H}(s)\|_{l^2} \leq \norm{2^{jD} \int^{K}_{-K} \psi(t) g_{H}\Big(t, \frac{s}{2^j}-k \Big) \Big(1 - \sum^{\infty}_{i= - \infty}a_{\varphi}\varphi_{i}(2^j t) \Big)  dt }_{l^2}
$$
\begin{equation}\label{e:bound_integrand_D-tildeD}
+ \norm{2^{jD} \int^{K}_{-K} \psi(t) \sum^{\infty}_{i= - \infty}\Big( g_{H}\Big(t, \frac{s}{2^j}-k \Big) -g_{H}\Big(\frac{i}{2^j}, \frac{s}{2^j}-k \Big) \Big)
a_{\varphi}\varphi_{i}(2^j t)dt }_{l^2}.
\end{equation}
By Lemma 6.2 in Stoev et al.\ \cite{stoev:pipiras:taqqu:2002}, the first term on the right-hand side of \eqref{e:bound_integrand_D-tildeD} is zero. Therefore, by a change of variables and Fubini's Theorem, \eqref{e:bound_D-tildeD} is bounded by
$$
C |a_{\varphi}|2^{j/2}\Big( \int_{\bbR}
\norm{2^{jD} \int^{K}_{-K} \psi(t) \sum^{\infty}_{i= - \infty}\Big( g_{H}(t, s ) -g_{H}\Big(\frac{i}{2^j}, s \Big) \Big)
\varphi_{i}(2^j t)dt }^2_{l^2} ds \Big)^{1/2}.
$$
$$
=: \Big(\int_{\bbR} \Big|\Big|\sum^{\infty}_{i = -\infty} F_{t,i}(s) \Big|\Big|^{2}_{l^2} ds \Big)^{1/2} = \norm{ \Big| \Big|
\sum^{\infty}_{i = -\infty} F_{t,i}(s)\Big|\Big|_{l^2}  }_{L^2(\bbR)}
\leq \sum^{\infty}_{i = -\infty} \norm{ \Big| \Big| F_{t,i}(s)\Big|\Big|_{l^2}  }_{L^2(\bbR)}
$$
$$
= C|a_{\varphi}|2^{j/2}\sum^{\infty}_{i = -\infty} \Big( \int_{\bbR} \norm{2^{jD}\int^{K}_{-K}\psi(t) \Big( g_{H}(t,s) - g_{H}\Big(\frac{i}{2^j},s\Big)\Big)
\varphi_{i}(2^j t)dt}^{2}_{l^2}  ds\Big)^{1/2}
$$
\begin{equation}\label{e:bound_D-tildeD_norm_difference_filters}
\leq C |a_{\varphi}|2^{j/2}\sum^{\infty}_{i = -\infty} \int^{K}_{-K}
|\psi(t)| |\varphi_{i}(2^j t)|\Big( \int_{\bbR} \norm{2^{jD}\Big( g_{H}(t,s) - g_{H}\Big(\frac{i}{2^j},s\Big)
\Big) }^{2}_{l^2}  ds\Big)^{1/2} dt,
\end{equation}
where the last inequality is a consequence of Lemma \ref{l:generalized_CS} below. On the other hand, from \eqref{e:gH(t,u)}, $g_{H}(t,s) - g_{H}(\frac{i}{2^j},s)=g_H(t - \frac{i}{2^j}, s-\frac{i}{2^j})$. Thus, by making the change of variables $u = s - i/2^j$,
$$
\int_{\bbR} \textnormal{tr} \Big\{2^{jD} g_{H}\Big(t-\frac{i}{2^j},s-\frac{i}{2^j}\Big)g_{H}\Big(t-\frac{i}{2^j},s-\frac{i}{2^j}\Big)^*2^{jD^*}\Big\}ds
$$
\begin{equation}\label{e:integ_tr(2jDgHgH*2jD*)}
= \int_{\bbR} \textnormal{tr}\Big\{ 2^{jD}g_{H}\Big(t - \frac{i}{2^j},u \Big)g_{H}\Big(t - \frac{i}{2^j},u \Big)^*2^{jD^*}  \Big\} du
= \textnormal{tr}\Big\{ 2^{jD} \Big| t - \frac{i}{2^j}\Big|^{H} \Sigma \Big| t - \frac{i}{2^j}\Big|^{H^*}2^{jD^*}  \Big\},
\end{equation}
where the last step follows from the fact that $EB_H(t)B_H(t)^* = |t|^{H}EB_H(\textnormal{sign}(t))B_H(\textnormal{sign}(t))^*|t|^{H^*} =: |t|^{H}\Sigma |t|^{H^*}$, $t \in \bbR$. By applying \eqref{e:integ_tr(2jDgHgH*2jD*)} and accounting for $\textnormal{supp} (\psi)$ and $\textnormal{supp} (\varphi_i(2^jt))$, the expression \eqref{e:bound_D-tildeD_norm_difference_filters} can be bounded from above by
$$
C |a_{\varphi}|2^{j/2}\norm{\psi}_{\infty}\norm{\varphi}_{\infty} \sum^{K(2^j+1)}_{i= - K(2^j+1)}
\int^{(i+K)/2^j}_{(i-K)/2^j} \Big( \textnormal{tr}\Big\{ 2^{jD} \Big|t - \frac{i}{2^j}\Big|^H \Sigma \Big| t - \frac{i}{2^j}\Big|^{H^*} 2^{jD^*}\Big\} \Big)^{1/2} dt
$$
$$
= C|a_{\varphi}|2^{j/2}\norm{\psi}_{\infty}\norm{\varphi}_{\infty} (2K(2^j+1)+1)2
\int^{K/2^j}_{0} (\textnormal{tr}\{2^{jD} w^H \Sigma w^{H^*}2^{jD^*}\})^{1/2} dw
$$
$$
= C |a_{\varphi}| 2^{j/2} \norm{\psi}_{\infty}\norm{\varphi}_{\infty}
\frac{(2K(2^j+1)+1)2}{2^{3j/2 }}\int^{K}_{0} (\textnormal{tr}\{z^H \Sigma z^{H^*}\})^{1/2} dz
\leq |a_{\varphi}|\norm{\psi}_{\infty}\norm{\varphi}_{\infty}C(H),
$$
where the equalities result from the consecutive changes of variables $w = t - \frac{i}{2^j}$, $z = 2^j w$. This proves the claim. $\Box$\\
\end{proof}

The next definition describes the discretized counterparts of the sample wavelet variance \eqref{e:W(a(v)2^j)} and the estimators \eqref{e:def_estimator} and \eqref{e:angle_estimator}.

\begin{definition}
Let $j, \log_2 a(\nu) \in \bbN$, and let $\widetilde{D}(a(\nu)2^j,k)$ be the discretized wavelet coefficient \eqref{e:approxD_2j,k} at scale $a(\nu)2^j$. We define the associated sample wavelet variance as
$$
\widetilde{W}(a(\nu)2^j) = \frac{1}{K_{a,j}}\sum^{K_{a,j}}_{k=1}\widetilde{D}(a(\nu)2^j,k)\widetilde{D}(a(\nu)2^j,k)^*.
$$
Under \eqref{e:bivariate_OFBM_h1_h2_real}, we also define the operator-renormalized discretized wavelet coefficients
\begin{equation}\label{e:Dnj,k(N)}
\widetilde{D}_\nu (2^j, k) := a(\nu)^{-H}\widetilde{D}(a(\nu)2^j, k), \quad D_\nu (2^j, k) := a(\nu)^{-H}D(a(\nu)2^j , k),
\end{equation}
and the random matrix
\begin{equation}\label{e:Btilde-nu}
\widetilde{B}_{\nu}(2^j)= P^{-1}\frac{1}{K_{a,j}}\sum^{K_{a,j}}_{k=1} \widetilde{D}_\nu (2^j, k)\widetilde{D}_\nu (2^j, k)^*  (P^*)^{-1},\quad P \in GL(2,\bbR)
\end{equation}
(c.f.\ relation \eqref{e:B^(j)}). Let $\widetilde{\lambda}_1(a(\nu)2^j) \leq \widetilde{\lambda}_2(a(\nu)2^j)$ be, respectively, the smallest and largest eigenvalues of $\widetilde{W}(a(\nu)2^j)$. By analogy to Definitions \ref{def:estimator} and \ref{def:theta_estimator}, we define the estimators
\begin{equation}\label{e:h1tilde_h2tilde}
\widetilde{h}_1(a(\nu)2^j) = \frac{\log \widetilde{\lambda}_1(a(\nu)2^j)}{2 \log a(\nu)}, \quad \widetilde{h}_2(a(\nu)2^j) = \frac{\log \widetilde{\lambda}_2(a(\nu)2^j)}{2 \log a(\nu)},
\end{equation}
\begin{equation}\label{e:theta-tilde}
\widetilde{\theta}(a(\nu)2^{j}) = \frac{\widetilde{\lambda}_{1}(a(\nu)2^j) - \widetilde{a}_{j,a(\nu)}}{\widetilde{b}_{j,a(\nu)}}
\end{equation}
for $h_1$, $h_2$ and $\theta$, respectively, where
\begin{equation}\label{e:W-tilde(a(nu)2j)}
\left(\begin{array}{cc}
\widetilde{a}_{j,a(\nu)} & \widetilde{b}_{j,a(\nu)}\\
\widetilde{b}_{j,a(\nu)} & \widetilde{c}_{j,a(\nu)}\\
\end{array}\right) := \widetilde{W}(a(\nu)2^j) = P \textnormal{diag}(a(\nu)^{h_1},a(\nu)^{h_2})\widetilde{B}_\nu(2^j) \textnormal{diag}(a(\nu)^{h_1},a(\nu)^{h_2})P^*.
\end{equation}
\end{definition}

Theorem \ref{t:asympt_log_a(nu)_discrete} below shows that, under mild conditions, estimating the Hurst eigenvalues based on either the discretized or the regular wavelet transform yields the same asymptotic behavior. Before stating and proving it, we will need the next lemma (c.f.\ Lemma \ref{l:vecB^a(2^j)_asympt}).

\begin{lemma}\label{l:vecBtilde(2^j)_asympt}
Under the assumptions (OFBM1$'$), (OFBM2) -- (4) and \eqref{e:a(nu)/J->infty}, let $\widetilde{B}_{\nu}(2^j)$, $\widehat{B}_a(2^j)$ be as in \eqref{e:Btilde-nu} and \eqref{e:B^(j)}. Then,
\begin{equation}\label{e:sqrt(Kaj)(Btilde-Bhat)}
\norm{ \sqrt{K_{a,j}}(\widetilde{B}_{\nu}(2^j) - \widehat{B}_a(2^j)) }_{L^1(P)} \rightarrow 0, \quad \nu \rightarrow \infty.
\end{equation}
\end{lemma}
\begin{proof}
By Minkowsky's inequality, the left-hand side of \eqref{e:sqrt(Kaj)(Btilde-Bhat)} is bounded by
\begin{equation}\label{e:bounding_normalized_distance_squared_wavecoef_vec}
C \frac{1}{\sqrt{K_{a,j}}} \sum^{K_{a,j}}_{k=1}\norm{
\widetilde{D}_\nu(2^j,k)\widetilde{D}_\nu(2^j,k)^* - D_\nu (2^j,k)D_\nu(2^j,k)^*}_{L^1(P)}.
\end{equation}
However, for $k = 1,\hdots,K_{a,j}$, the deviation between the associated non-normalized wavelet variance terms can be recast as
$$
\widetilde{D}(2^ja(\nu),k)\widetilde{D}(2^ja(\nu),k)^* - D(2^ja(\nu),k)D(2^ja(\nu),k)^*
$$
$$
= [\widetilde{D}(2^ja(\nu),k)-D(2^ja(\nu),k)][\widetilde{D}(2^ja(\nu),k)-D(2^ja(\nu),k)]^*
$$
$$
+ [\widetilde{D}(2^ja(\nu),k)-D(2^ja(\nu),k)]D(2^ja(\nu),k)^*
+ D(2^ja(\nu),k)[\widetilde{D}(2^ja(\nu),k)-D(2^ja(\nu),k)]^*.
$$
Therefore, by Proposition \ref{p:D-tildeD_is_bounded}, property ($P3$), and by using the fact that the $l^1$ matrix norm is sub-multiplicative, each term under the summation sign in \eqref{e:bounding_normalized_distance_squared_wavecoef_vec} can be bounded by
$$
\|a(\nu)^{-H}\{[\widetilde{D}(2^ja(\nu),k)-D(2^ja(\nu),k)][\widetilde{D}(2^ja(\nu),k)-D(2^ja(\nu),k)]^*
$$
$$
+ [\widetilde{D}(2^ja(\nu),k)-D(2^ja(\nu),k)]D(2^ja(\nu),k)^*
$$
$$
+ D(2^ja(\nu),k)[\widetilde{D}(2^ja(\nu),k)-D(2^ja(\nu),k)]^*\}a(\nu)^{-H^*}\|_{L^1(P)}
$$
$$
\leq \|a(\nu)^{-H}\|^{2}_{l^1}\|[\widetilde{D}(2^ja(\nu),k)-D(2^ja(\nu),k)][\widetilde{D}(2^ja(\nu),k)-D(2^ja(\nu),k)]^* \|_{L^1(P)}
$$
$$
+ \|a(\nu)^{-H}\|_{l^1} \|\widetilde{D}(2^ja(\nu),k)-D(2^ja(\nu),k)\|_{L^1(P)} \|D(2^j,k)\|_{L^{1}(P)}
$$
$$
+ \|D(2^j,k)\|_{L^1(P)}\|\widetilde{D}(2^ja(\nu),k)-D(2^ja(\nu),k)\|_{L^1(P)}\|a(\nu)^{-H}\|_{l^1}
$$
$$
\leq \|a(\nu)^{-H}\|^{2}_{l^1} \|\widetilde{D}(2^ja(\nu),k)-D(2^ja(\nu),k)\|^2_{L^2(P)}
$$
$$
+ \|a(\nu)^{-H}\|_{l^1}\|\widetilde{D}(2^ja(\nu),k)-D(2^ja(\nu),k)\|_{L^2(P)} \|D(2^j,0)\|_{L^{1}(P)}
$$
$$
+ \|D(2^j,0)\|_{L^1(P)}\|\widetilde{D}(2^ja(\nu),k)-D(2^ja(\nu),k)\|_{L^2(P)}\|a(\nu)^{-H^*}\|_{l^1}
$$
$$
\leq \|a(\nu)^{-H}\|^{2}_{l^1} C^2(H,\psi,\varphi) + \|a(\nu)^{-H}\|_{l^1} \|D(2^j,0)\|_{L^{1}(P)}\hspace{1mm}C(H,\psi,\varphi)
$$
$$
+ \|a(\nu)^{-H^*}\|_{l^1}  \|D(2^j,0)\|_{L^1(P)}\| \hspace{1mm}C(H,\psi,\varphi) \leq C \|a(\nu)^{-H}\|_{l^1}.
$$
Consequently, \eqref{e:bounding_normalized_distance_squared_wavecoef_vec} is bounded by $C \sqrt{K_{a,j}} \|a(\nu)^{-H}\|_{l^1}\rightarrow 0$, $\nu \rightarrow 0$, where the null limit results from \eqref{e:a(nu)/J->infty}. This establishes \eqref{e:sqrt(Kaj)(Btilde-Bhat)}. $\Box$\\
\end{proof}

\begin{theorem}\label{t:asympt_log_a(nu)_discrete}
For $m \in \bbN$, let $j_1 < \hdots < j_m$ be a set of fixed octaves $j$. Let $\widetilde{h}_{1}$, $\widetilde{h}_{2}$, $h^E_{1}$, $h^E_{2}$ be as in \eqref{e:h1tilde_h2tilde} and \eqref{e:lambda^E_1_lambda^E_2}. Under the assumptions (OFBM1$'$), (OFBM2) -- (4) and \eqref{e:a(nu)/J->infty}, as $\nu \rightarrow \infty$,
\begin{itemize}
\item [($i$)]
\begin{equation}\label{e:target_expression_H_diag_discrete}
\left(\begin{array}{c}
2 \log(a(\nu)2^j) \sqrt{K_{a,j}}[\widetilde{h}_1(a(\nu)2^j)-h^E_1(a(\nu)2^j)]\\
2 \log(a(\nu)2^j) \sqrt{K_{a,j}}[\widetilde{h}_2(a(\nu)2^j)-h^E_2(a(\nu)2^j)]\\
\end{array}\right)\stackrel{d}\rightarrow N(0,\Sigma_{h_1,h_2}(j_1,\hdots,j_m)),
\end{equation}
where the matrix $\Sigma_{h_1,h_2}(j_1,\hdots,j_m)$ is given in Theorem \ref{t:asympt_normality_lambda1};

\item [($ii$)] if, in addition, $p_{22} \neq 0$ and the condition \eqref{e:p12_neq_0_and_b12_neq_0} holds, then
\begin{equation}\label{e:thetatilde_weak_limit}
a(\nu)^{h_2-h_1}\sqrt{K_{a,j}} \{ \widetilde{\theta}(a(\nu)2^j) - \theta(a(\nu)2^j) \} \stackrel{d}\rightarrow \hspace{0.5mm} N(0,\sigma^2_{\theta}), \quad \nu \rightarrow \infty.
\end{equation}
where the variance $\sigma^2_{\theta}$ is given in Theorem \ref{t:weak_limit_eigenvector}.
\end{itemize}
\end{theorem}
\begin{proof}
We begin by showing ($i$). Rewrite the left-hand side of \eqref{e:target_expression_H_diag_discrete} as
\begin{equation}\label{e:target_expression_split}
\left(\begin{array}{c}
\sqrt{K_{a,j}}(\log \widetilde{\lambda}_1(a(\nu)2^j)- \log \lambda_1(a(\nu)2^j)) + \sqrt{K_{a,j}}(\log \lambda_1(a(\nu)2^j)- \log \lambda^E_1(a(\nu)2^j))\\
\sqrt{K_{a,j}}(\log \widetilde{\lambda}_2(a(\nu)2^j)- \log \lambda_2(a(\nu)2^j)) + \sqrt{K_{a,j}}(\log \lambda_2(a(\nu)2^j)- \log \lambda^E_2(a(\nu)2^j))
\end{array}\right)_j
\end{equation}
where $j = j_1,\hdots,j_m$. As in the proof of Theorem \ref{t:asympt_normality_lambda1}, without loss of generality we can fix $j$. Theorem \ref{t:asympt_normality_lambda1} implies the asymptotic normality of the second term in the sum \eqref{e:target_expression_split}. We will show that the first term in the sum (\ref{e:target_expression_split}) goes to zero in $L^1(P)$. Note that
\begin{equation}\label{e:loglambdatilde-loglambda}
2 \log(a(\nu)2^j)[\widetilde{h}_{i}(a(\nu)2^j) - \widehat{h}_{i}(a(\nu)2^j)] = \log \widetilde{\lambda}_{i}(a(\nu)2^j)-\log \lambda_{i}(a(\nu)2^j), \quad i = 1,2.
\end{equation}
So, we can repeat the proof of Theorem \ref{t:asympt_normality_lambda1} with $W(a(\nu)2^j)$ in place of $EW(a(\nu)2^j)$ and $\widetilde{W}(a(\nu)2^j)$ (see \eqref{e:W-tilde(a(nu)2j)}) in place of $W(a(\nu)2^j)$. In regard to $i=1$ in \eqref{e:loglambdatilde-loglambda}, we arrive at a relation analogous to \eqref{e:lambda1-lambdaE1=sum}, where
$$
\widetilde{r} := 4\frac{\widetilde{a}\widetilde{c}-\widetilde{b}^2}{(\widetilde{a}+\widetilde{c})^2} \stackrel{P}\sim \frac{1}{a(\nu)^{2(h_2-h_1)}}\frac{4(\det(P))^2\det \widetilde{B}_{\nu}(2^j) }{((p^{2}_{12}+p^{2}_{22})\widetilde{b}_{22})^2}, \quad \nu \rightarrow \infty.
$$
The analogous expression to \eqref{e:lambda1-lambdaE1=sum_second_term_limit} is
$$
\sqrt{K_{a,j}}\Big\{\frac{(\widehat{a}\widehat{c}-\widehat{b}^2)}{(\widehat{a}+\widehat{c})^2} - \frac{(\widetilde{a}\widetilde{c}-\widetilde{b}^2)}{(\widetilde{a}+\widetilde{c})^2} \Big\} = \frac{(\det(P))^2}{a(\nu)^{2(h_2-h_1)}} \Big\{ \frac{O_P(1)}{[O(a(\nu)^{h_1-h_2})+ (p^{2}_{12}+p^{2}_{22})\widehat{b}_{22}]^2}
$$
\begin{equation}\label{e:lambda1tilde-lambda1=sum_second_term_limit}
+ \det \widetilde{B}_{\nu}(2^j) \Big[ \frac{ O_P(1) }{[O_P(a(\nu)^{h_1-h_2})+ (p^{2}_{12}+p^{2}_{22})\widehat{b}_{22}]^2 [O_P(a(\nu)^{h_1-h_2})+ (p^{2}_{12}+p^{2}_{22})\widetilde{b}_{22}]^2}\Big] \Big\} \stackrel{P}\rightarrow 0.
\end{equation}
The numerators of the first and second terms appearing in the sum in \eqref{e:lambda1tilde-lambda1=sum_second_term_limit} are bounded in probability as a consequence of Lemma \ref{l:vecBtilde(2^j)_asympt} and of applying the mean value theorem to the function $f_3(x) = x^2$. The relation analogous to \eqref{e:theo_asympt_normality_lambda1_explicit} is
$$
\sqrt{K_{a,j}} \hspace{0.5mm} \frac{\det \widetilde{B}_{\nu}(2^j) - \det \widehat{B}_{a}(2^j)}{\det \widehat{B}_{a}(2^j)} - \sqrt{K_{a,j}}\hspace{0.5mm} \frac{\widetilde{b}_{22}- \widehat{b}_{22}}{\widehat{b}_{22}} \stackrel{P}\rightarrow 0,
$$
where the limit follows from Lemma \ref{l:vecBtilde(2^j)_asympt}. In regard to $i=2$ in \eqref{e:loglambdatilde-loglambda}, it suffices to repeat the argument in Theorem \ref{t:asympt_normality_lambda1} to arrive at the relation
$$
\sqrt{K_{a,j}}\{ \log(\widetilde{\lambda}_2(a(\nu)2^{j})) - \log(\lambda_2(a(\nu)2^{j})) \} \stackrel{P}\sim \sqrt{K_{a,j}} \Big(\frac{\widetilde{b}_{22}(2^j)- \widehat{b}_{22}(2^j)}{\widehat{b}_{22}(2^j)} \Big) \stackrel{P}\rightarrow 0,
$$
where the limit is again a consequence of Lemma \ref{l:vecBtilde(2^j)_asympt}. Therefore, \eqref{e:target_expression_H_diag_discrete} holds.\\

We now turn to ($ii$). Recast the left-hand side of \eqref{e:thetatilde_weak_limit} as
\begin{equation}\label{e:thetatilde-thetahat+thetahat-thetaE}
a(\nu)^{h_2 - h_1} \sqrt{K_{a,j}}\{ \widetilde{\theta}(a(\nu)2^j) - \widehat{\theta}(a(\nu)2^j) \} +  a(\nu)^{h_2 - h_1} \sqrt{K_{a,j}}\{ \widehat{\theta}(a(\nu)2^j) - \theta(a(\nu)2^j) \}.
\end{equation}
The weak limit of the second term in the sum \eqref{e:thetatilde-thetahat+thetahat-thetaE} is given by Theorem \ref{t:weak_limit_eigenvector}. By analogy with ($i$), we will show that the first term in \eqref{e:thetatilde-thetahat+thetahat-thetaE} converges to zero in probability. Rewrite this term as
\begin{equation}\label{e:1+2-tilde}
a(\nu)^{h_2-h_1}\sqrt{K_{a,j}} \Big\{\frac{\widetilde{\lambda}_1}{\widetilde{b}} - \frac{\lambda_1}{\widehat{b}}\Big\} + a(\nu)^{h_2-h_1}\sqrt{K_{a,j}} \Big\{ \frac{\widehat{a}}{\widehat{b}}- \frac{\widetilde{a}}{\widetilde{b}}\Big\}.
\end{equation}
As a consequence of the proof of \eqref{e:target_expression_H_diag_discrete}, $\log(\widetilde{\lambda}_1(a(\nu)2^j)/\lambda_1(a(\nu)2^j)) \stackrel{P}\rightarrow 0$, namely,
\begin{equation}\label{e:lambda-tilde1/lambda1->1}
\widetilde{\lambda}_1(a(\nu)2^j)/\lambda_1(a(\nu)2^j) \stackrel{P}\rightarrow 1.
\end{equation}
Based on \eqref{e:lambda-tilde1/lambda1->1}, Lemma \ref{l:vecBtilde(2^j)_asympt} and the expression \eqref{e:W-tilde(a(nu)2j)}, we obtain the relations
$$
\frac{\widetilde{\lambda}_1}{\widetilde{b}} \stackrel{P}\sim \frac{1}{[O_P(a(\nu)^{h_1-h_2}) + p_{12}p_{22}\widetilde{b}_{22}]a(\nu)^{2h_2}}\frac{a(\nu)^{2h_1}2 (\det(P))^2 \det \widetilde{B}_{\nu}(2^j)}{[O_{P}(a(\nu)^{h_1-h_2})+(p^{2}_{12}+p^{2}_{22})\widetilde{b}_{22}]},
$$
\begin{equation}\label{e:theta-tilde_auxiliary_limits}
\frac{\widehat{b}}{\widetilde{b}} = O_P(1), \quad \sqrt{K_{a,j}}\Big\{\frac{\widetilde{b} - \widehat{b}}{\widehat{b}}\Big\} \stackrel{P}\rightarrow 0.
\end{equation}
Therefore,
\begin{equation}\label{e:1a-tilde_second_term}
a(\nu)^{h_2-h_1}(-1) \frac{\lambda_1}{\widehat{b}} \sqrt{K_{a,j}}\Big\{ \frac{\widetilde{b} - \widehat{b}}{\widetilde{b}}\Big\} \stackrel{P}\rightarrow 0.
\end{equation}
Moreover, for $\sqrt{K_{a,j}}\{\widetilde{\lambda}_1 - \widehat{\lambda}_1\}$, we can obtain relations analogous to \eqref{e:sqrtKaj(lambda1-lambdaE1)}, \eqref{e:sqrt(K)_(ac-b2/a+c-allhat)-ac-b2/a+c}, \eqref{e:sqrt(K)(sqrt(1-rhat)-1/-rhat)-(sqrt(1-r)-1/-r)} and \eqref{e:1a_first_term} by making use of Lemma \ref{l:vecBtilde(2^j)_asympt}. We conclude that
$$
a(\nu)^{h_2-h_1}\sqrt{K_{a,j}} \Big\{\frac{\widetilde{\lambda}_1}{\widetilde{b}} - \frac{\lambda_1}{\widehat{b}}\Big\} \stackrel{P}\rightarrow 0.
$$
To show that
$$
a(\nu)^{h_2-h_1}\sqrt{K_{a,j}} \Big\{ \frac{\widehat{a}}{\widehat{b}}- \frac{\widetilde{a}}{\widetilde{b}}\Big\} \stackrel{P}\rightarrow 0,
$$
we can follow the steps of the proof of the analogous statement in Theorem \ref{t:weak_limit_eigenvector} to arrive at expressions which, in turn, are analogous to \eqref{e:N(0,sigma2(b12))_limit} and \eqref{e:N(0,sigma2(b22))_limit}. Then, we can make use of Lemma \ref{l:vecBtilde(2^j)_asympt} to find a zero limit in probability, instead of a distribution. Thus, \eqref{e:thetatilde_weak_limit} holds. $\Box$\\
\end{proof}

The next lemma, a multivariate generalized Cauchy-Schwarz inequality, is used in the proof of Proposition \ref{p:D-tildeD_is_bounded}. It can be shown by an adaptation of the univariate argument, which we provide for the reader's convenience.

\begin{lemma}\label{l:generalized_CS}
Let $g: \bbR^2 \rightarrow M(n,\bbR)$,  $(t,s) \mapsto g(t,s)$, be a function such that
\begin{equation}\label{e:generalized_CS_assumptions}
\norm{g(t,\cdot)}_{l^2} , \norm{\int_{\bbR}g(t,\cdot)dt}_{l^2} \in L^2(\bbR).
\end{equation}
Then,
\begin{equation}\label{e:generalized_CS}
\Big( \int_{\bbR} \norm{\int_{\bbR}g(t,s)dt}^2_{l^2} ds \Big)^{1/2} \leq \int_{\bbR} \Big( \int_{\bbR} \norm{g(t,s)}^2_{l^2} ds   \Big)^{1/2} dt.
\end{equation}
\end{lemma}
\begin{proof}
Starting from the square of the left-hand side of \eqref{e:generalized_CS},
$$
\int_{\bbR} \norm{\int_{\bbR}g(t_1,s)dt_1}_{l^2} \norm{\int_{\bbR}g(t_2,s)dt_2}_{l^2} ds
\leq \int_{\bbR} \Big(\norm{\int_{\bbR}g(t_1,s)d t_1}_{l^2} \Big) \Big(\int_{\bbR} \norm{g(t_2,s)}_{l^2} dt_2 \Big) ds
$$
$$
= \int_{\bbR} \Big( \int_{\bbR}  \norm{\int_{\bbR}g(t_1,s)dt_1}_{l^2}  \norm{g(t_2,s)}_{l^2} ds \Big) dt_2
$$
$$
\leq  \int_{\bbR} \Big( \int_{\bbR}  \norm{\int_{\bbR}g(t_1,s_1)dt_1}^2_{l^2}  ds_1\Big)^{1/2} \Big( \int_{\bbR}\norm{g(t_2,s_2)}^2_{l^2} ds_2 \Big)^{1/2} dt_2.
$$
The equality is a consequence of Fubini when applied to the integrators $dt_2 ds$, whereas the second inequality is a consequence of \eqref{e:generalized_CS_assumptions} and the Cauchy-Schwarz inequality. This establishes the claim \eqref{e:generalized_CS}. $\Box$\\
\end{proof}


\bibliography{mlrd}

\small

\bigskip

\noindent \begin{tabular}{lr}
Patrice Abry & \hspace{6cm} Gustavo Didier\\
Physics Lab & Mathematics Department\\
CNRS and \'{E}cole Normale Sup\'{e}rieure de Lyon & Tulane University  \\
46 all\'{e}e d'Italie & 6823 St.\ Charles Avenue\\
F-69364, Lyon cedex 7, France & New Orleans, LA 70118, USA\\
{\it patrice.abry@ens-lyon.fr} & {\it gdidier@tulane.edu}\\
\end{tabular}\\

\smallskip

\end{document}